\documentclass[a4paper,english,10pt]{amsart}
\usepackage[foot]{amsaddr}
\usepackage[utf8]{inputenc}
\usepackage{babel}
\usepackage{units}
\usepackage{amstext,amsmath}
\usepackage{amssymb}
\usepackage[]{graphicx}
\usepackage{amsthm}
\usepackage{color}
\usepackage{dsfont}
\usepackage{adjustbox} %
\usepackage{tikz}
\usepackage{enumitem}
\usepackage{todonotes}
\usepackage{marginnote}
\usepackage{booktabs}
\usepackage{forloop}
\usepackage[left = 3cm, right = 3cm, top= 2cm,bottom=2cm]{geometry}

\usepackage{color}
\definecolor{hanblue}{rgb}{0.27, 0.42, 0.81}
\definecolor{red}{rgb}{1.0, 0.0, 0.0}
\usepackage[colorlinks, citecolor=hanblue,linkcolor=red]{hyperref}
\usepackage{cleveref}
\usepackage[linesnumbered, ruled, vlined]{algorithm2e}
\usepackage{soul}
\usepackage[font=small,margin=0.5cm]{caption}
\usepackage{subcaption}

\makeatletter
\numberwithin{equation}{section}
\newtheorem{thm}{Theorem}[section]
\theoremstyle{plain}
\theoremstyle{remark}
\newtheorem{rem}[thm]{Remark}
\theoremstyle{plain}
\newtheorem{lem}[thm]{Lemma}
\theoremstyle{plain}
\newtheorem{prop}[thm]{Proposition}
\theoremstyle{plain}

\theoremstyle{definition}
\newtheorem{definition}[thm]{Definition}
\newtheorem{assumption}[thm]{Assumption}

\theoremstyle{plain}
	\newcommand{\bes}{\begin{equation*}}
	\newcommand{\ees}{\end{equation*}}
	\newcommand{\be}{\begin{equation}}
	\newcommand{\ee}{\end{equation}}
	\newcommand{\bi}{\begin{itemize}}
	\newcommand{\ei}{\end{itemize}}
		\def\bas#1\eas{\begin{align*}#1\end{align*}}
		\def\ba#1\ea{\begin{align}#1\end{align}}
	\newcommand{\baed}{\begin{aligned}}
	\newcommand{\eaed}{\end{aligned}}
	\newcommand{\R}{\mathbb{R}}
	\newcommand{\N}{\mathbb{N}}
	\newcommand{\C}{\mathbb{C}}

		\newcommand{\Om}{\Omega}

	\newcommand{\M}{\mathcal{M}}
	\newcommand{\calM}{\mathcal{M}}

	\newcommand{\calD}{\mathcal{D}}
	\newcommand{\F}{\mathcal{F}}

	\DeclareMathOperator*{\argmin}{arg\,min}
	\DeclareMathOperator*{\argmax}{arg\,max}
	
  \let\div\relax
  \DeclareMathOperator*{\div}{div}

	\newcommand{\norm}[1]{\left\lVert#1\right\rVert}
	\newcommand{\ps}[2]{ \langle #1, #2  \rangle }
	\newcommand{\p}{\partial}
	\newcommand{\de}{\partial}

	\newcommand{\olom}{\overline{\Om}}
        \newcommand{\WT}{X}
        
\newcommand{\e}{\varepsilon}
\newcommand{\f}{\varphi}
\newcommand{\Ltwo}{L^2_H}

\DeclareMathOperator{\supp}{supp} %
\newcommand{\curves}{C_{\rm w}}
\newcommand{\pcurves}{C_{\rm w}^+}
\newcommand{\weak}{\rightharpoonup}   %
\newcommand{\weakstar}{\stackrel{*}{\rightharpoonup}}  %

\DeclareMathOperator{\ext}{Ext} %
\DeclareMathOperator{\gr}{graph}

\newcommand{\AC}{{\rm AC}}
\newcommand{\ACe}{{\rm {AC}}_{\infty}}
\DeclareRobustCommand{\rchi}{{\mathpalette\irchi\relax}}
\newcommand{\irchi}[2]{\raisebox{\depth}{$#1\chi$}}   
        
\newcommand{\points}{\mathcal{C}_{\alpha,\beta}}

\title[A generalized conditional gradient method for dynamic inverse problems]{A generalized conditional gradient method for dynamic inverse problems with optimal transport regularization}

\author[K. Bredies]{Kristian Bredies}
\author[M. Carioni]{Marcello Carioni} 
\author[S. Fanzon]{Silvio Fanzon}
\author[F. Romero]{Francisco Romero}
 
 \address[Kristian Bredies, Silvio Fanzon, Francisco Romero]{University of Graz, Institute of Mathematics and Scientific Computing, Heinrichstra\ss e 36, 8010 Graz, Austria}

 \address[Marcello Carioni]{University of Cambridge, Department of Applied Mathematics and Theoretical Physics, Wilberforce Road, Cambridge
CB3 0WA, UK}

\email[Kristian Bredies]{Kristian.Bredies@uni-graz.at}
\email[Marcello Carioni]{mc2250@maths.cam.ac.uk}
\email[Silvio Fanzon]{Silvio.Fanzon@uni-graz.at}
\email[Francisco Romero]{Francisco.Romero-Hinrichsen@uni-graz.at}

\begin{document}

\begin{abstract}
\small{We develop a dynamic generalized conditional gradient method (DGCG) for dynamic inverse problems with optimal transport regularization. We consider the framework introduced in (Bredies and Fanzon, ESAIM: M2AN, 54:2351--2382, 2020), %
where the objective functional is comprised of a fidelity term, penalizing the pointwise in time discrepancy between the observation and the unknown in time-varying Hilbert spaces, and a regularizer keeping track of the dynamics, given by the Benamou-Brenier energy constrained via the homogeneous continuity equation. Employing the characterization of the extremal points of the Benamou-Brenier energy (Bredies et al., Bull. Lond. Math. Soc., 53(5):1436--1452, 2021) we define the \emph{atoms} of the problem as measures concentrated on absolutely continuous curves in the domain. %
We propose a dynamic generalization of a conditional gradient method that consists of iteratively adding suitably chosen \emph{atoms} to the current sparse iterate, and subsequently optimizing the coefficients in the resulting linear combination. We prove that the method converges with a sublinear rate to a minimizer of the objective functional. 
Additionally, we propose heuristic strategies and acceleration steps that allow to implement the algorithm efficiently. Finally, we provide numerical examples that demonstrate the effectiveness of our algorithm and model in reconstructing heavily undersampled dynamic data, together with the presence of noise.


}
\end{abstract}

\maketitle

\section{Introduction}

The aim of this paper is to develop a dynamic generalized condition gradient method (DGCG) to numerically compute solutions of ill-posed dynamic inverse problems regularized with optimal transport energies. The code is openly available on GitHub\footnote{\url{https://github.com/panchoop/DGCG_algorithm/}}. 

Lately, several approaches have been proposed to tackle dynamic inverse problems \cite{hk, sl, Schmitt_2002_2, shb,  ws}, all of which take advantage of redundancies in the data,  allowing to stabilize reconstructions, both in the presence of noise or undersampling.
 A common challenge faced in such time-dependent approaches is understanding how to properly connect, or relate, the time-neighbouring datapoints, 
in a way that the reconstructed object follows a presumed dynamic. In this paper we address such issue by means of dynamic optimal transport. 
A wide range of applications can benefit from motion-aware approaches.
In particular, the employment of dynamic reconstruction methods represents
one of the latest key mathematical advances in medical imaging. %
For instance,
magnetic resonance imaging (MRI) \cite{ lhdj, ocs,shsbs} and computed tomography (CT) \cite{  Bonnet2003,Burger_2017,Ding_2017} methods   allow dynamic modalities in which the time-dependent data
is further undersampled to reach high temporal sampling rates;
these are required to resolve organ motion, such as the beating heart or the breathing lung. A more accurate reconstruction of the image, and of the underlying dynamics, would yield valuable diagnostic information.

\subsection{Setting and existing approaches}

Recently, it has been proposed to regularize dynamic inverse problems using dynamic optimal transport energies both
in a balanced and unbalanced context \cite{bf, schmitzerwirth2020,schmitzerwirth2019}, with the goal of efficiently reconstructing time-dependent Radon measures. Such regularization choice is natural: optimal transport energies incorporate information about time correlations present in the data, and are thus favoring a more stable reconstruction. %
Optimal transport theory was originally developed to find the most efficient way to
move mass from a probability measure $\rho_0$ to another one $\rho_1$, with respect to a given cost \cite{kantorovich, sant}. More recently Benamou
and Brenier \cite{bb} showed that the optimal transport map can be computed by solving
\begin{equation} \label{intro:bb}
\min_{(\rho_t,v_t)} \frac{1}{2}\int_0^1\int_{\olom} |v_t(x)|^2 \, d\rho_t(x)\,dt \,\,\,\, \text{ s.t. } \,\,\,\, \partial_t \rho_t + \div(v_t \rho_t ) = 0\,,
\end{equation}
where $t \mapsto \rho_t$ is a curve of
probability measures on the closure of a bounded domain $\Om  \subset \R^d$, $v_t$ is a time-dependent vector field advecting the mass, and the continuity equation is intended in the sense of distributions 
with initial data $\rho_0$ and final data $\rho_1$.
Notably, the quantity at \eqref{intro:bb}, named Benamou-Brenier energy,
admits an equivalent convex reformulation. Specifically, consider the space of bounded Borel measures $\M := \mathcal{M}(X)  \times
\mathcal{M}(X; \R^d)$, $X:=(0,1)\times \olom$, and define the convex energy $B: \M  \rightarrow
[0,+\infty]$ by setting
\begin{equation*}
B(\rho,m) := \frac{1}{2} \int_0^1 \int_{\olom} \left|\frac{dm}{d\rho}\,(t,x)\right|^2\, d\rho(t,x)\,,
\end{equation*}
if $\rho \geq 0$ and  $m\ll \rho$, and $B:=+\infty$ otherwise. Then \eqref{intro:bb} is equivalent to minimizing $B$ under the linear constraint $\partial_t \rho + \div m = 0$. Such convex reformulation %
can be employed as a regularizer for dynamic inverse problems where, instead of fixing initial and
final data, a fidelity term is added to measure the discrepancy between the
unknown and the observation at each time instant, as proposed in \cite{bf}. There the authors consider the dynamic inverse problem of finding a curve of measures $t \mapsto \rho_t$, with $\rho_t \in \M(\olom)$, such that 
\begin{equation} \label{intro:dip}
K_t^*\rho_t = f_t \,\,\, \text{ for a.e. }\,\, t \in (0,1)\,,
\end{equation}
where $f_t \in H_t$ is some given data, $\{H_t\}$ is a family of Hilbert spaces and $K_t^* : \M(\olom) \rightarrow H_t$ are linear continuous observation operators. The problem at \eqref{intro:dip} is then regularized via the minimization problem
\begin{equation}\label{intro:inverseproblem}
\min_{(\rho,m) \in \M} \frac{1}{2}\int_0^1 \norm{K^*_t \rho_t - f_t}^2_{H_t}dt  + \beta B(\rho,m) +
      \alpha \norm{\rho}_{\mathcal{M}(X)} \,\,\,\, \text{ s.t. } \,\,\,\, \partial_t \rho + \div m = 0\,,
\end{equation}
where $\|\rho\|_{\mathcal{M}(X)}$ denotes the total variation norm of the measure $\rho:=dt \otimes \rho_t$, and $\alpha, \beta>0$ are regularization parameters. Notice that any curve $t \mapsto \rho_t$ having finite Benamou-Brenier energy and satisfying the continuity equation constraint must have constant mass %
(see Lemma \ref{lem:prop cont}). As a consequence, the regularization  \eqref{intro:inverseproblem} is especially suited to reconstruct motions where preservation of mass is expected.
We point our that such formulation is remarkably flexible, as the measurements spaces and measurement operators are allowed to be very general. In this way one could model, for example, undersampled acquisition strategies in medical imaging, particularly MRI \cite{bf}.

The aim of this paper is to design a numerical algorithm to solve \eqref{intro:inverseproblem}.
The main difficulties arise due to the non-reflexivity of measure spaces. Even in the static case, solving the classical LASSO problem \cite{lasso} in the space of bounded Borel measures (known as BLASSO \cite{blasso}), i.e.,
 \begin{equation}\label{intro:BLASSO}
 \inf_{\rho \in \mathcal{M}(\Om)}\frac{1}{2} \|K\rho - f\|^2_Y + \alpha \|\rho\|_{\M(\Omega)}\,,
 \end{equation}
 for a Hilbert space $Y$ and a linear continuous operator $K : \mathcal{M}(\Om) \rightarrow Y$,
 has proven to be challenging. Usual strategies to tackle \eqref{intro:BLASSO} numerically often rely on the discretization of the domain \cite{combetteswajs, ddd, tsengblock}; however grid-based methods are known to be affected by theoretical and practical flaws such as high computational costs and the presence of mesh-dependent artifacts in the reconstruction.
 The mentioned drawbacks have motivated algorithms that
 do not rely on domain discretization, but optimize directly on the space of
 measures. One class of such algorithms, first introduced in \cite{K-Pikkarainen}  and subsequently developed
 in different directions \cite{boydgcg, denoyelle2019sliding,  flinth2020, pieper2020},
 are named generalized conditional gradient methods (GCG) or Frank-Wolfe type algorithms. They can be regarded as the infinite dimensional
 generalization of the classical Frank-Wolfe optimization algorithm
 \cite{frankwolfe} and of GCG in Banach spaces \cite{bachgcg,bredieslorenz2008,Bredies2009, clarkson2010,   dunn1979,jaggi2013}. The basic idea behind such algorithms consists in 
 exploiting the structure of sparse solutions to  \eqref{intro:BLASSO}, which are given 
 by finite linear combinations of Dirac deltas supported on $\Om$.
In this case Dirac deltas represent the extremal points of the unit ball of the Radon norm regularizer.
With this
knowledge at hand, the GCG method iteratively
minimizes a linearized version of  \eqref{intro:BLASSO}; such minimum can be found in the set of extremal points. The iterate is then constructed by adding delta peaks at each iteration, and by subsequently optimizing the coefficients of the linear
combination. GCG methods have proven to be
successful at solving \eqref{intro:BLASSO}, and have been adapted to related
problems in the context of, e.g., super-resolution \cite{alberti2019dynamic, neitzelgcg, pieper2020,
schmitzerwirth2019}.

\subsection{Outline of the main contributions}

Inspired by GCG methods, the goal of this paper is to develop a dynamic
generalized conditional gradient method (DGCG) aimed at solving the dynamic
minimization problem at \eqref{intro:inverseproblem}. Similarly to the
classical GCG approaches, our DGCG algorithm is based on the structure of sparse solutions to \eqref{intro:inverseproblem}, and it is Lagrangian in essence, since it does not require a discretization of the space domain.  Lagrangian approaches have been proven useful for many different dynamic applications \cite{saumier2015optimal}, often outperforming Eulerian approaches, where the discretization in space is
        necessary. Indeed, since Eulerian approaches are based on the
        optimization of challenging discrete assignment problems in space,
        Lagrangian approaches allow to lower the computational costs and are
        more suitable to reconstruct coalescence phenomena in the dynamics.  Motivated by similar considerations,
our approach aims to reduce the reconstruction artifacts and lower the
computational cost when compared to grid-based methods designed to solve similar inverse problems to \eqref{intro:inverseproblem}  (see \cite{schmitzerwirth2020}).

The fundamentals of our approach rest on recent results concerning sparsity for variational inverse problems:
it has been empirically observed that the presence of a regularizer promotes the existence of sparse solutions, that is, minimizers that can be represented as a finite linear combination of simpler atoms. This effect is evident in reconstruction problems \cite{duval2019epigraphical, flinth2018, unser2,  Unser2019NativeBS,unsersplines}, as well as in variational problems in other applications, such as materials science \cite{FPP2,Fanzon:2020aa,LL,ShRe}. Existence of sparse solutions has been recently proven for a class of general functionals comprised of a fidelity term, mapping to a finite dimensional space, and a regularizer: in this case atoms correspond to the extremal points of the unit ball of the regularizer \cite{chambolle,bredies2020sparsity}. %
In the context of \eqref{intro:inverseproblem}, the extremal points of the Benamou-Brenier energy have been recently characterized in \cite{bcfr}; this provides an operative notion of atoms that  will be used throughout the paper. More precisely, for every absolutely continuous curve $\gamma: [0,1] \rightarrow \olom$, we name as \textit{atom} of the Benamou-Brenier energy the respective pair of measures $\mu_\gamma := (\rho_\gamma, m_\gamma) \in \M$ defined by 
    \begin{equation}\label{intro:atom}
     \rho_\gamma:=a_\gamma \, dt \otimes \delta_{\gamma(t)} \,, \,\,\,\, m_\gamma:=\dot \gamma(t) \rho_\gamma \,,\,\,\,\,
a_\gamma:=\left( \frac{\beta}{2} \int_0^1 |\dot \gamma (t)|^2 \, dt + \alpha \right)^{-1}\,.
    \end{equation}
The notion of atom described above can be regarded as the dynamic counterpart of the Dirac deltas for the Radon norm regularizer. Curves of measures of the form \eqref{intro:atom} constitute the building blocks used in our DGCG method to generate at iteration step $n$ the sparse iterate $\mu^n = (\rho^n , m^n)$
\begin{equation} \label{intro:iterates}
\mu^n = \sum_{j=1}^{N_n} c_j \mu_{\gamma_{j}}
\end{equation}
converging to a solution of \eqref{intro:inverseproblem}, where $c_j >0$. 

 The basic DGCG method proposed is comprised of two steps. The first one, called \textit{insertion step}, operates as follows. Given a sparse iterate $\mu^n$ of the form \eqref{intro:iterates}, we obtain a descent direction for the energy at \eqref{intro:inverseproblem}, by minimizing a version of \eqref{intro:inverseproblem} around $\mu^n$, in which the fidelity term is linearized. 
We show that, in order to find such descent direction, it is sufficient to solve
\begin{equation} \label{intro:extinsertion}
  \min_{(\rho,m) \in \ext(C_{\alpha,\beta})}  - \int_0^1 \langle \rho_t, w_t^n\rangle_{H_t}\, dt\,,
\end{equation}
where $w^n_t  := -K_t(K_t^* \rho^n_t -  f_t ) \in C(\olom)$ is the dual variable of the problem at the iterate $\mu^n$, and $\ext(C_{\alpha,\beta})$ denotes the set of extremal points of the unit ball of the regularizer in \eqref{intro:inverseproblem}, namely the set $C_{\alpha,\beta} := \{(\rho,m) \in \M : \partial_t \rho + \div m = 0,\ \  \beta B(\rho,m) +
      \alpha \norm{\rho}_{\mathcal{M}(X)}\leq 1\}$. Formula \eqref{intro:extinsertion} clarifies the connection between atoms and extremal points of $C_{\alpha,\beta}$, showing the fundamental role that the latter play in sparse optimization and GCG methods. In view of the characterization Theorem~\ref{thm:extremal}, proven in \cite{bcfr},  the minimization problem \eqref{intro:extinsertion} can be equivalently written in terms of atoms 
\begin{equation}
  \label{intro:hard_min}
\min_{\gamma \in \AC^2} \, \left\{ 0, \, 
 -a_{\gamma} \int_0^1 w_t^n(\gamma(t))\,dt \right\}\,,
\end{equation}
where $\AC^2$ denotes the set of absolutely continuous curves with values in $\olom$ and weak derivative in $L^2(\olom;\R^d)$. The insertion step then consists in finding a curve  $\gamma^*$ solving \eqref{intro:hard_min}, and considering the respective atom $\mu_{\gamma^*}$. Afterwards, naming $\gamma_{N_n+1}:=\gamma^*$, the \textit{coefficients optimization step} proceeds at optimizing the conic combination 
  $\mu^n + c_{N_{n}+1} \mu_{\gamma_{N_n+1}}$ with respect to \eqref{intro:inverseproblem}, among all non-negative coefficients $c_j$. Denoting by $c_1^*,\ldots, c_{N_n+1}^*$ a solution to such problem, the new iterate is  defined by $\mu^{n+1}:=\sum_j c_j^*\mu_{\gamma_j}$. %
The two steps of inserting a new atom in the linear combination and optimizing the coefficients are the building blocks of our \textit{core algorithm}, summarized in Algorithm~\ref{alg:core}. In Theorem \ref{thm:convergence} we prove that such algorithm has a sublinear convergence rate, similarly to the GCG method for the BLASSO problem \cite{K-Pikkarainen}, and the produced iterates $\mu^n$ converge in the weak* sense of measures to a solution of \eqref{intro:inverseproblem}. The core algorithm and its analysis are the subject of Section~\ref{sec:algorithm}.

From the computational point of view,
we observe that the coefficients optimization step can be solved efficiently, as it is equivalent to a finite dimensional quadratic program. Concerning the insertion step, however, even if the complexity of searching for a descent direction for \eqref{intro:inverseproblem} is reduced by only minimizing in the set of atoms, \eqref{intro:hard_min} remains a challenging non-linear and non-local problem. For this reason, we shift our attention to computing stationary points for \eqref{intro:hard_min}, relying on gradient descent strategies. Specifically, we prove that, under additional assumptions on $H_t$ and $K_t^*$, problem \eqref{intro:hard_min} can be cast in the Hilbert space $H^1([0,1];\R^d)$, and that the gradient descent algorithm, with appropriate stepsize, outputs stationary points to  \eqref{intro:hard_min} (see Theorem~\ref{thm:gradient_descent}). With this theoretical result at hand, in Section \ref{sec:insstepheu} we formulate a solution strategy for \eqref{intro:hard_min} based on multistart gradient descent methods, whose initializations are chosen according to heuristic principles. More precisely, the initial curves are chosen randomly in the regions where the dual variable $w_t^n$ has larger value, and new starting curves are produced combining pieces of stationary curves for \eqref{intro:hard_min} by means of a procedure named \emph{crossover}. %

We complement the core algorithm with acceleration strategies. First, we add multiple atoms in the insertion step (\emph{multiple insertion step}). Such new atoms can be easily obtained as a byproduct of the multistart gradient descent in the insertion step. Moreover, after optimizing the coefficients in the linear combination, we perform an additional gradient descent step with respect to \eqref{intro:inverseproblem}, varying the curves in the iterate $\mu^{n+1}$, while keeping the weights fixed.   
Such procedure, named \emph{sliding step}, will then be alternated with the coefficients optimization step for a fixed number of iterations, before searching for a new atom in the insertion step. These additional steps are described in Section~\ref{sec:insstepheu}. We mention that similar strategies were already employed for the BLASSO problem \cite{K-Pikkarainen, pieper}. They are then included in the basic core algorithm to obtain the complete DGCG method in Algorithm \ref{alg:full}.

In Section \ref{sec:numerics} we provide numerical examples. As observation operators $K^*_t$, we use time-dependent undersampled Fourier measurements, popular in imaging and medical imaging \cite{brediesbook, epsteinintro}, 
as well as in compressed sensing and super-resolution \cite{ 
alberti2019dynamic, crt,candes2014towards}. Such examples show the effectiveness of our DGCG method in reconstructing spatially sparse data, in presence of simultaneously strong noise and severe temporal undersampling. 
Indeed, satisfactory results are obtained for ground-truths with $20\%$ and $60\%$ of added Gaussian noise, and heavy 
temporal undersampling in the sense that, at each time $t$, the observation operator $K^*_t$  is not able to distinguish sources along lines.    
With such ill-posed measurements static reconstruction methods would not be able to accurately recover any ground-truth. In contrast, the time regularization chosen in \eqref{intro:inverseproblem} allows to resolve the dynamics by correlating the information of neighbouring data-points. 
As shown by the experiments presented, our DGCG algorithm produces accurate reconstructions of the ground-truth. %
In case of $20\%$ and $60\%$ of added noise we note a surge of low intensity artifacts;  nonetheless, the obtained reconstruction is close, in the sense of the measures, to the original ground-truth. 
Moreover, in all of the tried out examples a linear convergence rate has been observed; this shows that the algorithm is, in practice, faster than the theoretical guarantees (Theorem \ref{thm:convergence}), and a linear rate has to be expected in most of the cases.   

We conclude the paper with Section \ref{sec:perspectives}, in which we discuss future perspectives and open questions, such as the possibility of improving the theoretical convergence rate for Algorithm \ref{alg:core}, and alternative modelling choices. 
\subsection{Organization of the paper}
  The paper is organized as follows. In Section \ref{sec:preliminaries} we summarize all the relevant notations and preliminary results regarding the Benamou-Brenier energy that are needed in the paper. In particular, we recall the characterization of the extremal points of the unit ball of the Benamou-Brenier energy obtained in \cite{bcfr}. 
In Section \ref{sec:OT_regularization} we introduce the dynamic inverse problem under consideration and its regularization \eqref{intro:inverseproblem}, following the approach of \cite{bf}. We further establish basic theory needed to setup the DGCG method. In Section~\ref{sec:algorithm} we provide the definition of atoms and we give a high-level description of the DGCG method we propose in this paper, see Algorithm \ref{alg:core}, proving its sublinear convergence.  
In Section~\ref{sec:numerical_implementation} we describe the strategy employed to solve the insertion step problem \eqref{intro:hard_min}, based on a multistart gradient descent method. Moreover we outline the mentioned acceleration steps. Incorporating these procedures in the core algorithm, we obtain the complete DGCG method, see Algorithm \ref{alg:full}.
In Section \ref{sec:numerics} we show numerical results supporting the effectiveness of Algorithm \ref{alg:full}. Finally, we present the reader some open questions in Section \ref{sec:perspectives}.

\section{Preliminaries and notation}\label{sec:preliminaries}
In this section we introduce the mathematical concepts and results we need to formulate our minimization problem and consequent  algorithms. Throughout the paper $\Omega \subset \R^d$ denotes an open bounded domain with $d\in \N, d \geq 1$. We define the time-space cylinder $\WT := (0,1) \times \overline \Omega$. Following \cite{afp}, given a metric space $Y$ we denote by $\M(Y)$, $\M(Y ; \R^d)$,  $\M^+(Y)$, the spaces of bounded Borel measures, bounded vector Borel measures, and positive measures, respectively. For a scalar measure $\rho$ we denote by $\norm{\rho}_{\M(Y)}$ its total variation. In addition, we employ the notations $\R_+$ and $\R_{++}$ to refer to the non-negative and positive real numbers, respectively.

\subsection{Time dependent measures}
We say that $\{\rho_t\}_{t \in [0,1]}$ is a Borel family of measures in $\M(\olom)$ if $\rho_t \in \M(\olom)$ for every $t \in [0,1]$ and the map
$t \mapsto \int_{\olom} \varphi(x)\, d\rho_t(x)$
is Borel measurable for every function $\varphi \in  C(\olom)$.
Given a measure $\rho \in \M(X)$ we say that $\rho$ disintegrates with respect to time if there exists a Borel family of measures $\{\rho_t\}_{t \in [0,1]}$ in %
$\M(\olom)$ such that
	\[
	\int_X \f(t,x) \, d\rho(t,x) = \int_0^1 \int_{\olom} \f (t,x) \, d\rho_t(x) \, dt \quad \text{ for all } \quad  \f \in L^1_{\rho}(X) \,.
	\]
We denote such disintegration with the symbol $\rho =dt \otimes  \rho_t$. Further, we say that a curve of measures $t \in [0,1] \mapsto \rho_t \in \M(\olom)$ is narrowly continuous if, for all  $\f \in C(\olom)$, the map $t \mapsto \int_{\olom} \f(x) \, d\rho_t(x)$
is continuous. The family of narrowly continuous curves will be denoted by $\curves$. We denote by $\pcurves$ the family of narrowly continuous curves with values in $\M^+(\olom)$. 

\subsection{Optimal transport regularizer} 
Introduce the space
\begin{equation*}
  \M := \calM(\WT) \times \calM( \WT ; \R^d )\,.
\end{equation*}
We denote elements of $\M$ by $\mu=(\rho,m)$ with $\rho \in \M(X)$, $m \in \M(X;\R^d)$, and by $0$ the pair $(0,0)$.
Define the set of pairs in $\calM$ satisfying the continuity equation as 
\begin{equation*}
  \calD := \{\mu \in \calM \ : \ \p_t \rho + \div m = 0 \ \ \text{in} \ \ 
  \WT \}\,,
\end{equation*}
where the solutions of the continuity equation are intended in a distributional 
sense, that is,
\begin{equation} \label{cont weak}
\int_{X}  \de_t \f \, d\rho + 
\int_{X}  \nabla \f \cdot  dm   = 0   \quad \text{for all} \quad \f \in C^{\infty}_c ( X ) \,.
\end{equation}

The above weak formulation includes no flux boundary conditions for the momentum $m$ on $\de \Om$, and no initial and final data for $\rho$.  Notice that, by standard approximation arguments, it is equivalent to test \eqref{cont weak} against maps in $C^1_c ( X )$ (see \cite[Remark 8.1.1]{ags}).

We now introduce the Benamou-Brenier energy, as originally done in \cite{bb}. To this end, define the convex, one-homogeneous and lower semicontinuous map $\Psi \colon \R \times \R^d \to [0,\infty]$ as
\begin{equation} \label{intro Psi}
\Psi(t,x):=
\begin{cases}
 \frac{|x|^2}{2t} & \text{ if } t>0\,, \\
0    & \text{ if } t=|x|=0 \,, \\
+\infty & \text{ otherwise}\,. 
\end{cases}
\end{equation}
The Benamou-Brenier energy $B \colon \M \to [0,\infty]$ is defined by  
\begin{equation} \label{intro convex}
B (\mu) := \int_X \Psi \left( \frac{d\rho}{d\lambda}, \frac{dm}{d\lambda}\right) \, d \lambda \,,
\end{equation}
where $\lambda \in \M^+(X)$ is any measure satisfying $\rho,m \ll \lambda$. Note that \eqref{intro convex} does not depend on the choice of $\lambda$, as $\Psi$ is one-homogeneous. Following \cite{bf}, we introduce a coercive version of $B$: for fixed parameters $\alpha,\beta >0$ define the functional $J_{\alpha, \beta} \colon \M \to [0,\infty]$ as
\begin{equation} \label{prel:reg}
  J_{\alpha, \beta}(\mu) := \begin{cases} 
      \beta B(\mu) +
      \alpha \norm{\rho}_{\M(X)} & \,\, \text{ if }(\rho,m) \in \calD\,, \\ 
      + \infty\qquad  & \,\,   \text{ otherwise}\,.
    \end{cases} 
\end{equation}
As recently shown \cite{bf}, $J_{\alpha,\beta}$ can be employed as a regularizer for dynamic inverse problems in spaces of measures.

\subsection{\texorpdfstring{Extremal points of $J_{\alpha,\beta}$}{Extremal points of J\_\{alpha,beta\}} \label{sec:extremal}}

Define the convex unit ball
\[
C_{\alpha,\beta} := \left\{\mu \in \M \, \colon \, J_{\alpha,\beta}(\mu)  \leq 1 \right\} \,,
\]
and the set of measures concentrated on $\AC^2$ curves in $\olom$%
\begin{equation} \label{ext_points}
\points := \left\{ \mu_\gamma   \in \M \, \colon \, \gamma \in \AC^2([0,1];\R^d),\,\, 
\gamma([0,1]) \subset \olom \right\}  \,,
\end{equation}
where we denote by $\mu_\gamma$ the pair $(\rho_\gamma,m_\gamma)$ with
\begin{equation} \label{ext_meas}
\rho_\gamma:=a_\gamma \, dt \otimes \delta_{\gamma(t)} \,, \,\,\,\, m_\gamma:=\dot \gamma(t) \rho_\gamma \,,\,\,\,\,
a_\gamma:=\left( \frac{\beta}{2} \int_0^1 |\dot \gamma (t)|^2 \, dt + \alpha \right)^{-1}\,.
\end{equation}
Here $\AC^2([0,1];\R^d)$ denotes the space of curves having metric derivative in $L^2((0,1);\R^d)$. We can identify $\AC^2([0,1];\R^d)$ with the Sobolev space $H^1((0,1);\R^d)$ (see \cite[Remark 1.1.3]{ags}). For brevity, we will denote by $\AC^2:=\AC^2([0,1];\olom)$ the set of curves $\gamma$ belonging to $\AC^2([0,1];\R^d)$ such that $\gamma([0,1]) \subset \olom$.   %
For the extremal points of $C_{\alpha,\beta}$ we have the following characterization result, originally proven in \cite[Theorem 6]{bcfr}. 
\begin{thm} \label{thm:extremal}
Let $\alpha,\beta>0$ be fixed. Then it holds $\,\ext(C_{\alpha,\beta})=\{0\} \cup \points$.
\end{thm}

We now show that $J_{\alpha,\beta}$ is linear on non-negative combinations of points in $\points$. Such property will be crucial for several computations in this paper, and the proof is postponed to Section \ref{app:proof:lemma}  %

\begin{lem} \label{lem:additivity}
Let $N \in \N, N \geq 1$, $c_j \in \R$ with $c_j >0$, and $\gamma_j \in \AC^2$ for $j=1,\ldots,N$. Let $\mu_{\gamma_j}=(\rho_{\gamma_j},m_{\gamma_j}) \in \points$ be defined according to \eqref{ext_meas}. Then $J_{\alpha,\beta}(\mu_j)=1$ and
\[%
J_{\alpha,\beta} \left( \sum_{j=1}^N c_j \mu_{\gamma_j} \right)=
\sum_{j=1}^N c_j\,.
\]%
\end{lem}

\section{The dynamic inverse problem and conditional gradient method} \label{sec:OT_regularization}
In this section we introduce the dynamic inverse problem we aim at solving, following the approach of \cite{bf}. Moreover we set up the functional analytic framework necessary to state the numerical algorithm presented in Section \ref{sec:algorithm}. %
Recall that $\Om \subset \R^d$ is an open bounded domain, $d \in \N$, $d \geq 1$. Let $\{H_t\}_{t \in [0,1]}$ be a family of real Hilbert spaces, $K_t^* \colon \M(\olom) \to H_t$ a family of linear continuous forward operators parametrized by $t\in [0,1]$. 
Given some data $f_t \in H_t$ for a.e.~$t \in (0,1)$, consider the dynamic inverse problem of finding a curve $t \mapsto \rho_t \in \M(\olom)$ such that
\begin{equation} \label{prel:inverse}
K_t^* \rho_t = f_t \quad \text{for a.e.} \,\, t \in (0,1) \,. 
\end{equation}
It has been recently proposed \cite{bf} to regularize the above problem with the optimal transport energy $J_{\alpha,\beta}$ defined in \eqref{prel:reg}, where $\alpha,\beta>0$ are fixed parameters. This leads to consider the Tikhonov functional $T_{\alpha,\beta} \colon \M \to [0,+\infty]$ with associated minimization problem
\begin{equation}
\min_{\mu  \in \M} T_{\alpha,\beta}(\mu) \,, \quad T_{\alpha,\beta}(\mu):= 
  \F (\mu)
  + J_{\alpha,\beta}(\mu)\,,  
  \label{eq:prel_main_prob} \tag{$\mathcal{P}$}
\end{equation}
where the fidelity term $\F \colon \M \to [0,+\infty]$ is defined by
\begin{equation} \label{prel:fidelity}
\F(\mu):=
\begin{cases}
\displaystyle \frac{1}{2}\int_0^1 \norm{K^*_t \rho_t - f_t}^2_{H_t}dt &  \, \text{ if } \rho=dt \otimes \rho_t \,, \,\, (t\mapsto \rho_t) \in \curves \,, \\
+ \infty  & \, \text { otherwise.}
\end{cases}
\end{equation}
In the following we will denote by $f$, $K^*\rho$, $Kf$ the maps $t\mapsto f_t$, $t \mapsto K_t^* \rho_t$, $t \mapsto K_tf_t$ respectively,  where $f_t \in H_t$ for a.e.~$t \in (0,1)$ and $\rho_t$ is the disintegration of $\rho$ with respect to time.  
The fidelity term $\F$ serves to track the discrepancy in \eqref{prel:inverse} continuously in time. Following \cite{bf}, this is achieved by introducing the Hilbert space of square integrable maps $f \colon [0,1] \to H:=\cup_t H_t$, denoted by $\Ltwo $. The data $f$ is then assumed to belong to $\Ltwo$. The assumptions under which this procedure can be made rigorous are briefly summarized in Section \ref{sec:assumptions} below, see \ref{ass:H1}-\ref{ass:H3}, \ref{ass:K1}-\ref{ass:K3}. Under these assumptions, we have that $\mathcal{F}$ is well defined, see Remark \ref{rem:fidelity}.  
Such framework allows to model a variety of time-dependent acquisition strategies in dynamic imaging, as seen in Section \ref{subsec:Fourier_example}. %
We are now ready to recall an existence result for \eqref{eq:prel_main_prob} (see \cite[Theorem 4.4]{bf}).

\begin{thm} \label{thm:existence}
Assume \ref{ass:H1}-\ref{ass:H3}, \ref{ass:K1}-\ref{ass:K3} as in Section \ref{sec:assumptions}. Let $f \in\Ltwo$ and $\alpha, \beta >0$.  Then $T_{\alpha,\beta}$ is weak* lower semicontinuous on $\M$ and
there exists $\mu^* \in \mathcal{D}$ that solves the minimization problem \eqref{eq:prel_main_prob}. Moreover $\rho^*$ disintegrates in $\rho^*=dt \otimes \rho_t^*$ with $(t \mapsto \rho^*_t) \in \pcurves$.
If in addition $K_t^*$ is injective for a.e.~$t \in (0,1)$, then $\mu^*$ is unique. 
\end{thm}

The proposed numerical approach for 
\eqref{eq:prel_main_prob} is based on the conditional gradient method, which consists in seeking minimizers of local linear
approximations of the target functional. %
As standard practice \cite{Bredies2009,K-Pikkarainen,pieper}, we first replace \eqref{eq:prel_main_prob} with a surrogate minimization problem, by defining the functional $\tilde{T}_{\alpha,\beta}$ as in \eqref{eq:main_prob_mod} below. %
The key step in a conditional gradient method is then to find the steepest descent direction for a linearized version of $\tilde{T}_{\alpha,\beta}$. In Section \ref{subsec:linearized} we show that, in order to find such direction, it is sufficient to solve the minimization problem 
\begin{equation} \label{prel:aux4} %
\min_{\mu \in \ext(C_{\alpha,\beta})} -\ps{\rho}{w} \,,%
\end{equation}
where $C_{\alpha,\beta}:=\{J_{\alpha,\beta} \leq 1\}$, $w_t:=-K_t(K_t^* \tilde{\rho_t} - f_t) \in C(\olom)$ is the dual variable associated to the current iterate $(\tilde{\rho},\tilde{m})$,  and the linear term  $\mu \mapsto \ps{\rho}{w}$ is defined in \eqref{eq:scalarproduct_alg} below. 
Finally, in Section~\ref{subsec:primal_dual_gap} we define the primal-dual gap $G$ associated to \eqref{eq:prel_main_prob}, and prove optimality conditions. %

\subsection{Functional analytic setting for time continuous fidelity term}\label{sec:assumptions}

  In order to define the continuous sampling fidelity term $\F$ at \eqref{prel:fidelity}, the authors of \cite{bf} introduce suitable assumptions on the measurement spaces $H_t$ and on the forward operators $K_t^* \colon \M(\olom) \to H_t$. %

    \begin{assumption}
    For a.e.~$t \in (0,1)$, let $H_t$ be a real Hilbert space with norm $\norm{\cdot}_{H_t}$ and scalar product $\ps{\cdot}{\cdot}_{H_t}$.  Let $D$ be a real Banach space with norm denoted by $\norm{\cdot}_{D}$. Assume that for a.e.~$t \in (0,1)$ there exists a linear continuous operator $i_t : D \to H_t$ with the following properties:
    \begin{enumerate}[label=\textnormal{(H\arabic*)}]
    \item $\norm{i_t} \leq C$ for some constant $C>0$ not depending on $t$, \label{ass:H1} 
    \item $i_t(D)$ is dense in $H_t$, \label{ass:H2}
    \item  the map $t \in [0,1] \mapsto \ps{i_t\varphi}{i_t \psi}_{H_t} \in \mathbb{R}$ is Lebesgue measurable for every fixed $\varphi,\psi \in D$. 	\label{ass:H3}
    \end{enumerate}
\end{assumption}

  Setting $H := \bigcup_{t\in [0,1]} H_t$, it is possible to define the space of square integrable maps $f \colon [0,1] \to H$ such that $f_t \in H_t$ for a.e.~$t \in (0,1)$, that is,
  \begin{equation} \label{def:ltwoh}
  \Ltwo=L^2([0,1] ;H):= \left\{ f \colon [0,1] \to H \, \colon \, f \ \text{strongly measurable}, \int_0^1 \|f_t\|_{H_t}^2 \,dt < \infty \right\}\,.
  \end{equation}
  The strong measurability mentioned in \eqref{def:ltwoh} is an extension to time dependent spaces of the classical notion of strong measurability for Bochner integrals. The common subset $D$ is employed to construct step functions in a suitable way. An important property of strong measurability is that $t \mapsto \ps{f_t}{g_t}_{H_t}$ is Lebesgue measurable whenever $f,g$ are strongly measurable \cite[Remark 3.4]{bf}. %
 Moreover $\Ltwo$ is a Hilbert space with inner product and norm given by
  \begin{equation} \label{eq:scalar_product}
  \ps{f}{g}_{L^2_H}:= \int_0^1 \ps{f_t}{g_t}_{H_t} \, dt \,, \quad
  \norm{f}_{L^2_H}:=\left( \int_0^1 \norm{f_t}_{H_t}^2 \, dt \right)^{1/2}\,,
  \end{equation}
  respectively \cite[Theorem 3.13]{bf}.  We refer the interested reader to \cite[Section 3]{bf} for more details on the construction of such spaces and their properties.  
  We will now state the assumptions required for the measurement operators $K_t^*$.
    
    \begin{assumption}
   For a.e.~$t \in (0,1)$ the linear continuous operators  $K_t^* : \mathcal{M}(\olom) \to H_t$ satisfy:
	\begin{enumerate}[label=\textnormal{(K\arabic*)}]
    \item $K_t^*$ is weak*-to-weak continuous, with pre-adjoint denoted by $K_t: H_t \to C(\overline \Omega)$,  \label{ass:K1}
    \item $\norm{K^*_t} \leq C$ for some constant $C>0$ not depending on $t$, \label{ass:K2}
    \item the map $t \in [0,1] \mapsto K_t^* \rho \in H_t$ is strongly measurable for every fixed $\rho \in \mathcal{M}(\olom)$.	\label{ass:K3}
    \end{enumerate}
\end{assumption}

\begin{rem}\label{rem:fidelity}
After assuming \ref{ass:H1}-\ref{ass:H3}, \ref{ass:K1}-\ref{ass:K3}, the fidelity term $\F$ introduced at \eqref{prel:fidelity} is well-defined. Indeed, the conditions $\rho=dt \otimes \rho_t$ and $(t \mapsto \rho_t) \in \curves$ imply that $t \mapsto K_t^*\rho_t$ belongs to $\Ltwo$ by Lemma \ref{lem:prop K}. We further remark that $\F(\mu)$ is finite whenever $J_{\alpha,\beta}(\mu)<+\infty$, as in this case we have $\rho =dt \otimes \rho_t$ with $(t \mapsto \rho_t) \in \pcurves$, by Lemmas \ref{lem:prop cont}, \ref{lem:prop B}. 
\end{rem}

\subsection{Surrogate minimization problem} \label{subsec:surrogate}

Let $f \in \Ltwo$ and define the map  $\varphi \colon \R_+\rightarrow \R_{+}$
  \begin{equation} \label{def:fi}
    \varphi(t) := 
    \begin{cases} 
        t  & \text{ if } t \leq M_0, \\
        \frac{t^2 +  M_0^2}{2 M_0} & \text{ if } t > M_0,
  \end{cases} 
  \end{equation}
where we set $M_0 := T_{\alpha,\beta}(0)$. Notice that by \eqref{formula B} we
have $J_{\alpha,\beta}(0)=0$, so that
\begin{equation} \label{def:M0}
M_0= \frac12 \int_0^1 \norm{f_t}_{H_t}^2 \, dt \,,
\end{equation}
highlighting the dependence of $\varphi$ on $f$.
Recalling the definition of $\F$ at \eqref{prel:fidelity},
 define the surrogate minimization problem
  \begin{equation} \label{eq:main_prob_mod} \tag{$\tilde{\mathcal{P}}$}
  \min_{\mu \in \M} \tilde{T}_{\alpha,\beta} (\mu)\,, \quad \tilde{T}_{\alpha,\beta} (\mu) :=\F(\mu) + \varphi(J_{\alpha,\beta} (\mu))   \,.
  \end{equation} 
Notice that \eqref{eq:prel_main_prob} and \eqref{eq:main_prob_mod} 
  share the same set of minimizers, and they are thus equivalent. This is readily seen after noting that solutions to \eqref{eq:prel_main_prob} and \eqref{eq:main_prob_mod} belong to the set
  $\{\mu \in \M :J_{\alpha,\beta}(\mu) \leq M_0\}$, thanks to the estimate $t \leq \f(t)$, and that $T_{\alpha,\beta}$ and $\tilde
  T_{\alpha,\beta}$ coincide on the said set. 
  
  We remark that the surrogate minimization problem
  \eqref{eq:main_prob_mod} is a technical modification of
  \eqref{eq:prel_main_prob} introduced just to ensure that the partially
  linearized problem defined in the following section is coercive.

  \subsection{Linearized problem}\label{subsec:linearized}

  Fix some data $f \in \Ltwo$ and a curve $(t \mapsto \tilde
  \rho_t) \in \curves$.  We  define the associated dual variable $t \mapsto w_t$ by
  \begin{equation}
    w_t := -K_t(K_t^* \tilde{\rho}_t -  f_t ) \in C(\olom)\,,
    \label{eq:dual_variable}
  \end{equation}
  and the map $\mu \in \M \mapsto \langle \rho, w\rangle \in \R \cup\{\pm \infty\}$ as   
  \begin{equation}\label{eq:scalarproduct_alg}
  \ps{\rho}{w} :=  
  \begin{cases}
  \displaystyle \int_{0}^1\langle \rho_t, w_t \rangle_{\M(\olom),C(\olom)} \, dt & \,\, \text{if } \rho=dt \otimes \rho_t\, , \,\, (t \mapsto \rho_t) \in \curves\,, \\ %
  -\infty &\,\, \text{otherwise.}
 \end{cases}
  \end{equation}
\begin{rem}\label{rem:scalar_product}
The above map is well-defined: indeed, assuming that $(t \mapsto \rho_t) \in \curves$, we have $(t \mapsto K_t^*\rho_t ) \in \Ltwo$ by Lemma \ref{lem:prop K}. Similarly, also $(t \mapsto K_t^*\tilde\rho_t) \in \Ltwo$. Thus, recalling \ref{ass:K1}, we infer
 \begin{equation} \label{eq:equiv_prod_scalare}
\ps{\rho}{w}= - \ps{K^*\rho}{K^* \tilde \rho - f}_{L^2_H}\,,
 \end{equation}
 which is well-defined and finite, being a scalar product in the Hilbert space $\Ltwo$ (see \eqref{eq:scalar_product}). Moreover if $J_{\alpha,\beta}(\mu)<+\infty$, then $\ps{\rho}{w}$ is finite, since $\rho =dt \otimes \rho_t$ for $(t \mapsto \rho_t) \in \pcurves$, by Lemmas \ref{lem:prop cont}, \ref{lem:prop B}.
  \end{rem}
   Let $\f$ and $M_0$ be as in \eqref{def:fi}-\eqref{def:M0}. We consider the following linearized version of \eqref{eq:main_prob_mod}
  \begin{equation} \label{aux2}
  \min_{\mu \in \M} - \ps{\rho}{w} + \f (J_{\alpha,\beta}(\mu))\,,
  \end{equation}
  which is well-posed by Theorem \ref{ex min gen}. %
 The objective of this section is to prove the existence of a solution to \eqref{aux2} belonging, up to a multiplicative constant, to the extremal points of the sublevel set $C_{\alpha,\beta}:=\{J_{\alpha,\beta} \leq 1\}$. %
To this end, consider the problem
  \begin{equation} \label{eq:aux4}
  \min_{\mu \in C_{\alpha,\beta}} -\ps{\rho}{w}\,.
  \end{equation}
  In the following proposition we prove that \eqref{eq:aux4} admits a minimizer $\mu^* \in \ext(C_{\alpha,\beta})$. Moreover we show that a suitably rescaled version of $\mu^*$ solves \eqref{aux2}.

 \begin{prop}\label{prop:wolfe equiv}
  Assume \ref{ass:H1}-\ref{ass:H3}, \ref{ass:K1}-\ref{ass:K3} as in Section \ref{sec:assumptions}. Let  $f \in \Ltwo$, $\alpha, \beta >0$. Then, there exists a solution $\mu^* \in \ext(C_{\alpha,\beta})$ to \eqref{eq:aux4}. Moreover $M \mu^*$ is a minimizer for \eqref{aux2},
  where 
  \begin{equation}
      \label{eq:M_choice}
      M :=
      \begin{cases}
          0  & \,\, \text{ if } \,\,\, \langle \rho^*, w \rangle  \leq 1 \,,\\
          M_0  \langle \rho^*, w \rangle  & \,\, \text{ if } \,\,\, \langle \rho^*, w \rangle > 1 \,.
  \end{cases}
    \end{equation}
    \end{prop}

 The above statement is reminiscent of the classical Bauer Maximum Principle \cite[Theorem 7.69]{aliprantis}. In our case, however, there is no clear topology that makes the set $C_{\alpha,\beta}$ compact and the linearized map defined in \eqref{eq:scalarproduct_alg} continuous (or upper semicontinuous). Therefore an ad-hoc proof is required.
  \begin{proof}
   Let $ \hat\mu \in C_{\alpha,\beta}$ be a solution to \eqref{eq:aux4}, which exists thanks to Theorem \ref{ex min gen} with the choice $\f(t):=\rchi_{(-\infty,1]}(t)$. Consider the set $S := \left\{\mu \in C_{\alpha,\beta} :  \langle\rho,w\rangle=\langle \hat\rho,w\rangle\right\}$ of all solutions to \eqref{eq:aux4}. 
Note that $S$ is bounded with respect to the total variation on $\M$, due to  \eqref{lem:prop J est} and definition of $C_{\alpha,\beta}$. In particular the weak* topology of $\M$ is metrizable in $S$. We claim that $S$ is compact in the same topology. Indeed, given a sequence $\{\mu^n\}$ in $S$ we have by definition that
  \begin{gather} \label{eq:min ext:1}
  \sup_n J_{\alpha,\beta}(\mu^n) \leq 1\,, \qquad
  \langle\rho^n,w\rangle = \langle \hat\rho,w\rangle \,\, \text{ for every }  \,\, n \in \N\,.
  \end{gather} 
  Therefore, \eqref{eq:min ext:1} and Lemma \ref{lem:prop J} imply that, up to subsequences, $\mu^n$ converges to some $\mu \in \M$ in the sense of \eqref{topology}. By \eqref{eq:min ext:1} and by the weak* sequential lower semicontinuity of $J_{\alpha,\beta}$ (Lemma \ref{lem:prop J}), we infer $\mu \in C_{\alpha,\beta}$. Moreover by \eqref{eq:min ext:1}, \eqref{eq:equiv_prod_scalare} and Lemma \ref{lem:prop K} we also conclude that $\mu \in S$, hence proving compactness. Also notice that $S$ is convex due to the convexity of $J_{\alpha,\beta}$ (Lemma \ref{lem:prop J}) and linearity of the constraint. Since $S \neq \emptyset$,  by Krein-Milman's Theorem we have that $\ext(S) \neq \emptyset$. Let $\mu^* \in \ext(S)$. If we show that $\mu^* \in \ext(C_{\alpha,\beta})$, the thesis is achieved by definition of $S$. Hence, assume that $\mu^*$ can be decomposed as
   \begin{equation} \label{eq:min ext:3}
  \mu^* = \lambda \mu^1 + (1-\lambda) \mu^2 
  \end{equation}
  with $\mu^j =(\rho^j,m^j)\in C_{\alpha,\beta}$ and $\lambda \in (0,1)$. Assume that $\mu^1$ belongs to $C_{\alpha,\beta} \smallsetminus S$. By \eqref{eq:min ext:3} and the minimality of the points in $S$ for \eqref{eq:aux4} we infer 
  $-\langle \hat\rho,w\rangle<-\langle \rho^*, w\rangle$, 
  which is a contradiction since $\mu^* \in S$. Therefore $\mu^1 \in S$. Similarly also $\mu^2 \in S$. Since $\mu^* \in \ext(S)$, from \eqref{eq:min ext:3} we infer
  $\mu^* = \mu^1=\mu^2$, showing that $\mu^* \in \ext(C_{\alpha,\beta})$.
  
Assume now that $\mu^* \in \ext(C_{\alpha,\beta})$ minimizes in  \eqref{eq:aux4}. If $\mu^*=0$, it is straightforward to check that $0$ minimizes in \eqref{aux2}. Hence assume $\mu^* \neq 0$, so that $J_{\alpha,\beta}(\mu^*)>0$ by \eqref{lem:prop J est}. Since the functional at \eqref{eq:aux4} is linear and $\mu^*$ is a minimizer, we can scale by $J_{\alpha,\beta}(\mu^*)$ and exploit the one-homogeneity of $J_{\alpha,\beta}$ to obtain $J_{\alpha,\beta}(\mu^*)=1$.  
For every $\mu=(\rho,m) \in \mathcal{M}$ such that $J_{\alpha,\beta}(\mu)<+\infty$ one has
   \begin{equation}\label{eq:fenchel:0}
  - \langle \rho,w \rangle + \varphi(J_{\alpha,\beta}(\mu)) \geq  - J_{\alpha,\beta}(\mu) \langle \rho^*, w\rangle+  \varphi(J_{\alpha,\beta}(\mu))\,,
  \end{equation}
  since $\mu^*$ is a minimizer, $J_{\alpha,\beta}$ is non-negative and one-homogeneous, and since $J_{\alpha,\beta}(\mu)=0$ if and only if $\mu=0$ by Lemma \ref{lem:prop B}. Again one-homogeneity implies 
    \begin{equation}\label{eq:fenchel}
  \inf_{\mu \in \calM} - J_{\alpha,\beta}(\mu) \ps{\rho^*}{w} +  \varphi(J_{\alpha,\beta}(\mu))  =  \inf_{\lambda \geq 0} - \lambda \langle \rho^*, w\rangle+  \varphi(\lambda )\, .
  \end{equation}
It is immediate to check that $M$ defined in \eqref{eq:M_choice} is a minimizer for the right-hand side problem in \eqref{eq:fenchel}. 
  Hence from \eqref{eq:fenchel:0}-\eqref{eq:fenchel}, one-homogeneity of $J_{\alpha,\beta}$ and the fact that $J_{\alpha,\beta}(\mu^*) = 1$, we conclude that $M \mu^*$ is a minimizer for \eqref{aux2}.
  \end{proof}

\subsection{The primal-dual gap} \label{subsec:primal_dual_gap}
In this section we introduce the primal-dual gap associated to \eqref{eq:prel_main_prob}. 
  \begin{definition} \label{def:pd_gap} 
  The primal-dual gap is defined as the map
   $G \colon \M \to [0,+\infty]$ such that  
  \begin{equation}
  \label{eq:dual_gap}
  G(\mu):= 
  \begin{cases}
   J_{\alpha,\beta}(\mu) - \f (J_{\alpha,\beta}(\hat\mu)) - \ps{\rho - \hat\rho}{w}  & \, \text{ if } \, J_{\alpha,\beta}(\mu)<+\infty\,, \\
   + \infty    &  \, \text{ otherwise,}
   \end{cases}
  \end{equation}
  for $\mu \in \M$. 
Here $w_t := -K_t(K_t^*  \rho_t -  f_t )$, the product 
    $\langle \cdot,\cdot \rangle$ is defined in \eqref{eq:scalarproduct_alg}, the map $\f$ at \eqref{def:fi},  and $ \hat\mu \in \M$ is a solution to \eqref{aux2}.%
 \end{definition}
 
Notice that $G$ is well-defined. Indeed, assume that $\mu \in \M$ is such that $J_{\alpha,\beta}(\mu)<+\infty$ %
and let $ \hat\mu \in \M$ be a solution to \eqref{aux2}. %
In particular $J_{\alpha,\beta}(\hat \mu)<+\infty$ %
 by Theorem \ref{ex min gen}. Therefore the scalar product in \eqref{eq:dual_gap} is finite, see Remark \ref{rem:scalar_product}.  The purpose of $G$ becomes clear in its relationship with the \textit{functional distance} associated to $T_{\alpha,\beta}$, which is defined by
  \begin{equation} \label{def:residual}
  r(\mu):=T_{\alpha,\beta}(\mu) - \min T_{\alpha,\beta} \,,
  \end{equation}
  for all $\mu \in \M$. Such relationship is described in the following lemma. We remark that a similar result is standard in the context of Frank-Wolfe-type algorithms and generalized conditional gradient methods (see e.g. \cite[Lemma 5.5]{K-Pikkarainen} and \cite[Section 2]{jaggi2013}). Due to the specificity of our dynamic problem and for sake of completeness we present it in our setting as well.

  \begin{lem} \label{lem:primal dual}
  Let $\mu \in \M$ be such that $J_{\alpha,\beta}(\mu)<+\infty$. Then 
  \begin{equation} \label{est primal dual}
 r(\mu) \leq  G(\mu )\,. 
  \end{equation}
  Moreover $\mu^*$ solves \eqref{eq:prel_main_prob} if and only if $G(\mu^*)=0$. 
  \end{lem}

  \begin{proof} 
Let $w$ and $\hat \rho$ be as in \eqref{eq:dual_gap}.  
By using the Hilbert structure of $\Ltwo$, for any $\tilde{\mu}$ such that $J_{\alpha,\beta}(\tilde \mu)<+\infty$, we have, by the polarization identity, %
  \begin{equation}\label{eq:bilinear2}
  -\langle \rho - \tilde\rho, w\rangle = \frac{\|K^* \rho - f\|^2_{L^2_H}}{2} - \frac{\|K^* \tilde\rho - f\|^2_{L^2_H}}{2} + \frac{\|K^*(\tilde\rho - \rho)\|^2_{L^2_H}}{2}\,.
  \end{equation}
	Let $\mu^*$ be a minimizer for $T_{\alpha,\beta}$, which exists by Theorem \ref{thm:existence}. Since $J_{\alpha,\beta}(\mu^*) \leq T_{\alpha,\beta}(\mu^*) \leq T_{\alpha,\beta}(0)=M_0$, by definition of $\f$ we have $\varphi(J_{\alpha,\beta}(\mu^*)) = J_{\alpha,\beta}(\mu^*)$. Using \eqref{eq:bilinear2} and the definition of $G$ in \eqref{eq:dual_gap}, where $\hat  \mu$ is chosen to be a solution to \eqref{aux2}, we obtain
  \begin{align*}
  G(\mu) %
  & \geq -\langle \rho - \rho^*, w\rangle + J_{\alpha,\beta}(\mu) -  J_{\alpha,\beta}(\mu^*)\\
  & \geq \frac{\|K^* \rho - f\|^2_{L^2_H}}{2} - \frac{\|K^* \rho^*  - f\|^2_{L^2_H}}{2}+ J_{\alpha,\beta}(\mu) -  J_{\alpha,\beta}(\mu^*)\\
  & = T_{\alpha,\beta}(\mu) - T_{\alpha,\beta}(\mu^*) = r(\mu)\,,
  \end{align*}
  proving \eqref{est primal dual}. If $G(\mu^*)=0$ then $\mu^*$ minimizes in \eqref{eq:prel_main_prob} by \eqref{est primal dual}. Conversely, assume that $\mu^*$ is a solution of \eqref{eq:prel_main_prob} and denote by $w^*_t:=-K_t(K_t^*\rho_t^* - f_t)$ the associated dual variable. Let $\mu$ be arbitrary and such that $J_{\alpha,\beta}(\mu)<+\infty$. Let $s \in [0,1]$ and set $\mu^s:=\mu^*+s(\mu-\mu^*)$. By convexity of $J_{\alpha,\beta}$ (see Lemma \ref{lem:prop J}) we have that $J_{\alpha,\beta}(\mu^s)<+\infty$. As $\mu^*$ is optimal in \eqref{eq:prel_main_prob} we infer
  \[
  \begin{aligned}
  0 & \leq T_{\alpha,\beta}(\mu^s) - T_{\alpha,\beta}(\mu^*) \leq  \frac{ \norm{K^*\rho^s - f}^2_{L^2_H}}{2} - \frac{ \norm{K^*\rho^* - f}^2_{L^2_H}}{2} +  s(J_{\alpha,\beta}(\mu)-J_{\alpha,\beta}(\mu^*)) \\
  & = - s \ps{\rho-\rho^*}{w^*} + s^2 \,  \frac{\norm{K^*(\rho^* - \rho)}^2_{L^2_H}}{2} +   s(J_{\alpha,\beta}(\mu)-J_{\alpha,\beta}(\mu^*))\,,
  \end{aligned}
  \]
  where we used convexity of $J_{\alpha,\beta}$, and the identity at \eqref{eq:bilinear2} with respect to $w^*, \rho^*$ and $\rho^s$. Dividing the above inequality by $s$ and letting $s \to 0$ yields
  \begin{equation}\label{eq:bilinear3}
 0 \leq  -\ps{\rho-\rho^* }{w^*} +J_{\alpha,\beta}(\mu)  - J_{\alpha,\beta}(\mu^*)  \,,
  \end{equation}
  which holds for all $\mu$ with $ J_{\alpha,\beta}(\mu)<+\infty$. Now notice that $J_{\alpha,\beta}(\mu^*) \leq T_{\alpha,\beta}(\mu^*) \leq M_0$, since $\mu^*$ solves \eqref{eq:prel_main_prob}. Therefore $\f(J_{\alpha,\beta}(\mu^*))=J_{\alpha,\beta}(\mu^*)$. Moreover $t \leq \f(t)$ for all $t \geq 0$. As a consequence of \eqref{eq:bilinear3} we then infer 
  \[
  \begin{aligned}
  -\ps{\rho^*}{w^*}+\f(J_{\alpha,\beta}(\mu^*)) &=
    -\ps{\rho^*}{w^*}+J_{\alpha,\beta}(\mu^*) \\
    & \leq 
      -\ps{\rho}{w^*}+J_{\alpha,\beta}(\mu) 
       \leq   -\ps{\rho}{w^*}+\f(J_{\alpha,\beta}(\mu))\,,
       \end{aligned}
  \] 
  proving that $\mu^*$ minimizes in \eqref{aux2} with respect to $w^*$. Therefore, by definition, $G(\mu^*)=0$.
  \end{proof}

\section{The algorithm: theoretical analysis} \label{sec:algorithm}

In this section we give a theoretical description of the dynamic generalized conditional gradient algorithm anticipated in the introduction, which we call \emph{core algorithm}. %
The proposed algorithm aims at finding minimizers to $T_{\alpha,\beta}$ as defined in \eqref{eq:prel_main_prob}, 
for some fixed data $f \in \Ltwo$ and parameters $\alpha,\beta>0$. 
It is comprised of an \emph{insertion step}, where one seeks a minimizer to \eqref{eq:aux4} among the extremal points of the set $C_{\alpha,\beta}:=\{J_{\alpha,\beta}\leq 1\}$, and of a \emph{coefficients optimization step}, which will yield a finite dimensional quadratic program. As a result, each iterate will be a finite linear combination, with non-negative coefficients, of points in $\ext(C_{\alpha,\beta})$. 
We remind the reader that $\ext(C_{\alpha,\beta})= \{0\} \cup \points$ in view of Theorem \ref{thm:extremal}, where $\points$ is defined at \eqref{ext_points}. From the definition of $\points$, we see that,
except for the zero element, the extremal points are in 1-on-1 correspondence with the space of curves $\AC^2:=\AC^2([0,1]; \overline \Omega)$. This observation
motivates us to define the 
atoms of our problem.

\begin{definition}[Atoms] \label{def:atom} We denote by $\ACe^2$ the one-point
         extension of the set $\AC^2$, where we include a point
        denoted by $\gamma_\infty$.
  For any $\gamma \in \AC^2$ we name as \textit{atom} the respective extremal point
  $\mu_\gamma=(\rho_\gamma, m_\gamma) \in \M$ defined according to \eqref{ext_meas}. 
    For $\gamma_{\infty}$ the corresponding atom is defined by $\mu_{\gamma_\infty}:=(0,0)$. %
    We call \textit{sparse} any measure $\mu  \in \M$ such that  
 \begin{equation}\label{eq:sparse_meas}
 \mu=\sum_{j=1}^N c_j \mu_{\gamma_j} %
    \end{equation} 
    for some $N \in \N$, $c_j > 0$ and $\gamma_j \in \AC^2$, with $\gamma_i \neq \gamma_j$ for $i \neq j$. 
\end{definition}
Note that $\gamma_\infty$ can be regarded as the infinite length curve: indeed
if $\{\gamma^n\}$ in $\AC^2$ is a sequence
of curves with diverging length, that is $\int_0^1 |\dot \gamma^n| \,dt \to \infty$ as $n \to \infty$, then
$\rho_{\gamma^n} \weakstar \rho_{\gamma_\infty}$, since $a_{\gamma^n} \to
0$ by H\"older's inequality.  
Additionally, it is convenient to introduce the following map associating to vectors of curves and coefficients the corresponding sparse measure: 
\begin{equation} \label{eq:map}
(\boldsymbol{\gamma}, \boldsymbol{c}) \mapsto \mu:= \sum_{j=1}^N c_j \mu_{\gamma_j} \,,
\end{equation}
where $N \in \N$ is fixed and $\boldsymbol{c}:= (c_1, \ldots, c_{N})$ with $c_j > 0$, $\boldsymbol{\gamma} := (\gamma_1, \ldots, \gamma_{N})$, with $\gamma_j \in \AC^2$.

\begin{rem} \label{rem:crossing}
The decomposition in extremal points of a given sparse measure might not be unique, that is, the map at \eqref{eq:map} is not injective. For example, let $N:=2$, $\Om:=(0,1)^2$ and 
\[
\begin{aligned}
& \gamma_1(t):=(t,t) \,, \,\,\,\,\quad & \tilde{\gamma}_1 (t):= (t,t) \rchi_{[0,1/2)}(t)+ (t,1-t) \rchi_{[1/2,1]}(t)\,, \\
&\gamma_2(t):=(t,1-t) \,, &\tilde{\gamma}_2 (t):= (t,1-t) \rchi_{[0,1/2)}(t)+ (t,t) \rchi_{[1/2,1]}(t)\,.
\end{aligned}
\]
Note that injectivity fails for $((\gamma_1,\gamma_2),(1,1))$ and $((\tilde{\gamma}_1,\tilde{\gamma}_2),(1,1))$, given that they map to the same measure $\mu^*$, but $\gamma_1$ and $\gamma_2$ cross at $t=1/2$, while $\tilde{\gamma}_1$ and $\tilde{\gamma}_2$ rebound. This observation is relevant for the algorithms presented, seeing that they operate in terms of extremal points: if for example $\mu^*$ was the unique solution to \eqref{eq:prel_main_prob} for some data $f$, due to the lack of unique sparse representation for $\mu^*$,  the numerical reconstruction could favor the representation having the least energy in terms of the regularizer $J_{\alpha,\beta}$. This is not surprising, since our method aims at reconstructing sparse measures, rather than their extremal points. %
A numerical example displaying the behavior of our algorithm on crossings, such as the case of $\mu^*$, is given in Section \ref{subsec:Ex3}. 
\end{rem}

The rest of the section is organized as follows. In Section \ref{subsec:alg_descr} we present the core algorithm, describing its basic steps and summarizing it in Algorithm \ref{alg:core}. In Section \ref{subsec:quadratic} we discuss the equivalence of the coefficients optimization step to a quadratic program, while in Section \ref{sec:convergence} we show sublinear convergence of Algorithm \ref{alg:core} in terms of the residual defined at \eqref{def:residual}. In  Section \ref{subsec:stop} we detail on a theoretical stopping criterion for our algorithm.  To conclude, in Section \ref{sec:time-discrete}, we give a description of how to alter Algorithm \ref{alg:core} in case the fidelity term $\mathcal{F}$ at \eqref{eq:prel_main_prob} is replaced by a time-discrete version. All the results presented in this section and in the above will hold also for this particular case, with minor modifications.

\subsection{Core Algorithm} \label{subsec:alg_descr}

The core algorithm consists of two steps. In the first one, named the
\emph{insertion step}, an atom is added to the current iterate, this atom being the minimizer of the linearized problem defined at \eqref{eq:aux4}. In the second step, named the
\emph{coefficients optimization step}, the atoms are fixed and their associated
weights are optimized to minimize the target functional $T_{\alpha,\beta}$
defined in \eqref{eq:prel_main_prob}. %
In what follows $f \in \Ltwo$ is a given datum and $\alpha,\beta>0$ are fixed parameters. %

\subsubsection{Iterates} \label{subsec:iterates}
We initialize the algorithm to the zero atom $\mu^0 := 0$. The $n$-th iteration
 $\mu^n$ is a \emph{sparse} element of $\mathcal{M}$ according to \eqref{eq:sparse_meas}, that is, 
  \begin{equation}\label{eq:alg_iterates}
    \mu^n = \sum_{j=1}^{N_n} c_j^n \mu_{\gamma_j^n} ,
  \end{equation}
  where $N_n \in \N \cup\{0\}$, $\gamma_j^n \in \AC^2$, $c_j^n >0$ and $\gamma_i \neq \gamma_j$ if $i \neq j$. Notice that $N_n$ is counting the number of atoms present at the $n$-th iteration, and is not necessarily equal to $n$,
  since the optimization step could discard atoms by setting their associated
  weights to zero. In practice, Algorithm \ref{alg:core} operates in terms of curves and weights. That is, the $n$-th iteration outputs pairs $(\gamma_j,c_j)$ with $\gamma_j \in \AC^2$, $c_j >0$: the iterate at \eqref{eq:alg_iterates} can be then constructed via the map \eqref{eq:map}.

\subsubsection{Insertion step}
\label{subsec:insertion}
Assume $\mu^n$ is the current iterate. 
Define the dual variable associated to $\mu^n$ as in \eqref{eq:dual_variable}, that is,
\begin{equation} \label{eq:alg_dual_var}
  w^n_t := -K_t (K_t^* \rho^n_t - f_t  ) \in C(\olom) \quad \text{for a.e.}\  t\in (0,1)\,.
\end{equation}
With it, consider the minimization problem of the form \eqref{prel:aux4}, that is,
\begin{equation} \label{eq:alg_insertion}
  \min_{\mu \in \ext(C_{\alpha,\beta})} -\ps{\rho}{w^n}\,,
\end{equation}
where the term $\ps{\cdot}{\cdot}$ is defined in \eqref{eq:scalarproduct_alg}. We recall that \eqref{eq:alg_insertion} admits solution by Proposition \ref{prop:wolfe equiv}. Thanks to the characterization $\ext(C_{\alpha,\beta})=\{0\}\cup \points$ provided by Theorem \ref{thm:extremal}, problem \eqref{eq:alg_insertion} can be cast on the space $\ACe^2$. Indeed, given $\mu \in \points$, following the notations at \eqref{ext_points}-\eqref{ext_meas}, we have that $\mu=\mu_\gamma$ for some $\gamma \in \AC^2$. The curve $t \mapsto \delta_{\gamma(t)}$ belongs to $\pcurves$, and hence $\ps{\rho_\gamma}{w}$ is finite (see Remark \ref{rem:scalar_product}). Thus, by definition, we have
\begin{equation} \label{eq:diff_computation:11}
  \ps{\rho_\gamma}{w^n}= a_\gamma \int_0^1 \ps{\delta_{\gamma(t)}}{w_t^n}_{\M(\olom),C(\olom)} \, dt = a_\gamma \int_0^1 w_t^n(\gamma(t)) \,dt \,,
\end{equation}
showing that \eqref{eq:alg_insertion} is equivalent to
\begin{equation}
  \label{eq:hard_min}
  \min_{\gamma \in \ACe^2}
  -\left< \rho_\gamma ,w^n \right>  %
=  \min_{\gamma \in \AC^2} \left\{0, 
 -a_{\gamma} \int_0^1 w_t^n(\gamma(t))\,dt \right\}\,.
\end{equation}
The insertion step consists 
in finding a curve $\gamma^* \in \ACe^2$ 
that solves \eqref{eq:hard_min}. To such curve, we associate a new atom $\mu_{\gamma^*}$ via Definition \ref{def:atom}. Note that $\gamma^*$ depends on the current iterate $\mu^n$, as well as on the datum $f$ and parameters $\alpha,\beta$: however, in order to simplify notations, we omit such dependencies. After, we have a stopping condition:
\begin{itemize}
  \item if $\left<\rho_{\gamma^*}, w^n\right> \leq 1$, 
     then $\mu^n$ is solution to \eqref{eq:prel_main_prob}. The algorithm outputs $\mu^n$ and stops,  
  \item if $\left< \rho_{\gamma^*}, w^n\right> > 1$, then $\mu^n$ is not a solution to \eqref{eq:prel_main_prob} and $\gamma^* \in \AC^2$. The found atom $\mu_{\gamma^*}$ is inserted in the $n$-th iterate $\mu^n$ and the algorithm continues.
\end{itemize}
The optimality statements in the above stopping condition
correspond to positivity conditions on the subgradient of $T_{\alpha,\beta}$
and are rigorously proven in Section \ref{subsec:stop} below. Moreover, the
mentioned stopping condition can be made quantitative as discussed in Remark
\ref{rem:relax_stop} below.

 \begin{rem} \label{rem:non-linear}
    In this section we will always assume the availability of an exact solution $\gamma^*$ to \eqref{eq:hard_min}. In particular, this allows to obtain a sublinear convergence rate for the core algorithm (Theorem~\ref{thm:convergence} below), and make the stopping condition rigorous. In practice, however, obtaining $\gamma^*$ is not always possible, due to the non-linearity and non-locality of the functional at \eqref{eq:hard_min}. For this reason, in Section \ref{sec:insstepheu}, we propose a strategy aimed at obtaining stationary points of \eqref{eq:hard_min}. Based on such strategy, a relaxed version of the insertion step is proposed for Algorithm \ref{alg:full}, which we employ for the numerical simulations of Section \ref{sec:numerics}.  
\end{rem}

\subsubsection{Coefficients optimization step} \label{subsec:optimization}

This step is realized after the stopping condition is checked, with the condition  $\left< \rho_{\gamma^*}, w^n\right> > 1$ being satisfied. In particular, as observed above, in this case  $\gamma^* \in \AC^2$. We then set  $\gamma_{N_n+1}^n := \gamma^*$ and consider the 
coefficients optimization problem %
\begin{equation}
  \min_{ (c_1,c_2,\ldots,c_{N_n+1}) \in
  \R_{+}^{N_n+1}} T_{\alpha,\beta} \left(
  \sum_{j=1}^{N_n+1} c_j \mu_{\gamma_j^{n}} \right)\,,
  \label{eq:min_coeff}
\end{equation}
where $\mu_{\gamma_j^n}$ for $j=1,\ldots,N_n$ are the atoms present in the $n$-th iterate $\mu^n$. 
If $c \in \R_{+}^{N_n+1}$ is a solution to the above problem, the next iterate is defined by
\begin{equation} \label{def:new_iterate}
  \mu^{n+1} := \sum_{\substack{c_j > 0 }} c_j \mu_{\gamma_j^n}\,,
\end{equation}
thus discarding the curves that do not contribute to \eqref{eq:min_coeff}.

\begin{rem} \label{rem:quadratic_optimization}
  Problem \eqref{eq:min_coeff}
  is equivalent to a quadratic program of the form
  \begin{equation} \label{eq:alg_quadratic_program}
    \min_{c= (c_1,\ldots,c_{N_n+1}) \in \R^{N_n+1}_+} \,\frac{1}{2}\, c^T \Gamma 
    c + b^T c \,,
  \end{equation}
  where $\Gamma \in \R^{(N_n+1) \times (N_n+1)}$ is a positive-semidefinite and symmetric matrix and $b \in \R^{N_n+1}$, as proven in Proposition \ref{prop:quadraticoptimization} below.
  Therefore, throughout the paper, we will always assume the availability of an exact solution to  \eqref{eq:min_coeff}. In practice we solved \eqref{eq:min_coeff} by means of the free Python software package \textit{CVXOPT} \cite{CVXOPT2,CVXOPT1}. %
\end{rem}

\subsubsection{Algorithm summary} \label{subsec:alg_summary}
As discussed in Section \ref{subsec:iterates}, the iterates of Algorithm \ref{alg:core} are pairs of curves and weights $(\gamma_j,c_j)$ for $\gamma \in \AC^2$, $c_j>0$ and $j=1, \ldots, N_n$. In the pseudo-code we denote such iterates with the tuples $\boldsymbol{\gamma} = (\gamma_1, \ldots, \gamma_{N_n})$ and $\boldsymbol{c} = (c_1, \ldots, c_{N_n})$. Note that such tuples vary in size at each iteration, and the initial iterate of the algorithm, that is the zero atom, corresponds to the empty tuples $\boldsymbol{\gamma} = ()$ and $\boldsymbol{c} = ()$. We denote the number of elements contained in a tuple $\boldsymbol{\gamma}$ with the symbol $|\boldsymbol{\gamma}|$. Via the map \eqref{eq:map}, the iterates $(\boldsymbol{\gamma},\boldsymbol{c})$ define a sequence of sparse measures $\{\mu^n\}$ of the form \eqref{eq:alg_iterates}. 
The generated sequence $\{\mu^n\}$   weakly* converges (up to subsequences) to a minimizer of $T_{\alpha,\beta}$ for the datum $f \in \Ltwo$, as shown in Theorem \ref{thm:convergence} below. Notice that the assignments at lines 4 and 9 are meant to choose one element in the respective argmin set.  
The function $\texttt{delete\_zero\_weighted}( \boldsymbol{\gamma}, \boldsymbol{c}) $ at line $11$ in Algorithm \ref{alg:core} is designed to input a tuple 
$(\boldsymbol{\gamma}, \boldsymbol{c})$ 
and output another tuple where the curves $\gamma_j$ and corresponding weights $c_j$ are deleted if $c_j = 0$.

\SetKwInput{KwInput}{Input}
\SetKwInput{KwOutput}{Output}
\newcommand\mycommfont[1]{\footnotesize\ttfamily\textcolor{blue}{#1}}
\SetCommentSty{mycommfont}
\begin{algorithm}[h!]
  \caption{Core algorithm} \label{alg:core}
\DontPrintSemicolon
\KwInput{
Data $f \in \Ltwo
    $, 
      parameters $\alpha,\beta > 0$,
forward operators $K^*_t: \mathcal{M}(\overline \Omega) \to H_t$ \smallskip \smallskip
}
$\boldsymbol{c} \leftarrow ()$, \,\,
$\boldsymbol{\gamma} \leftarrow ()$ \\
\For{$n=0,1,\ldots$} 
{  $N_n \leftarrow |\boldsymbol{\gamma}|$, \,\, $\mu^n \leftarrow \sum_{j=1}^{N_n} c_j \mu_{\gamma_j}$,\,\,
  $w_t^n \leftarrow  -K_t (K_t^* \rho^n_t - f_t)$ \\ 
   \tcc{Insertion step}
  $\gamma^* \leftarrow \argmin_{\gamma \in \ACe^2} \ - \left< \rho_{\gamma} ,w^n\right>$ \\
   \tcc{Stopping condition}
  \If{$\left<\rho_{\gamma^*}, w^n \right> \leq 1 $}
  {
    \Return $\mu^n$
  }
  \Else
  {
 $\boldsymbol{\gamma} \leftarrow (\boldsymbol{\gamma}, \gamma^*)$\\
  \tcc{Coefficients optimization step}
  $\boldsymbol{c} \leftarrow
    \argmin_{(c_1,\ldots,c_{N_n + 1})\in \R_{+}^{N_n + 1}} T_{\alpha,\beta} \left(
  \sum_{j=1}^{N_n}c_j \mu_{\boldsymbol{\gamma}_j}+ c_{N_n + 1}\mu_{\gamma^*} \right) $\\ 
  $N_{n+1}  \leftarrow \# \{j \, \colon \, c_j > 0\}$\\
 $( \boldsymbol{\gamma}, \boldsymbol{c}) \leftarrow \texttt{delete\_zero\_weighted}( \boldsymbol{\gamma}, \boldsymbol{c})$\\
$\mu^{n+1} \leftarrow    \sum_{j=1}^{N_{n+1}} c_j \mu_{\gamma_j}$ 
  }
}
\end{algorithm}

 \subsection{Quadratic optimization} \label{subsec:quadratic}

We prove the statement in Remark \ref{rem:quadratic_optimization}. To be more precise, assume \ref{ass:H1}-\ref{ass:H3}, \ref{ass:K1}-\ref{ass:K3} from Section \ref{sec:assumptions} and let $f \in \Ltwo$, $\alpha,\beta >0$ be given. Fix $N \in \N$, $\gamma_1,\ldots,\gamma_N \in \AC^2$ with $\gamma_i \neq \gamma_j$ for all $i\neq j$, and consider the coefficients optimization problem
\begin{equation} \label{eq:subsec:quadratic}
\min_{c= (c_1,\ldots,c_N)\in \R^N_+} \, T_{\alpha,\beta} \left( \sum_{j=1}^N c_j \mu_{\gamma_j} \right)\,,
\end{equation}
where $\mu_{\gamma_j}$ is defined according to \eqref{ext_meas}. For \eqref{eq:subsec:quadratic} the following holds. 

  \begin{prop}\label{prop:quadraticoptimization}
  Problem \eqref{eq:subsec:quadratic} is equivalent to 
  \begin{equation} \label{eq:subsec:quadratic:2}
\min_{c=(c_1,\ldots,c_N) \in \R^N_+}\, \frac12 c^T \Gamma c + b^T c\,,
\end{equation}
where $\Gamma=(\Gamma_{i,j}) \in \R^{N \times N}$ is a positive semi-definite symmetric matrix and $b=(b_i) \in \R^N$, with 
 \begin{equation} \label{eq:subsec:quadratic:3}
\Gamma_{i,j}:=a_{\gamma_i} a_{\gamma_j} \int_0^1 \ps{K_t^*\delta_{\gamma_i(t)}}{K_t^*\delta_{\gamma_j(t)}}_{H_t}\, dt \,, \quad b_i:= 1- a_{\gamma_i} \int_0^1 \ps{K_t^*\delta_{\gamma_i}(t)}{f_t}_{H_t}\, dt \,.
\end{equation}
  \end{prop}

  \begin{proof}
As $\gamma_j$ is continuous, the curve $t \mapsto \delta_{\gamma_j(t)}$ belongs to $\pcurves$. Hence the map $t \mapsto K_t^* \delta_{\gamma_j(t)}$ belongs to $\Ltwo$ by Lemma \ref{lem:prop K} and the quantities at \eqref{eq:subsec:quadratic:3} are well-defined. 
  Thanks to definition of $T_{\alpha,\beta}$ and Lemma \ref{lem:additivity} we immediately see that 
\[
   \begin{aligned}
 T_{\alpha,\beta} \left( \sum_{j=1}^N c_j \mu_{\gamma_j} \right) &  = 
   \frac12 \sum_{i,j=1}^N c_i c_j   \ps{K^* \rho_{\gamma_i}}{K^* \rho_{\gamma_j}}_{L^2_H}  \\ 
 & \qquad \qquad \qquad  - \sum_{j=1}^N c_j  
   \ps{K^* \rho_{\gamma_j}}{f}_{L^2_H} + \frac{1}{2} \norm{f}_{L^2_H}^2 + \sum_{j=1}^N c_j J_{\alpha,\beta}(\mu_{\gamma_j}) \\
  & = \frac12 c^T \Gamma c + b^T c + M_0\,,  
    \end{aligned}
\]
   where $M_0\geq 0$ is defined at \eqref{def:M0}. This shows that \eqref{eq:subsec:quadratic} and \eqref{eq:subsec:quadratic:2} are equivalent. The rest of the statement follows since $\Gamma$ is the Gramian with respect to the vectors $K^*\rho_{\gamma_1}, \ldots, K^*\rho_{\gamma_N}$ in $L^2_H$.  
  \end{proof}

 \subsection{Convergence analysis}
  \label{sec:convergence}
We prove sublinear convergence for Algorithm \ref{alg:core}. The convergence rate is given in terms of the functional distance \eqref{def:residual} associated to $T_{\alpha,\beta}$. %
Throughout the section we assume that $f \in \Ltwo$ is a given datum, $\alpha, \beta >0$ are fixed regularization parameters and \ref{ass:H1}-\ref{ass:H3}, \ref{ass:K1}-\ref{ass:K3} as in Section \ref{sec:assumptions} hold.  
  The convergence result states as follows.

  \begin{thm} \label{thm:convergence}
    Let $\{\mu^n\}$ be a sequence generated by Algorithm \ref{alg:core}. Then $\{T_{\alpha,\beta}(\mu^n)\}$ is non-increasing and the residual at \eqref{def:residual} satisfies
    \begin{equation}\label{eq:decaytheorem}
    r(\mu^n) \leq \frac{C}{n} \,, \,\, \text{ for all } \,\, n \in \N\,,
    \end{equation}
    where $C>0$ is a constant depending only on
    $\alpha,\beta$, $f$ and $K_t^*$. Moreover each
    accumulation point $\mu^*$ of $\mu^n$ with respect to the weak* topology of
    $\M$ is a minimizer for $T_{\alpha,\beta}$. If $T_{\alpha,\beta}$ admits a
    unique minimizer $\mu^*$, then $\mu^n \weakstar \mu^*$ along the whole
    sequence.
  \end{thm}

  \begin{rem} \label{rem:convergence} The proof of Theorem \ref{thm:convergence} follows similar steps to \cite[Theorem 5.8]{K-Pikkarainen} and \cite[Theorem~5.4]{pieper}. We highlight that the proof of monotonicity of $\{T_{\alpha,\beta}(\mu^n)\}$ and of the decay estimate \eqref{eq:decaytheorem}, does not make full use of the coefficients optimization step (Section \ref{subsec:optimization}) of Algorithm \ref{alg:core}.  Rather, the proof relies on energy estimates for the surrogate iterate $\mu^s := \mu^n + s(M\mu_{\gamma*} - \mu^n)$, where $\gamma^*\in \AC^2$ is a solution of the insertion step \eqref{eq:hard_min}, $M \geq 0$ a suitable constant, and the step-size $s$ is chosen according to the Armijo-Goldstein condition (see \cite[Section 5]{K-Pikkarainen}). Since $\mu^s$ is a candidate for problem \eqref{eq:min_coeff}, the added coefficients optimization step in Algorithm \ref{alg:core} does not worsen the convergence rate. 
  \end{rem}

\begin{proof}  
Fix $n \in \N$, and let $\{\mu^n\}$ be a sequence generated by Algorithm \ref{alg:core}, which by construction is of the form \eqref{eq:alg_iterates}.
Recall that $w^n_t:=-K_t(K_t^* \rho_t^n - f_t) \in C(\olom)$ is the associated dual variable. Let $\mu^*=(\rho^*,m^*)$ in $\ext(C_{\alpha,\beta})$ be a solution to the insertion step, that is, $\mu^*$ solves \eqref{eq:alg_insertion}. The existence of such $\mu^*$ is guaranteed by Proposition \ref{prop:wolfe equiv}. 
Without loss of generality, we can assume that the algorithm does not stop at iteration $n$, that is, $\ps{\rho^*}{w^n}>1$ according to the stopping condition in Section \ref{subsec:insertion}. In particular, $\rho^*\neq 0$.  Recalling that $\ext(C_{\alpha,\beta})=\{0\} \cup \points$ (Theorem \ref{thm:extremal}), we then have $\mu^*=\mu_{\gamma^*}$ for some $\gamma^* \in \AC^2$. By the coefficients optimization step, the next iterate is of the form  
\[
\mu^{n+1} = \sum_{j=1}^{N_n} c_j^{n+1} \mu_{\gamma_{j}^n} + c^* \mu_{\gamma^*}   \,,
\] 
and the coefficients $(c_1^{n+1},\ldots, c_{N_n}^{n+1},c^*)$ solve the quadratic problem 
\begin{equation} \label{eq:proof_quadratic}
\min_{c=(c_1,\ldots,c_{N_n +1} ) \in \R_+^{N_{n}+1}} T_{\alpha,\beta} \left(  \sum_{j=1}^{N_n} c_j \mu_{\gamma_j^n} + c_{N_n+1} \mu_{\gamma^*} \right)\,. 
\end{equation}
\textit{Claim.} There exists a constant $\tilde{C}>0$ independent of $n$ such that, if $T_{\alpha,\beta}(\mu^n) \leq M_0$, then
   \begin{equation} \label{decay}
    r(\mu^{n+1}) - r(\mu^n) \leq - \tilde C \, r(\mu^n)^2 \,.
    \end{equation}

\smallskip

To prove the above claim, assume that $T_{\alpha,\beta}(\mu^n) \leq M_0$. Define $\hat \mu:=M\mu_{\gamma^*}$, with $M:=M_0 \ps{\rho_{\gamma^*}}{w^n}$.   
Since $\mu_{\gamma^*}$ solves \eqref{eq:alg_insertion} and $\ps{\rho_{\gamma^*}}{w^n}>1$, by Proposition \ref{prop:wolfe equiv} we have that $\hat \mu$ minimizes in \eqref{aux2} with respect to $w^n$. Therefore, according to \eqref{eq:dual_gap}, 
\begin{equation} \label{eq:proof_dualgap}
G(\mu^n)=J_{\alpha,\beta}(\mu^n) - \f (J_{\alpha,\beta}(\hat \mu)) - \ps{\rho^n-\hat \rho}{w^n} \,,
\end{equation}
given that $J_{\alpha,\beta}(\mu^n)\leq T_{\alpha,\beta}(\mu^n) \leq M_0 <+\infty$. For $s \in [0,1]$ set $\mu^s:=\mu^n + s (\hat \mu - \mu^n)$. By convexity of $J_{\alpha,\beta}$ (see Lemma \ref{lem:prop J}) we have $J_{\alpha,\beta}(\mu^s)<+\infty$. Hence we can apply the polarization identity at \eqref{eq:bilinear2} with respect to $w^n,\rho^n,\rho^s$ to obtain
\[
\begin{aligned}
T_{\alpha,\beta}(\mu^s) - T_{\alpha,\beta}(\mu^n) & = -s \ps{\hat \rho - \rho^n}{w^n} + \frac{s^2}{2} \norm{K^*( \rho^n - \hat \rho  )}^2_{L^2_H} + J_{\alpha,\beta}(\mu^s) - J_{\alpha,\beta}(\mu^n) \\
& \leq -s \ps{\hat \rho - \rho^n}{w^n} + \frac{s^2}{2} \norm{K^*( \rho^n - \hat \rho  )}^2_{L^2_H} + s(\f(J_{\alpha,\beta}(\hat \mu)) - J_{\alpha,\beta}(\mu^n)  ) \,,
\end{aligned}
\]
where in the second line we used convexity of $J_{\alpha,\beta}$ and the inequality $t \leq \f(t)$ for all $t \geq 0$. Note that $\mu^s$ is a competitor for \eqref{eq:proof_quadratic}, so that $T_{\alpha,\beta}(\mu^{n+1}) \leq T_{\alpha,\beta}(\mu^s)$. Recalling \eqref{eq:proof_dualgap} we then obtain
\begin{equation} \label{eq:conv_claim}
r(\mu^{n+1})-r(\mu^n)=T_{\alpha,\beta}(\mu^{n+1}) - T_{\alpha,\beta}(\mu^n) \leq - s G(\mu^n) +  \frac{s^2}{2} \norm{K^*(\hat \rho -\rho^n   )}^2_{L^2_H}\,,
\end{equation}
which holds for all $s \in [0,1]$.  Choose the stepsize $\hat{s}$ according to the Armijo-Goldstein condition (see e.g. \cite[Definition 4.1]{pieper}) as 
\begin{equation} \label{eq:armijo}
    \hat{s}:= \min \left\{  1, \frac{G(\mu^n)}{  \norm{ K^* (\hat{\rho} - \rho^n)  }^2_{L^2_H}} \right\} \,,
 \end{equation}
  with the convention that $C/0 = +\infty$ for $C>0$. If $\hat s <1$, by \eqref{eq:conv_claim} and \eqref{est primal dual} we obtain
\begin{equation} \label{eq:conv_claim:2}
r(\mu^{n+1})-r(\mu^n)  \leq - \frac12  
\frac{G(\mu^n)^2}{ \norm{ K^* (\hat{\rho} - \rho^n)  }^2_{L^2_H}}
 \leq -\frac12  \frac{r(\mu^n)^2}{ \norm{ K^* (\hat{\rho} - \rho^n)  }^2_{L^2_H}}\,.
\end{equation}
Since we are assuming $T_{\alpha,\beta}(\mu^n) \leq M_0$, by definition of $T_{\alpha,\beta}$ we deduce that $\|\rho^n\|_{\mathcal{M}(X)} \leq M_0/\alpha$. Moreover, since $\hat \mu = M\mu_{\gamma^*}$, by \ref{ass:K2} in Section \ref{sec:assumptions}, the estimate $a_{\gamma^*} \leq 1/\alpha$, and the Cauchy-Schwarz inequality yield  
\[
\begin{aligned}
\norm{\hat \rho}_{\M(X)} = a_{\gamma^*}M & \leq a_{\gamma^*}^2 M_0 \int_0^1 |\ps{\delta_{\gamma^*(t)}}{w^n_t}| \, dt \leq \frac{M_0}{\alpha^2} \int_0^1 \norm{w^n_t}_{C(\olom)} \, dt \\
& \leq \frac{M_0}{\alpha^2} \left(\int_0^1 \norm{w^n_t}_{C(\olom)}^2 \, dt\right)^2 
 \leq \frac{M_0}{\alpha^2} C \sqrt{2} \left( \frac12 \norm{K^*\rho^n - f}_{L^2_H}^2 \right)^{1/2}\\
 & \leq \frac{M_0}{\alpha^2} C \sqrt{2} \, T_{\alpha,\beta} (\mu^n)^{1/2} \leq  \frac{M_0^{3/2}}{\alpha^2} C \sqrt{2}\,,
\end{aligned}
\]
where $C>0$ is the constant in \ref{ass:K2}. ~%
Thus we can estimate
\[%
 \| K^*(\hat \rho - \rho^n)\|^2_{L^2_H} \leq 2C^2 (\|\hat \rho\|_{\mathcal{M}(X)}^2 + \|\rho^n\|_{\mathcal{M}(X)}^2) \leq \tilde{C}\,,
\]%
  with $\tilde{C}>0$ not depending on $n$. Inserting the above estimate in \eqref{eq:conv_claim:2} yields \eqref{decay} and the claim follows. Assume now $\hat s=1$, so that $ \| K^*(\hat \rho - \rho^n)\|^2_{L^2_H} \leq G(\mu^n)$. From \eqref{eq:conv_claim} and \eqref{est primal dual} we obtain
  \begin{equation} \label{eq:conv_claim:3}
  r(\mu^{n+1})-r(\mu^n)  \leq - \frac12  G(\mu^n) \leq -\frac12 r(\mu^n) \,.
  \end{equation}
  As $T_{\alpha,\beta}(\mu^n) \leq M_0$, we also have $r(\mu^n) \leq T_{\alpha,\beta}(\mu^n) \leq M_0$. Therefore we can find a constant $\tilde{C}>0$ not depending on $n$ such that $\tilde{C}r(\mu^n)^2 \leq r(\mu^n)$. Substituting the latter in \eqref{eq:conv_claim:3} yields \eqref{decay} and the proof of the claim is concluded.
  
\smallskip

Finally, we are in position to prove \eqref{eq:decaytheorem}. Since $\mu^0=0$, we have that $T_{\alpha,\beta}(\mu^0) \leq M_0$. Thus we can inductively apply \eqref{decay} and %
obtain that $T_{\alpha,\beta}(\mu^n) \leq M_0$ for all $n\in \N$. In particular, \eqref{decay} holds for every $n\in \N$ and, as a consequence, the sequence $\{r(\mu^n)\}$ is non-increasing. Setting $r_n := r(\mu^n)$, we then get
  \begin{equation*}
  \frac{1}{r_{n}} - \frac{1}{r_{0}} = \sum_{j=0}^{n-1} \frac{1}{r_{j+1}} - \frac{1}{r_{j}} = \sum_{j=0}^{n-1} \frac{r_j - r_{j+1}}{r_j r_{j+1}} \geq \sum_{j=0}^{n-1}\frac{\tilde{C} r_j^2}{r_j r_{j+1}} \geq \tilde Cn\,,
  \end{equation*}
and \eqref{eq:decaytheorem} follows.
The remaining claims follow from the weak* lower semicontinuity of $T_{\alpha,\beta}$ (Theorem \ref{thm:existence}) and estimate \eqref{lem:prop J est}, given that $\{\mu^n\}$ is a minimizing sequence for \eqref{eq:prel_main_prob}. %
\end{proof}

 \subsection{Stopping condition}\label{subsec:stop}

We prove the optimality statements in the stopping criterion for the core algorithm  anticipated in Section~\ref{subsec:insertion}. In the following $G$ denotes the primal-dual gap introduced in \eqref{eq:dual_gap}. We denote by $\mu^n$ the $n$-th iterate of Algorithm \ref{alg:core}, which is of the form \eqref{eq:alg_iterates}, and by $w^n$ the corresponding dual variable \eqref{eq:alg_dual_var}. Moreover let $\mu_{\gamma^*}$ be the atom associated to the curve $\gamma^* \in \ACe^2$ solving the insertion step \eqref{eq:hard_min}. %

\begin{lem}\label{lem:stopalgorithm}
For all $n \in \N$ we have
\begin{equation} \label{eq:stop:2}
G(\mu^n) = \Lambda(\ps{\rho_{\gamma^*}}{w^n})\,, \qquad \Lambda(t):= 
\begin{cases}
	0   & \, \text{ if } \, t \leq 1\,, \\
 \frac{M_0}{2} \left( t^2 -1\right)   & \, \text{ if } \, t>1 \,, 
\end{cases}
\end{equation}
where $M_0\geq0$ is defined at \eqref{def:M0}. 
In particular $\mu^n$ is a solution of \eqref{eq:prel_main_prob} if and only if $\ps{\rho_{\gamma^*}}{ w^n} \leq 1$.
\end{lem}

Before proving Lemma \ref{lem:stopalgorithm}, we give a quantitative version of the stopping condition of Section~\ref{subsec:insertion}.

\begin{rem} \label{rem:relax_stop}
With the same notations as above, consider the condition
 \begin{equation} \label{eq:alg:tol}
\Lambda(\ps{\rho_{\gamma^*}}{w^n})  < {\rm TOL}
 \end{equation}
where ${\rm TOL}>0$ is a fixed tolerance. Notice that 
$G(\mu^n) =\Lambda(\ps{\rho_{\gamma^*}}{w^n})$
by Lemma \ref{lem:stopalgorithm}. Thus, assuming \eqref{eq:alg:tol}, and  using \eqref{est primal dual}, we see that the functional residual defined at \eqref{def:residual} satisfies  $r(\mu^n) <{\rm TOL}$, i.e., $\mu^n$ almost minimizes \eqref{eq:prel_main_prob}, up to the tolerance. Therefore, the condition at \eqref{eq:alg:tol} can be employed as a quantitative stopping criterion for Algorithm \ref{alg:core}.
\end{rem}

\begin{proof}[Proof of Lemma \ref{lem:stopalgorithm}]
We start by computing $G(\mu^n)$ for a fixed $n \in \N$. Set $\hat \mu:=M\mu_{\gamma^*}$, where $M$ is defined as in \eqref{eq:M_choice} with $w=w^n$.   
Since $\mu_{\gamma^*} \in \ext(C_{\alpha,\beta})$ solves \eqref{eq:alg_insertion}, we have that $\hat \mu$ solves \eqref{aux2} with respect to $w^n$ (see Proposition \ref{prop:wolfe equiv}). By Theorem \ref{thm:convergence}, the sequence $\{T_{\alpha,\beta}(\mu^n)\}$ is non-increasing. Thus $J_{\alpha,\beta}(\mu^n) \leq T_{\alpha,\beta}(\mu^n) \leq T_{\alpha,\beta}(\mu^0)  = M_0<+\infty$, since $\mu^0=0$. Then \eqref{eq:dual_gap} reads
\begin{equation} \label{eq:stop:3}
G(\mu^n)= J_{\alpha,\beta}(\mu^n) - \ps{\rho^n}{w^n} -  \f(J_{\alpha,\beta}(\hat \mu))+ \ps{\hat \rho}{w^n} \,.
\end{equation}
Notice that by one-homogeneity of $J_{\alpha,\beta}$ (see Lemma \ref{lem:prop J}) and the fact that $\mu_{\gamma^*} \in \ext(C_{\alpha,\beta})$ we have $J_{\alpha,\beta}(\hat \mu)=M$. Recalling the definition of $\f$ at \eqref{def:fi}, by direct calculation we obtain
\begin{equation} \label{eq:stop:4}
-\f(J_{\alpha,\beta}(\hat \mu)) + \ps{\hat \rho}{w^n}= \Lambda( \ps{\rho_{\gamma^*}}{w^n})\,.  %
\end{equation}
We now compute the remaining terms in \eqref{eq:stop:3}. By \eqref{eq:min_coeff} and \eqref{def:new_iterate} at the step $n-1$,  we know that the coefficients $c^n=(c_1^n,\ldots,c_{N_n}^n) \in \R^{N_n}_{++}$ of $\mu^{n}$ solve   the minimization problem
   \begin{equation} \label{eq:stop:minprob}
   \min_{c=(c_1,\ldots,c_{N_n}) \in \R^{N_n}_{++}} T_{\alpha,\beta}  \left( \sum_{j=1}^{N_n} c_j \mu_{\gamma_j^n}\right)\,.
   \end{equation}
 Since $\gamma_j^n$ in \eqref{eq:alg_iterates} is continuous, we have that $(t \mapsto \delta_{\gamma_j^n(t)}) \in \pcurves$ and thus $t \mapsto K_t^* \delta_{\gamma_j^n(t)}$ belongs to $\Ltwo$, by Lemma \ref{lem:prop K}.  In view of \eqref{eq:alg_iterates}, definition of $T_{\alpha,\beta}$ and Lemma \ref{lem:additivity}, we can expand the expression at \eqref{eq:stop:minprob}, differentiate with respect to each component of $c$, and recall that $c^n$ is optimal in \eqref{eq:stop:minprob}, to obtain
\begin{equation}\label{eq:F_minus1}
\begin{aligned}
 0 &  = \sum_{j=1}^{N_n}  c_j^n a_{\gamma_i^n} a_{\gamma_j^n} 
   \int_0^1 \ps{K_t^* \delta_{\gamma_i^n(t)}}{K_t^* \delta_{\gamma_j^n(t)}}_{H_t} \, dt  
  -  a_{\gamma_i^n} \int_0^1 
   \ps{K_t^* \delta_{\gamma_i^n(t)}}{f_t}_{H_t} \, dt 
     + 1\\
     & =  \ps{K^* \rho_{\gamma^n_i}}{K^* \rho^n}_{L^2_H}  
  -   \ps{K^*  \rho_{\gamma^n_i}}{f}_{L^2_H} 
     + 1 = 1 - \langle  \rho_{\gamma_i^n}, w^n\rangle\,,
\end{aligned}
\end{equation}
which holds for all $i=1,\ldots,N_n$. By Lemma \ref{lem:additivity} and linearity of $\ps{\cdot}{\cdot}$, we obtain the identity $J_{\alpha,\beta}(\mu^n) = \ps{\rho^n}{w^n}$. 
The latter, together with \eqref{eq:stop:3} and \eqref{eq:stop:4}, yields \eqref{eq:stop:2}. For the remaining part of the statement, notice that by \eqref{eq:stop:2} we have $G(\mu^n)=0$ if and only if $\ps{\rho_{\gamma^*}}{w^n} \leq 1$. Therefore, the thesis follows by Lemma \ref{lem:primal dual}. 
\end{proof}

  \subsection{Time-discrete version} \label{sec:time-discrete}

The minimization problem \eqref{eq:prel_main_prob} presented in Section \ref{sec:OT_regularization} is posed for
  time-continuous measurements $f \in \Ltwo$, whose discrepancy to the reconstructed curve $t \mapsto \rho_t$ is modelled by the fidelity term $\F$ at \eqref{prel:fidelity}.   
  In real-world applications, however, the measured data is time-discrete, i.e.,  we can assume that measurements are taken at times $0 = t_0 < t_1 < \ldots < t_T = 1$, with $T \in \N$ fixed. Hence the data is of the form  $f = (f_{t_0}, f_{t_1},\ldots, f_{t_T})$, where $f_{t_i} \in H_{t_i}$.  
		For this reason, and with the additional goal of lowering the computational cost, we decided to present numerical experiments (Section \ref{sec:numerics}) where the time-continuous fidelity term $\F$ is replaced by a discrete counterpart. In  Section \ref{subsub:discr_fid} we show how to modify the mathematical framework discussed so far, in order to deal with the time-discrete case. Consequently, it is immediate to adapt Algorithm \ref{alg:core} to the resulting time-discrete functional, as discussed in Section \ref{subsub:discr_alg}. We remark that all the results up to this point will hold, in a slightly modified version, also for the discrete setting discussed below.
		
\subsubsection{Time-discrete framework}	\label{subsub:discr_fid}	
 We replace  problem  \eqref{eq:prel_main_prob} with
\begin{equation}
  \min_{\mu \in \M } \, T^{\mathcal{D}}_{\alpha,\beta}(\mu)\,, \quad T^{\mathcal{D}}_{\alpha,\beta}(\mu) := \F^{\mathcal{D}}(\mu) + J_{\alpha,\beta}(\mu)\,,
  \label{eq:discrete_minimization} \tag{$\mathcal{P}_{\rm discr}$}
\end{equation}
where the time-discrete fidelity term $\F^{\mathcal{D}} \colon \M \to [0,+\infty]$ is defined by
\[
\F^{\mathcal{D}}(\mu):=
\begin{cases}
\displaystyle \frac{1}{2 (T+1)}\sum_{i=0}^T 
  \norm{K_{t_i}^* \rho_{t_i} - f_{t_i}}_{H_{t_i}}^2  &  \, \text{ if } \rho=dt \otimes \rho_t \, , \, \, (t \mapsto \rho_t) \in \curves \,,	\\
  + \infty    & \, \text{ otherwise.}
\end{cases}
\]
Here, $H_{t_i}$ are real Hilbert spaces, and the given data vector $f=(f_{t_0},\ldots,f_{t_T})$ satisfies $f_{t_i} \in H_{t_i}$ for all $i=0,\ldots,T$. The forward operators $K_{t_i}^* \colon \M(\olom) \to H_i$ are assumed to be linear continuous and weak*-to-weak continuous, for each $i=0,\ldots,T$. The minimization problem \eqref{eq:discrete_minimization} is well-posed by the direct method of calculus of variations: indeed $T^{\mathcal{D}}_{\alpha,\beta}$ is proper, $J_{\alpha,\beta}$ is weak* lower semicontinuous and coercive in $\M$ (Lemma \ref{lem:prop J}), and $\F^{\mathcal{D}}$ is lower semicontinuous with respect to the convergence in \eqref{topology}. In particular a solution $\mu^*= (\rho^*, m^*)$ to \eqref{eq:discrete_minimization} 
will satisfy $\rho^*=dt \otimes \rho_t^*$ with $(t \mapsto \rho_t^*  )\in \pcurves$. We now define the other quantities which are needed to formulate the discrete counterpart of the theory developed so far. 
For a given curve of measures $(t \rightarrow \tilde \rho_t) \in \curves$,
the corresponding dual variable \eqref{eq:dual_variable}
is redefined to be
\begin{equation} \label{eq:discr_variable}
  w_{t_i} := -K_{t_i}( K_{t_i}^* \tilde \rho_{t_i} - f_{t_i}) \in C(\olom)\,,
\end{equation}
for each $i=0,\ldots,T$. 
Consequently, we redefine the associated scalar product \eqref{eq:scalarproduct_alg} to
\[%
  \ps{\rho}{w}_{\mathcal{D}} :=  
  \begin{cases}
    \displaystyle \frac{1}{T+1}\sum_{i=0}^T  \langle \rho_{t_i}, w_{t_i} \rangle_{\M(\olom),C(\olom)}  & \,\, \text{if } \rho=dt \otimes \rho_t\, , \,\, (t \mapsto \rho_t) \in \curves\,, \\ %
  -\infty &\,\, \text{otherwise.}
 \end{cases}
  \label{eq:discrete_scalarproduct}
\]%
It is straightforward to check that all the results in Section \ref{sec:OT_regularization} hold with $T_{\alpha,\beta}$ and $\ps{\cdot}{\cdot}$ replaced by $T^{\mathcal{D}}_{\alpha,\beta}$ and $\ps{\cdot}{\cdot}_{\mathcal{D}}$ respectively, with the obvious modifications. In particular, the problem
\begin{equation} \label{eq:min_discrete}
\min_{\mu \in \ext(C_{\alpha,\beta})} \,  -\ps{\rho}{w}_{\mathcal{D}}
\end{equation}
admits a solution $\mu^*$, where $C_{\alpha,\beta}:=\{ J_{\alpha,\beta}(\mu)\leq 1\}$. One can perform a similar computation to the one at \eqref{eq:diff_computation:11} to obtain equivalence between \eqref{eq:min_discrete} and
\begin{equation}
  \label{eq:hard_min_discrete}
  \min_{\gamma \in \ACe^2}
  -\left< \rho_\gamma ,w \right>_{\mathcal{D}}  = \min_{\gamma \in \AC^2} \left\{ 0, 
  -\frac{a_{\gamma}}{T+1}\sum_{i=0}^T w_{t_i}(\gamma(t_i)) \right\}\,.
\end{equation}

\begin{rem} \label{rem:piecewise}
Assume additionally that $\olom$ is convex. Then 
a solution $\gamma^* \in \ACe^2$ to problem \eqref{eq:hard_min_discrete} is either $\gamma^*=\gamma_\infty$, or $\gamma^* \in \AC^2$ with $\gamma^*$ linear in each interval $[t_i,t_{i+1}]$. %
Indeed, given any curve $\gamma \in \AC^2$, denote by $\tilde{\gamma}$ the piecewise linear version of $\gamma$ sampled at $t_i$ for $i=0,\ldots,T$. Then $\left< \rho_\gamma ,w \right>_{\mathcal{D}} \leq \left< \rho_{\tilde{\gamma}} ,w \right>_{\mathcal{D}}$, due to the inequality %
$\int_0^1 |\dot{\tilde{\gamma}}(t)|^2 \, dt \leq \int_0^1 |\dot \gamma(t)|^2 \, dt$. 
\end{rem}

\subsubsection{Adaption of Algorithm \ref{alg:core} to the time-discrete setting} \label{subsub:discr_alg}
The core algorithm in Section \ref{subsec:alg_descr} is  readily adaptable to the task of minimizing  \eqref{eq:discrete_minimization}. To this end, let $f_{t_i} \in H_{t_i}$ be a given datum and $\alpha,\beta>0$ be fixed parameters for \eqref{eq:discrete_minimization}. 
Assume that $\mu^n$ is the current sparse iterate, of the form \eqref{eq:alg_iterates}.  
The dual variable associated to $\mu^n$ is defined, according to \eqref{eq:discr_variable}, by
$w_{t_i}^n := -K_{t_i}( K_{t_i}^* \rho_{t_i}^n - f_{t_i})$.
Similarly to Section \ref{subsec:insertion}, 
the insertion step in the time-discrete version consists 
in finding a curve $\gamma^* \in \ACe^2$ which solves 
\eqref{eq:hard_min_discrete} with respect to $w^n_i$. 
Such curve defines 
a new atom $\mu_{\gamma^*}$ according to Definition \ref{def:atom}. Adapting the proofs of Lemmas \ref{lem:primal dual}, \ref{lem:stopalgorithm} to the discrete setting, one deduces the following stopping condition:
\begin{itemize}
  \item if $\left<\rho_{\gamma^*}, w^n\right>_{\mathcal{D}} \leq 1$, 
     then $\mu^n$ is solution to \eqref{eq:discrete_minimization}. The algorithm outputs $\mu^n$ and stops,  
  \item if $\left< \rho_{\gamma^*}, w^n\right>_{\mathcal{D}} > 1$, then $\mu^n$ is not a solution to \eqref{eq:discrete_minimization} and $\gamma^* \in \AC^2$.
    The found atom $\mu_{\gamma^*}$ is inserted in the $n$-th iterate $\mu^n$ and the algorithm continues.  
\end{itemize}

Set $\gamma_{N_n+1}^n := \gamma^*$. 
The time-discrete version of the coefficients optimization step discussed in Section~\ref{subsec:optimization} consists in solving 
\begin{equation} \label{eq:coeff_discrete}
  \min_{ (c_1,c_2,\ldots,c_{N_n+1}) \in
  \R_{+}^{N_n+1}} T^{\mathcal{D}}_{\alpha,\beta} \left(
  \sum_{j=1}^{N_n+1} c_j \mu_{\gamma_j^{n}} \right)\,.
 \end{equation}
By proceeding as in Section \ref{subsec:quadratic}, one can check that \eqref{eq:coeff_discrete} is equivalent to a quadratic program of the form \eqref{eq:alg_quadratic_program}, where the matrix $\Gamma \in \R^{(N_n +1) \times (N_n+1)}$ and the vector $b \in \R^{N_n+1}$ are given by
\[
  \Gamma_{j,k}:= \frac{a_{\gamma_j} a_{\gamma_k}}{T+1}
  \sum_{i=0}^{T} \ps{K_{t_i}^*\delta_{\gamma_j(t_i)}}{K_{t_i}^*\delta_{\gamma_k(t_i)}}_{H_{t_i}}\,, 
  \quad b_j:= 1- \frac{a_{\gamma_j}}{T+1}\sum_{i=0}^{T} \ps{K_{t_i}^*\delta_{\gamma_j}(t_i)}{f_{t_i}}_{H_{t_i}} \,.
\]
In view of Remark \ref{rem:piecewise} and of the above construction, we note that the iterates of the discrete algorithms are of the form \eqref{eq:alg_iterates} with $\gamma_j^n \in \AC^2$ piecewise linear. Finally, we remark that the time-discrete algorithm obtained with the above modifications has the same sublinear rate of convergence stated in Theorem \ref{thm:convergence}. %

\section{The algorithm: numerical implementation} \label{sec:numerical_implementation}

This aim of this section is twofold. First, in Section \ref{sec:insstepheu} we describe how to approach the 
minimization of the insertion step problem \eqref{eq:hard_min} by means of gradient descent strategies. This analysis is performed under additional assumptions on the operators $K_t \colon H_t \to C(\olom)$, which, loosely speaking, require that $K_t$ map into the space $C^{1,1}(\olom)$ of differentiable functions with bounded Lipschitz gradient. The strategies proposed will result in the \emph{multistart gradient descent} Subroutine \ref{alg:multistart} (Section \ref{subsec:subroutine}), which, given a dual variable $w$, outputs a set of stationary points for  \eqref{eq:hard_min}. Then, in Section \ref{subsec:acceleration}, we present two acceleration steps that can be added to Algorithm \ref{alg:core} to, in principle, enhance its performance. %
	The first acceleration strategy, called the \emph{multiple insertion step}, proceeds by adding all the outputs of Subroutine \ref{alg:multistart} to the current iterate. The second strategy, termed \emph{sliding step}, consists in locally descending the target functional $T_{\alpha,\beta}$ at \eqref{eq:prel_main_prob} in a neighbourhood of the curves composing the current iterate, while keeping the coefficients fixed. These strategies are finally added to Algorithm \ref{alg:core}. The outcome is Algorithm \ref{alg:full}, presented in Section \ref{subsec:full}, which we name  \emph{dynamic generalized conditional gradient} (DGCG).

\subsection{Insertion step implementation}\label{sec:insstepheu}
We aim at minimizing the linearized problem \eqref{eq:hard_min} in the insertion step, which is of the form 
\begin{equation} \label{eq:sec5_ins}
\min_{\gamma \in H^1([0,1];\olom)} \left\{
    0, F(\gamma)
  \right\}\,, \quad F(\gamma):=-a_{\gamma} \int_0^1 w_t(\gamma(t)) \, dt \,,
\end{equation}
where the dual variable $t \mapsto w_t$ is defined for a.e.~$t \in (0,1)$ by $w_t:=-K_t(K_t^* \tilde{\rho}_t - f_t) \in C(\olom)$, for some curve $(t \mapsto \tilde{\rho}_t) \in \curves$ and data $f \in \Ltwo$ fixed. In \eqref{eq:sec5_ins} we also identified $\AC^2$ with $H^1([0,1];\olom)$, and employed the notations at \eqref{ext_meas}. We remind the reader that, although  \eqref{eq:sec5_ins} admits solutions (see Section \ref{subsec:iterates}), in practice they may be difficult to compute numerically (Remark~\ref{rem:non-linear}). %
Therefore we turn our attention at finding stationary points for the functional $F$ at \eqref{eq:sec5_ins}, relying
on gradient descent methods. %
To make this approach feasible, we require additional assumptions on the operators $K_t^*$ (see Assumption \ref{def:additional K} below), which allow to extend the functional $F$ to the Hilbert space $H^1:=H^1([0,1];\R^d)$ and make it Fr\'echet differentiable, without altering the value of the minimum at \eqref{eq:sec5_ins}. In particular this allows for a gradient descent procedure to be well defined (Section~ \ref{subsec:main_descent}). With this at hand, in Section \ref{subsec:multi} we define a descent operator $\mathcal{G}_w$ associated to $F$, which, for a starting curve $\gamma \in H^1$, outputs either a stationary point of $F$ or the infinite length curve $\gamma_\infty$.  The starting curves for $\mathcal{G}_w$ are of two types:
\begin{itemize}
  \item random starts which are guided by the values of the dual variable
    $w$ (Section \ref{subsec:random}),
  \item crossovers between known stationary points (Section \ref{subsec:crossover}).
\end{itemize}
We then propose a minimization strategy for \eqref{eq:sec5_ins}, which is implemented in the \textit{multistart gradient descent algorithm} contained in  
Subroutine \ref{alg:multistart} (Section \ref{subsec:subroutine}). Such algorithm inputs a set of curves and a dual variable $w$, and returns a set $\mathcal{S} \subset \AC^2$, which is either empty or contains stationary curves for $F$ at \eqref{eq:sec5_ins}.  

\subsubsection{Minimization strategy} \label{subsec:main_descent}

The additional assumptions required on $K_t^*$ are as follows.

 \begin{assumption}\label{def:additional K}
For a.e.~$t \in (0,1)$ the linear continuous operator $K_t^* \colon C^{1,1}(\olom)^* \to H_t$ satisfies 
   \begin{enumerate}[label=\textnormal{(F\arabic*)}]
      \item $K_t^*$ is weak*-to-weak continuous, with pre-adjoint denoted by $K_t: H_t \to C^{1,1}(\overline \Omega)$,\label{ass:F1}
    \item $\norm{K^*_t} \leq C$ for some constant $C>0$ not depending on $t$, \label{ass:F2}
    \item the map $t \mapsto K_t^*\rho$ is strongly measurable for every fixed $\rho \in \M(\olom)$, \label{ass:F3}
  \item  there exists a closed convex set $E \Subset \Omega$ such
      that $
     \supp (K_t f),
     \supp (\nabla(K_t f))\subset E$ for all $f \in H_t$ and a.e.~$t \in (0,1)$, where $\nabla$ denotes the spatial gradient. \label{ass:F4}
  \end{enumerate}
  \end{assumption}

Notice that \ref{ass:F1}-\ref{ass:F3} imply \ref{ass:K1}-\ref{ass:K3} of Section \ref{sec:assumptions}, due to the embedding $\M(\olom) \hookrightarrow C^{1,1}(\olom)^*$. Also note that \ref{ass:F4} has no counterpart in the assumptions of Section \ref{sec:assumptions}, and is only assumed for computational convenience, as discussed below. 
The problem of solving \eqref{eq:sec5_ins} under \ref{ass:F1}-\ref{ass:F4} is addressed in Appendix \ref{sec:computingextremal}; here we summarize the main results obtained. First, we extend $F$ to the Hilbert space $H^1$, by setting $w_t(x):=0$ for all $x\in \R^d \smallsetminus \olom$ and a.e.~$t \in (0,1)$. Due to \ref{ass:F4} we have that $w_t \in C^{1,1}(\R^d)$. We then show that $F$ is continuously Fr\'echet differentiable on $H^1$, with locally Lipschitz derivative (Proposition \ref{prop:gateaux}). 
     Denote by $D_\gamma F \in (H^1)^*$ the Fr\'echet derivative of $F$ at $\gamma$ (see \eqref{gateaux F} for the explicit computation) and introduce the set of stationary points of $F$ with non-zero energy 
\begin{equation*}
\mathfrak{S}_F :=\{\gamma \in H^1 \, \colon \, F(\gamma) \neq 0 \,, \,\, D_\gamma F = 0 \}\,.%
\end{equation*}
     In Proposition \ref{cor:optimality} we prove that all the points in $\mathfrak{S}_F$ satisfy  
$\gamma([0,1]) \subset \Om$. As a consequence, $\mathfrak{S}_F$ contains all the solutions to the insertion problem \eqref{eq:sec5_ins} whenever
\begin{equation} \label{eq:stop_inf}
\min_{\gamma \in H^1([0,1];\olom)} \{0,F(\gamma)\}<0\,.
\end{equation}
\begin{rem} \label{rem:min_ins}
The case \eqref{eq:stop_inf} is the only one of interest: indeed Algorithm \ref{alg:core} stops if \eqref{eq:stop_inf} is not satisfied, with the current iterate being a solution to the target problem \eqref{eq:prel_main_prob} (see Section \ref{subsec:insertion}). In  Proposition \ref{prop:test} we prove that  \eqref{eq:stop_inf} is equivalent to
\begin{equation} \label{def:Pw}
P(w):=\int_0^1 \max_{x \in \olom} w_t(x)\, dt > 0 \,.
\end{equation}
As $P(w)$ is easily computable, condition \eqref{def:Pw} provides an implementable test for \eqref{eq:stop_inf}. Moreover, note that \eqref{eq:stop_inf} is satisfied when $F$ is computed from the dual variable $w^n$ associated to the \mbox{$n$-th} iterate $\mu^n=\sum_{j=1}^{N_n}c_j^n \delta_{\gamma_j^n}$ of Algorithm \ref{alg:core}, and $n \geq 1$. This is because $F(\gamma_j^n)=-1$ for all $j=1,\ldots,N_n$, as shown in \eqref{eq:F_minus1}.  
	
\end{rem}\label{rem:GD}
If \eqref{eq:stop_inf} is satisfied, we aim at computing points in $\mathfrak{S}_F$ by gradient descent. To this end, we say that $\{\gamma^n\}$ in $H^1$ is a \textit{descent sequence} if
\begin{equation} \label{def:sec5_descent}
F(\gamma^0)<0 \,, \,\,\,\,\, \gamma^{n+1} = \gamma^n - \delta_n D_{\gamma^n} F\,,  \,\,\,\, \text{ for all }\,\, n \in \N \cup\{0\}\,,
\end{equation}
where $D_{\gamma^n}F \in (H^1)^*$ is identified with its Riesz representative in $H^1$, and $\{\delta_n\}$ is a stepsize chosen according to the Armijo-Goldstein or Backtracking-Armijo rules.
In Theorem \ref{thm:gradient_descent} we prove that, if $\{\gamma^n\}$ is a descent sequence, there exists at least a subsequence such that $\gamma^{n_k} \to \gamma^*$ strongly in $H^1$; moreover, any such accumulation point $\gamma^*$ belongs to $\mathfrak{S}_F$. To summarize, descent sequences in the sense of \eqref{def:sec5_descent} enable us to compute points in $\mathfrak{S}_F$, which are candidate solutions to \eqref{eq:sec5_ins} whenever \eqref{eq:stop_inf} holds. %

\subsubsection{Descent operator} \label{subsec:multi}

Fix a dual variable $w$ and consider the functional $F$ defined as in \eqref{eq:sec5_ins}. The descent operator $\mathcal{G}_w \colon H^1([0,1];\olom) \to H^1 \cup \{\gamma_\infty\}$ associated to $w$ is defined by
\begin{equation} \label{def:desc_G}
\mathcal{G}_w(\gamma) :=
\begin{cases}
\gamma^*  & \,\, \text{ if }  \,\, F(\gamma)<0 \,,\\
\gamma_\infty    & \,\, \text{ otherwise}\,,
\end{cases}
\end{equation}
where $\gamma^*$ is an accumulation point for the descent sequence $\{\gamma^n\}$ defined according to \eqref{def:sec5_descent} with starting point $\gamma^0:=\gamma$. In view of the discussion in the previous section, we know that the image of $\mathcal{G}_w$ is contained in $\mathfrak{S}_F\cup \{\gamma_\infty\}$.
Note that, if $P(w)>0$, in principle the image of $\mathcal{G}_w$ will contain at least one stationary point $\gamma^* \in \mathfrak{S}_F$, as in this case \eqref{eq:stop_inf} holds (Remark \ref{rem:min_ins}). However, in simulations, we can only compute $\mathcal{G}_w$ on some finite family of curves $\{ \gamma_i\}_{i \in I}$, 
which we name the \textit{starts}. Thus, in general, we have no guarantee of finding points in $\mathfrak{S}_F$, even if \eqref{eq:stop_inf} holds. The situation improves if $w=w^n$ is the dual variable associated to the $n$-th iterate $\mu^n=\sum_{j=1}^{N_n}c_j^n \delta_{\gamma_j^n}$ of Algorithm \ref{alg:core} and $n \geq 1$. In this case, setting $\mathcal{A}:=\{\gamma_{1}^n, \ldots, \gamma_{N_n}^n\}$, we have that $\mathcal{G}_{w^n}(\mathcal{A}) \subset \mathfrak{S}_F$ by definition \eqref{def:desc_G} and Remark \ref{rem:min_ins}. Therefore, by including the curves in $\mathcal{A}$ in the set of considered starts, we are guaranteed of obtaining at least one point in $\mathfrak{S}_F$.

\subsubsection{Random starts} \label{subsec:random}

We now describe how we randomly generate starting points \(\gamma\) in $H^1([0,1];\olom)$ for the descent operator $\mathcal{G}$ at \eqref{def:desc_G}. %
We start by selecting time nodes $0=t_0 < t_1< \ldots < t_T=1$ drawn uniformly
in $[0,1]$ (if operating in the time-discrete setting, we sample instead with
a uniform probability on the finite set of sampling times on which the fidelity term of
\eqref{eq:discrete_minimization} is defined). 
We choose the value of a random start $\gamma$ at time $t_i$ 
seeking to maximize the dual variable $w_{t_i}$.  
To achieve this, let $Q: \R \rightarrow \R_+$ be non-decreasing and monotonous, %
and define the probability measure on $E$ 
\begin{equation*}
  \mathbb{P}_{w_{t_i}}(A) := \frac{\int_A Q(w_{t_i}(x)) \,dx }{\int_{\Omega} Q(w_{t_i}(x))\, dx }\,,
\end{equation*}
for $A \subset E$ Borel measurable and \(E \Subset \Omega\) introduced in 
Assumption \ref{def:additional K}.
We then draw samples from $\mathbb{P}_{w_{t_i}}$ with the rejection-sampling 
algorithm, and assign those samples to $\gamma(t_i)$. Using that $E$ is a
convex set, the random curve $\gamma \in \AC^2$ is obtained by interpolating
linearly the values $\gamma(t_i) \in E$.  This procedure is executed by the
routine \texttt{sample}, which inputs a dual variable $w$ and outputs a
randomly generated curve $\gamma \in \AC^2$.

\subsubsection{Crossovers between stationary points}\label{subsec:crossover}
It is heuristically observed that 
stationary curves have a tendency to share common ``routes'', as for example seen in the reconstructions presented in  
Figures~\ref{fig:3:reconstruction_comparison} and
\ref{fig:3:reconstruction_comparison2}. %
It is then a reasonable ansatz to combine curves which are sharing routes, in order to increase the likelihood for the newly obtained crossovers to share common routes with the sought global  
minimizers of $F$. Such crossovers will then be employed as starts for the descent operator $G$ at \eqref{def:desc_G}. Formally, the crossover is achieved as follows. We fix small parameters $\e>0$ and  $0<\delta<1$. For  $\gamma_1, \gamma_2 \in \AC^2$ define the set 
\begin{equation*}
R_\e(\gamma_1,\gamma_2):=  \left\{ t\in [0,1] \, \colon \, 
    |\gamma_1(t) -\gamma_2(t)| < \e \right\}\,.
\end{equation*}
We say that $\gamma_1$ and $\gamma_2$ share routes if $R_\e(\gamma_1,\gamma_2) \neq \emptyset$. If $R_\e(\gamma_1,\gamma_2) = \emptyset$, we perform no operations on $\gamma_1$ and $\gamma_2$. If instead $R_\e(\gamma_1,\gamma_2) \neq \emptyset$, first notice that $R_\e(\gamma_1,\gamma_2)$ is relatively open in $[0,1]$.  Denote by $I$ any of its connected components. Then $I$ is an interval with endpoints $t^-$ and $t^+$, satisfying $0\leq t^- < t^+\leq 1$. The crossovers of  $\gamma_1$ and $\gamma_2$ in $I$ are the two curves $\gamma_3, \gamma_4 \in \AC^2$ defined by 
\begin{equation*}
\gamma_3(t):=
  \begin{cases}
     \gamma_1(t), &\  t \in [\,0, \hat t- \delta \tilde t \,] \\
    \gamma_2(t), &\  t \in [\,\hat  t + \delta \tilde{t}, 1]
  \end{cases} \,\,,
  \qquad
   \gamma_4(t):=
  \begin{cases}
    \gamma_2(t), &\  t \in [\,0, \hat t - \delta \tilde{t}\,] \\
   \gamma_1(t), &\  t \in [\,\hat t + \delta \tilde{t}, 1\,] 
  \end{cases}
\end{equation*}
and linearly interpolated in $(\hat t - \delta \tilde t, \hat t + \delta \tilde t)$, where $\hat t :=(t^++t^-)/2$ and $\tilde t :=(t^+-t^-)/2$, i.e.,
\begin{equation*}
\frac{t - (\hat t -\delta \tilde t)}{2\delta \tilde t} \gamma_2(\hat t + \delta \tilde{t}) - \frac{t - (\hat t + \delta \tilde t)}{2\delta \tilde t} \gamma_1(\hat t - \delta \tilde{t})\,, \quad t \in (\hat t - \delta \tilde t, \hat t + \delta \tilde t)
\end{equation*}
is, for instance, the result of the linear interpolation of $\gamma_3$ in $(\hat t - \delta \tilde t, \hat t + \delta \tilde t)$.
We construct the crossovers of $\gamma_1$ and $\gamma_2$ in each connected component of $R_\e(\gamma_1,\gamma_2)$ obtaining $2M$ new curves, with $M$ being the number of connected components of $R_\e(\gamma_1,\gamma_2)$. The described procedure is executed by the routine \texttt{crossover}, which inputs two curves $\gamma_1, \gamma_2 \in \AC^2$ and outputs a set of curves in $\AC^2$, possibly empty.

\subsubsection{Multistart gradient descent algorithm} \label{subsec:subroutine}

In Subroutine~\ref{alg:multistart} we sketch the proposed method to 
search for a minimizer of \eqref{eq:sec5_ins}: this is implemented in the function \texttt{MultistartGD}, which inputs a set of curves $\mathcal{A}$ and a dual variable $w$, and outputs a (possibly empty) set $\mathcal{S}$ of stationary points for $F$ defined at \eqref{eq:sec5_ins}. We now describe how to interpret such subroutine in the context of Algorithm~\ref{alg:core}, and how it can be used to replace the insertion step operation at line 4. 

Given the $n$-th iteration $\mu^n=\sum_{j=1}^{N_n}c_{j}^n \delta_{\gamma_j^n}$
of Algorithm \ref{alg:core}, we define $\mathcal{A}:=\{\gamma_{1}^{n}, \ldots,
\gamma_{N_n}^n\}$ if $n \geq 1$ and $\mathcal{A}:=\emptyset$ if $n=0$. The dual
variable $w^n$ is as in \eqref{eq:alg_dual_var}. We initialize to empty the
sets $\mathcal{S}$ and $\mathcal{O}$ of known stationary and crossover points
respectively. The condition at line 2 of Subroutine~\ref{alg:multistart} checks
if we are at the $0$-th iteration and $P(w^0)\leq 0$. In case this is
satisfied, then $0$ is the minimum of \eqref{eq:sec5_ins}, and no stationary
point is returned: indeed in this situation Algorithm \ref{alg:core} stops,
with $\mu^0=0$ being the minimum of \eqref{eq:prel_main_prob} (see Section
\ref{subsec:main_descent}). Otherwise, the set $\mathcal{A}$ (possibly empty)
is inserted in $\mathcal{O}$. Then $N_{\rm max}$ initializations of the
multistart gradient descent are performed, where a starting point $\gamma$ is
either chosen from the crossover set $\mathcal{O}$, if the latter is non empty,
or sampled at random by the function \texttt{sample} described in Section
\ref{subsec:random}. We then descend $\gamma$, obtaining the new curve
$\gamma^*:=\mathcal{G}_{w^n}(\gamma)$, where $\mathcal{G}_{w^n}$ is defined at
\eqref{def:desc_G}. If the outputted point $\gamma^*$ does not belong to the
set of known stationary points $\mathcal{S}$, and $\gamma^* \neq
\gamma_{\infty}$, then we first compute the crossovers of $\gamma^*$ with all
the elements of $\mathcal{S}$, and afterwards insert it in $\mathcal{S}$. After
$N_{\rm max}$ iterations, the set $\mathcal{S}$ is returned. Notice that
$\mathcal{S}$ could be empty only if $\mathcal{A}=\emptyset$, i.e., if
\texttt{MultistartGD} is called at the first iteration of Algorithm
\ref{alg:core} (see Section \ref{subsec:multi}).

The modified insertion step, that is, line 4 of Algorithm \ref{alg:core}, reads as follows. First we set $\mathcal{A}=\boldsymbol{\gamma}$ and compute $\mathcal{S}:=\texttt{MultistartGD}(\mathcal{A},w^n)$. If $\mathcal{S}=\emptyset$, the algorithm stops and returns the current iterate $\mu^n$. Otherwise, the element in $\mathcal{S}$ with minimal energy with respect to $F$ is chosen as candidate minimizer, and inserted as $\gamma^*$ in line 4.  
\renewcommand{\algorithmcfname}{Subroutine}
\setcounter{algocf}{0}
  \SetKwInput{KwInput}{Input}
  \SetKwInput{KwOutput}{Function}
  \SetCommentSty{mycommfont}
  \begin{algorithm}[h!]
    \caption{Multistart gradient descent for the insertion step} \label{alg:multistart}
  \DontPrintSemicolon \smallskip
  \KwOutput{\texttt{MultistartGD}
  }
  \KwInput{Set of curves $\mathcal{A}$, dual variable $w_t \in C^{1,1}(\olom)$ with $\supp w_t \subset \Om$
  }
  $\mathcal{S} := \emptyset $\,, \, $\mathcal{O} := \emptyset $ 
  \tcp{Sets of known stationary and crossover points} 
  \If{ $\mathcal{A} == \emptyset$  \bf{and} $P(w)\leq 0$}
	{\Return  $\mathcal{S}$
	}
  $\mathcal{O} \leftarrow \mathcal{A}$ \\	
  \For{$k=1,\ldots,N_{\rm max}$}
  {
  \tcc{Restart from a random curve or a crossover one} 
   \If{ $\mathcal{O} == \emptyset$ }
    {
      $\gamma \leftarrow \texttt{sample}(w)$ 
    }
    \Else
    {
      $\gamma \leftarrow$ {\bf get from} $\mathcal{O}$\,,\,\,
      $\mathcal{O} \leftarrow \mathcal{O} \smallsetminus  \{\gamma\}$
    }
    \tcc{Descend, crossover, and incorporate to the stationary points set}
    $\gamma^* \leftarrow \mathcal{G}_w(\gamma)$ \\
    \If{ $\gamma^* \not \in \mathcal{S}$ {\bf and} $\gamma^* \neq \gamma_{\infty}$}
    {
      \For{$\eta \in \mathcal{S}$}
      {
        $\mathcal{O} \leftarrow \mathcal{O} \cup \texttt{crossover}(\gamma^*, \eta)$ 
        \tcp{crossover with all known stationary points}
      }
      $\mathcal{S} \leftarrow \mathcal{S} \cup \{\gamma^*\}$ 
    }
  }
  \Return $\mathcal{S}$
  \end{algorithm}
\renewcommand{\algorithmcfname}{Algorithm}

\subsection{Acceleration strategies} \label{subsec:acceleration}

In this section we describe two acceleration strategies that can be incorporated in Algorithm \ref{alg:core}. 

\subsubsection{Multiple insertion step}
\label{subsec:multiple_insertion}

This is an extension of the insertion step for Algorithm \ref{alg:core} described in Section \ref{subsec:insertion}. 
Precisely, the multiple insertion step consists in 
inserting into the current iterate $\mu^n$ all the atoms associated to the stationary points in the set $\mathcal{S}$ produced by Subroutine \ref{alg:multistart} with respect to the dual variable $w^n$. %
This procedure is motivated by the following observations. First, computationally speaking,
    the coefficients optimization step described in Section \ref{subsec:optimization} is cheap and fast. %
Second, it is observed that stationary points  are good candidates for the insertion step in the GCG method presented in \cite{pieper} to solve \eqref{intro:BLASSO}. 
This observation can be extended similarly to our framework, noticing that the addition of multiple stationary points is encouraging the iterates to concentrate around every atom of the ground-truth measure $\mu^\dagger$ and consequently the algorithm could need fewer iterations to efficiently locate the support of $\mu^\dagger$.

\subsubsection{Sliding step} \label{subsec:sliding}
The sliding step proposed in this paper is a natural extension of the one introduced for BLASSO in \cite{K-Pikkarainen}
and further analyzed in \cite{denoyelle2019sliding}. Precisely, 
given a current iterate $\mu = \sum_{j=1}^{N} c_j \mu_{\gamma_j}$ of Algorithm \ref{alg:core}, we fix the weights  $ \boldsymbol{c} = (c_1 ,\ldots, c_{N})$ and define the functional $T_{\alpha,\beta,\boldsymbol{c}} : (H^1)^{N} \rightarrow \R$ as
\begin{equation} \label{def:sliding}
T_{\alpha,\beta,\boldsymbol{c}}(\eta_1,\ldots, \eta_{N})  := T_{\alpha,\beta} \left( \sum_{j=1}^{N} c_j \mu_{\eta_j} \right) \quad \text{for all} \quad (\eta_1,\ldots, \eta_{N}) \in \left(H^1\right)^{N}\,.
\end{equation}
We then perform additional gradient descent steps in the space $(H^1)^{N}$ for the functional $T_{\alpha,\beta,\boldsymbol{c}}$, starting from the tuple of curves $(\gamma_1,\ldots,\gamma_N)$ contained in the current iterate $\mu$. Formally, this procedure is possible: in Proposition~\ref{prop:grad_flow} we prove that $T_{\alpha,\beta,\boldsymbol{c}}$ is continuously Fr\'echet differentiable in $(H^1)^{N}$ under Assumption \ref{def:additional K}, with derivative given by \eqref{prop:grad_flow:gat}.
This step can be intertwined with the coefficients optimization
one, by alternating between modifying the position of the current curves, and
optimizing their associated weights.

\subsection{Full algorithm} \label{subsec:full}
By including Subroutine \ref{alg:multistart} and the proposed acceleration steps of Section~\ref{subsec:acceleration} into Algorithm~\ref{alg:core}, 
we obtain Algorithm~\ref{alg:full},
which we name the \emph{dynamic generalized conditional gradient} (DGCG).  %

\subsubsection{Algorithm summary}
We employ the  notations of Section \ref{subsec:alg_summary}. In particular, Algorithm~\ref{alg:full} generates, as iterates, tuples of coefficients $\boldsymbol{c} = (c_1, \ldots, c_{N})$ with $c_j>0$, and curves $\boldsymbol{\gamma} = (\gamma_1, \ldots, \gamma_{N})$ with $\gamma_j \in \AC^2$.
We now summarize the main steps of Algorithm \ref{alg:full}. The first 3 lines are unaltered from Algorithm \ref{alg:core}, and they deal with tuples initializations and assembly of the measure iterate $\mu^n$. The multiple insertion step is carried out by Subroutine \ref{alg:multistart}, via the function \texttt{MultistartGD} at line 4, which is called with arguments $\boldsymbol{\gamma}$ and $w^n$. The output is a set of curves $\mathcal{S}$ which contains stationary points for the functional $F$ at \eqref{eq:sec5_ins} with respect to the dual variable $w^n$. If $\mathcal{S}=\emptyset$, the algorithm stops and outputs the current iterate $\mu^n$.  As observed in Section~\ref{subsec:multi}, this can only happen at the first iteration of the algorithm. Otherwise, in line $7$, the function \texttt{order} is employed to input $\mathcal{S}$ and output a tuple of curves, obtained by ordering the elements of $\mathcal{S}$ increasingly with respect to their value of $F$. As anticipated in Section \ref{sec:insstepheu}, the first element of $\texttt{order}(\mathcal{S})$, named $\gamma_{N_n + 1}^*$, is considered to be the best available candidate solution to the insertion step problem \eqref{eq:sec5_ins}, and is used in the stopping condition at lines $8$, $9$. Such stopping condition has been used similarly in Algorithm \ref{alg:core}, with the difference that in Algorithm \ref{alg:full} the curve $\gamma_{N_n + 1}^*$ is not necessarily the global minimum of $F$, but, in general, just a stationary point. We further discuss such stopping criterion in Section \ref{subsec:stop_alg2} below. After, the found stationary points are inserted in the current iterate, and the algorithm alternates between the  coefficients optimization step (Section~\ref{subsec:optimization}) and the sliding step (Section~\ref{subsec:sliding}). Such operations are executed from line $11$ to $16$ in Algorithm \ref{alg:full}, for $K_{\rm max}$ times. %

\subsubsection{Stopping condition}  \label{subsec:stop_alg2}

The stopping condition for Algorithm \ref{alg:full} is implemented in lines $8$, $9$: 
when $\gamma_{N_n + 1}^*$ satisfies $\ps{\rho_{\gamma_{N_n+1}^*}}{ w^n} \leq 1$, the algorithm stops and outputs the current iterate $\mu^n$.
Due to the definition of $F$, such condition is equivalent to say that the Subroutine~\ref{alg:multistart} has not been
  able to find any curve $\gamma^* \in \AC^2$ satisfying 
$\left<\rho_{\gamma^*}, w^n\right> > 1 $.
The main difference when compared to the stopping condition for Algorithm \ref{alg:core} (Section \ref{subsec:insertion}) is that, in Algorithm \ref{alg:full}, the curve  $\gamma_{N_n + 1}^*$ is generally not a global minimum of $F$. %
 As a consequence, Lemma \ref{lem:stopalgorithm} does not hold, and the condition $\langle\rho_{\gamma_{N_n + 1}^*}, w^n\rangle \leq 1$ is not equivalent to the minimality of the current iterate $\mu^n$ for \eqref{eq:prel_main_prob}. The evident drawback is that %
 Algorithm~\ref{alg:full} could stop even if the current iterate does not solve \eqref{eq:prel_main_prob}. However, it is at least possible to say that, if Algorithm \ref{alg:full} continues after line 9, then  the current iterate does not solve \eqref{eq:prel_main_prob}. Thus, in this situation, the correct decision is taken. 
We remark that, even if the stopping condition for Algorithm \ref{alg:full} does not ensure the minimality of the output, from a practical standpoint, if Subroutine \ref{alg:multistart} is employed with a high number of restarts the reconstruction is satisfactory. We also point out that the condition at line 8 can be replaced by the quantitative condition defined at \eqref{eq:alg:tol} in Remark \ref{rem:relax_stop}, with some predefined tolerance.

\subsubsection{Convergence and numerical residual} \label{subsec:conv_alg2}

In the following we will say that Algorithm \ref{alg:full} converges if $\texttt{MultistartGD}(\boldsymbol{\gamma},w^0) = \emptyset$, or if some iterate satisfies the stopping condition at line $8$. %
In case of convergence, we will denote by 
 $\mu^{N}$ the output value of Algorithm~\ref{alg:full}. We remind the reader that, due to the discussion in Section \ref{subsec:stop_alg2}, $\mu^N$ is considered to be an approximate solution for the minimization problem~\eqref{eq:prel_main_prob}. In order to analyze the convergence rate for Algorithm \ref{alg:full} numerically  we define the numerical residual as
\begin{equation}
  \label{eq:functional_distance}
  \tilde{r}(\mu^n) := T_{\alpha,\beta}(\mu^n) -
  T_{\alpha,\beta}(\mu^N)\,, \qquad n = 0,\ldots,N-1\,,
\end{equation}
with $\mu^n$ each of the computed intermediate iterates.
We further define the numerical primal-dual gap $\tilde{G}(\mu^n)$, 
which we compute by employing \eqref{eq:stop:2} with $\gamma^*=\gamma^*_{N_n+1}$.

\SetKwInput{KwInput}{Input}
\SetKwInput{KwOutput}{Output}
\SetCommentSty{mycommfont}
\begin{algorithm}[h!]
  \caption{Dynamic generalized conditional gradient (DGCG)} \label{alg:full}
\DontPrintSemicolon
\KwInput{
  Data $f \in \Ltwo$, parameters $\alpha,\beta > 0$, forward operators $K^*_t: \mathcal{M}(\overline \Omega) \mapsto H_t$ \smallskip \smallskip
} 
$\boldsymbol{c} \leftarrow ()$, \,\,
$\boldsymbol{\gamma} \leftarrow ()$ \\
\For{$n=0,1,\ldots$}
{ $N_n \leftarrow |\boldsymbol{\gamma} |$ , \,\,
$\mu^n \leftarrow \sum_{j=1}^{N_n} c_j \mu_{\gamma_{j}} $ ,\,\,
 $w_t^n \leftarrow  -K_t (K_t^* \rho^n_t - f_t)$ \\
  \tcc{Multiple insertion step employing Subroutine \ref{alg:multistart}}
$\mathcal{S} \leftarrow  \texttt{MultistartGD}(\boldsymbol{\gamma},w^n)$\\
\tcc{Stopping conditions}
 \If{$\mathcal{S}== \emptyset $} 
 	{\Return $\mu^n$
 	}
$(\gamma^*_{N_n + 1}, \ldots, \gamma^*_{N_n +  |\mathcal{S}|} )  \leftarrow \texttt{order}(\mathcal{S})$\\
  \If{$\ps{\rho_{\gamma_{N_n+1}^*}}{ w^n} \ \leq 1 $}
  {
    \Return $\mu^n$
  }
  $\tilde N_n \leftarrow N_n +  |\mathcal{S}|$, \,\,
   $ \boldsymbol{\gamma} = (\gamma_{1}, \ldots, \gamma_{\tilde N_n} ) \leftarrow (\boldsymbol{\gamma},\gamma_{N_n+1}^* , \ldots,\gamma^*_{N_n + |\mathcal{S}|} )$\\
  \For{$k=1,\ldots,K_{\rm max}$}
  {
  \tcc{Coefficients optimization step}
  $\boldsymbol{c} \leftarrow
    \argmin_{(c_1,\ldots,c_{\tilde N_n})\in \R_{+}^{\tilde N_n}} T_{\alpha,\beta} \left(
  \sum_{j=1}^{\tilde N_n}c_j \mu_{\gamma_j}\right)$ \\
 $\tilde N_n \leftarrow \# \{j: c_j > 0\}$\,,\,\,
  $( \boldsymbol{\gamma}, \boldsymbol{c}) \leftarrow \texttt{delete\_zero\_weighted}( \boldsymbol{\gamma}, \boldsymbol{c})$ \\
  
  \tcc{Sliding step}
$T_{\alpha,\beta,\boldsymbol{c}}(\eta_1,\ldots, \eta_{\tilde N_n}) := T_{\alpha,\beta}\left(\sum_{j = 1}^{\tilde N_n} c_j \mu_{\eta_j}\right) $\\
  $(\eta_1^*,\ldots,\eta_{\tilde N_n}^*) \leftarrow \texttt{descend} \quad T_{\alpha,\beta,\boldsymbol{c}} \quad
  \texttt{from}\quad  (\gamma_1,\ldots, \gamma_{\tilde N_n})$  \\
  $\boldsymbol{\gamma} \leftarrow (\eta_1^*,\ldots,\eta_{\tilde N_n}^*)$
  }
 $N_{n + 1} \leftarrow \tilde N_n$,\,\,
   $\mu^{n+1} \leftarrow \sum_{j=1}^{N_{n + 1}} c_j \mu_{\gamma_j}$
}
\end{algorithm}

\section{Numerical implementation and experiments}
\label{sec:numerics}
In this section we present the produced numerical experiments.  
In order to lower the computational cost, we chose to implement Algorithm \ref{alg:full} for the minimization of the time-discrete functional \eqref{eq:discrete_minimization} discussed in Section~\ref{sec:time-discrete}. The adaptation of Algorithm \ref{alg:full} to such setting is easily obtainable as a corollary of the discussion in Section \ref{subsub:discr_alg}.   
The simulations were produced
by a Python code that
is openly available at \url{https://github.com/panchoop/DGCG_algorithm/}.
For all the simulations we employ the following:
\begin{itemize}
  \item the considered domain is $\Omega := (0,1)\times (0,1) \subset \R^2$,
  \item the number of time samples is fixed to $T = 50$, with $t_i := i/T$ for $i =0,\ldots,T$,
  \item the data spaces $H_{t_i}$ and forward operators $K_{t_i}^* \colon \M(\olom) \to H_{t_i}$ are as in Section \ref{subsec:Fourier} below, and model a Fourier transform with time-dependent spatial undersampling. A specific choice of sampling pattern will be made in each experiment,
  \item %
     in each experiment problem \eqref{eq:discrete_minimization} is considered for specific
    choices of regularization parameters $\alpha,\beta>0$ and data
    $f=(f_{t_0},\ldots , f_{t_T})$ with $f_{t_i} \in H_{t_i}$,
     \item convergence for Algorithm \ref{alg:full} is intended as in Section
             \ref{subsec:conv_alg2}. The number of restarts $N_{\rm max}$ for
             Subroutine \ref{alg:multistart} is stated in each experiment.
             Also, we employ the quantitative  stopping condition for Algorithm
             \ref{alg:full} described in Remark \ref{rem:relax_stop}, with
             tolerance ${\rm TOL}:=10^{-10}$.
     \item the crossover parameters 
           described in Section \ref{subsec:crossover} are chosen as \(\varepsilon=0.05\) 
           and \(\delta = 1/T\). The random starts presented in
           Section \ref{subsec:random} are implemented using the function \(Q(x) = \exp(\max(x + 0.05,0)) - 1\). These parameter choices are purely heuristical and there is no reason to believe that they are optimal.
\end{itemize}
The remainder of the section is organized as follows. In Section \ref{subsec:data_noise_visual} we first introduce the measurement spaces and Fourier-type forward operators employed in the experiments. After, we explain how the synthetic data is generated in the noiseless case, and then detail on the noise model we consider. Subsequently, we show how the data and the obtained reconstructions can be visualized, by means of the so-called backprojections and intensities.
We then pass to the actual experiments in Section \ref{subsec:experiments}, detailing three of them. The first experiment (Section \ref{subsec:Ex1}), which is basic in nature, serves the purpose of illustrating how the measurements are constructed and how the data can be visualized. We then showcase the reach of the proposed regularization and algorithm in the second example (Section \ref{example:complex}). There, we consider more complex data with various levels of added noise. In particular, we show that the proposed dynamic setting is capable of reconstructing severely spatially undersampled data. The final experiment (Section \ref{subsec:Ex3}) illustrates a particular behaviour of our model when reconstructing sparse measures whose underling curves cross, as discussed in Remark \ref{rem:crossing}.  %
	We point out that in all the experiments presented, the algorithm converged faster, indeed linearly, than the sublinear rate predicted by Theorem \ref{thm:convergence}.  We conclude the section with a few general observations on the model and algorithm proposed.

\subsection{Measurements, data, noise model and visualization}
\label{subsec:data_noise_visual}

\subsubsection{Measurements}
\label{subsec:Fourier}
The measurements employed in the experiments are given by the time-discrete version, in the sense of Section \ref{sec:time-discrete}, of the spatially undersampled Fourier measurements introduced in Section \ref{subsec:Fourier_discrete_sample}. Such measurements will be suitably cut-off, in order to avoid boundary conditions when dealing with the insertion step (Sections \ref{subsec:insertion}, \ref{subsec:multiple_insertion}) and the sliding step (Section \ref{subsec:sliding}). %
Precisely, at each time instant $t_i$ we sample $n_i \in \N$ frequencies, encoded in the given vector $S_i=(S_{i,1},\ldots,S_{i,n_i}) \in (\R^2)^{n_i}$, for all $i \in \{0,\ldots,T\}$. The sampling spaces $H_{t_i}$ are defined as  the realification of $\C^{n_{i}}$, equipped with the inner product 
$ \ps{u}{v}_{H_{t_i}} := {\rm Re}\ps{u}{v}_{\C^{n_i}}/n_i$, where ${\rm Re}$ denotes the real part of a complex number.  According to \eqref{eq:app_Fourier3} we define the cut-off Fourier kernels $\psi_{t_i} \colon \R^2 \mapsto \C^{n_i}$ by 
\[%
\psi_{t_i}(x) := \left( \exp(-2\pi i x \cdot S_{i,k})\chi(x_1)\chi(x_2) \right)_{k=1}^{n_i},
\]%
where the map 
$\chi:(0,1) \mapsto (0,1)$ is defined by 
\begin{equation*}
\chi(z) :=\begin{cases}
    10 (\nicefrac{z}{0.1})^3 - 15 (\nicefrac{z}{0.1})^4 +
    6 (\nicefrac{z}{0.1})^5 & \quad \text{for } \, z \in [0, 0.1), \\
    1 & \quad \text{for } \, z \in [0.1, 0.9], \\
    10 (\nicefrac{(1-z)}{0.1})^3 - 15 (\nicefrac{(1-z)}{0.1})^4 +
    6 (\nicefrac{(1-z)}{0.1})^5 & \quad \text{for } \,z  \in (0.9, 1].
  \end{cases}
\end{equation*}
Notice that $\chi$ is twice differentiable, strictly increasing in $[0,0.1]$, 
strictly decreasing in $[0.9,1]$, and it satisfies $\chi(0) = \chi(1) = 0$, $\chi(z) = 1$ 
for all $z \in [0.1,0.9]$. Following \eqref{eq:app_Fourier2}, the cut-off undersampled Fourier transform and its pre-adjoint are given by the linear continuous operators $K_{t_i}^* \colon \mathcal{M}(\overline \Omega) \to H_{t_i}$ and $K_{t_i} \colon H_{t_i} \to C(\overline \Omega)$ defined by
\begin{equation} \label{def:forward_numerics}
  K_{t_i}^*(\rho) := \int_{\R^2} \psi_{t_i}(x) \, d \rho(x)\,, \qquad 
    K_{t_i}(h) := \left(\  x \mapsto  \left< \psi_{t_i}(x), h \right>_{H_{t_i}} \right)\,,
\end{equation}
for all $\rho \in \mathcal{M}(\olom)$, $h \in H_t$, where  $\rho$ is extended to zero outside of $\olom$, and the first integral is intended component-wise. 

\subsubsection{Data} \label{subsubsec:data}
For all the experiments the ground-truth consists of a sparse measure of the form
\begin{equation} \label{eq:ground_truth}
  \mu^\dagger =(\rho^\dagger,m^\dagger):= \sum_{j=1}^N  c^\dagger_j \mu_{\gamma^\dagger_j}\,,
\end{equation}
for some $N \in \N$, $c^\dagger_j>0$, $\gamma^\dagger_j \in \AC^2$, 
where we follow the notations at \eqref{ext_meas}. 
Given a ground-truth $\mu^\dagger$, %
the respective noiseless  data %
is constructed by $f_{t_i} := K_{t_i}^* \rho^\dagger_{t_i} \in H_{t_i}$, for $i =0,\ldots,T$. 

\subsubsection{Noise model} \label{subsubsec:noise}
Let $U_{i,k}, V_{i,k}$, for $i = 0,\ldots,T$
and $k = 1,\ldots, n_i$, be the realization of 
two jointly independent families
of standard 1-dimensional Gaussian random variables, with which we define the noise vector $\nu$ by
\begin{equation}\label{eq:noise}
  \nu_{i,k} := U_{i,k} + \mathrm{i} V_{i,k} \in \C, \quad 
  \nu_{t_i} := (\nu_{i,k})_{k=1}^{n_i} \in H_{t_i}, \quad 
  \nu := (\nu_0, \nu_1, \ldots,\nu_T)\,.
\end{equation}
Given some data $f=(f_{t_0},\ldots,f_{t_T})$, with $f_{t_i} \in H_{t_i}$, when $\nu \neq 0$, the corresponding noisy data $f^\e$ with noise level $\e \geq 0$ is taken as 
\[%
  f^\e := f + \e \ 
  \sqrt{
    \frac{ \sum_{i=0}^T \norm{f_{t_i}}_{H_{t_i}}^2}{ \sum_{i=0}^T \norm{\nu_{t_i}}_{H_{t_i}}^2}
  } \ \nu\,. 
  \label{eq:noisy_data}
\]%

\subsubsection{Visualization via backprojection} \label{subsubsec:back}
In general it is not illustrative to directly visualize the data. The proposed way to gain some insight on the data structure is by means of backprojections: given $f = (f_{t_0},\ldots, f_{t_T})$, with $f_{t_i} \in H_{t_i}$,
we call backprojection the map $w_{t_i}^0:= K_{t_i}f_{t_i}  \in C(\olom)$.
Note that $w_{t_i}^0$ corresponds to the dual 
variable at the first iteration of Algorithm \ref{alg:full}. 
As $\Omega = (0,1)^2$, such functions can be plotted at each time sample, 
allowing us to display the data.

\subsubsection{Reconstruction's intensities} \label{subsubsec:intensity}

Given a sparse measure 
$\mu := \sum_{j=1}^N c_j \mu_{\gamma_j}$, with $N \in \N$, $c_j >0$, $\gamma_j \in \AC^2$,  
the
intensity associated to the atom $\mu_{\gamma_j}$ is defined by
$I_{j} := c_j a_{\gamma_j}$. %
The quantity $I_j$ measures the intensity at time $t_i$ of the signal for a single source, as $K^*_{t_i}(c_j \rho_{\gamma_j(t_i)})=I_j K^*_{t_i}(\delta_{\gamma_j(t_i)})$.
    Therefore, 
when presenting reconstructions 
and comparing them to the given ground-truth,
we will use the intensity $I_j$ of each atom
instead of its associated weight $c_j$.

\subsection{Numerical experiments}  \label{subsec:experiments}

\subsubsection{Experiment 1 - Single atom with constant speed}
\label{subsec:Ex1}

We start with a basic example that serves at illustrating how to observe the data, the obtained reconstructions, 
and their respective distortions
due to the employed regularization.
We use constant-in-time forward Fourier-type measurements, 
with frequencies sampled from an Archimedean spiral:
for each sampling time $i \in \{0,\ldots,50\}$, we consider the same frequencies vector $S_{i} \in (\R^2)^{n_i}$ with  $n_i := 20$ and $S_{i,k}$ lying on a spiral for $k = 1,\ldots,20$ (see Figure~\ref{fig:1:data}). Thus, the corresponding forward operators defined by \eqref{def:forward_numerics} are constant in time, i.e., $K_{t_i}^*=K^*$.  
The employed ground-truth $\mu^\dagger$ is composed of a single atom
with intensity
$I^\dagger = 1$, 
and respective curve
$\gamma^\dagger(t) := (0.2,0.2) + t (0.6,0.6)$.
Accordingly, we have $\mu^\dagger = c^\dagger \mu_{\gamma^\dagger}$ with $c^\dagger = 1/a_{\gamma^\dagger}$.
We consider the case of noiseless data  $f_{t_i} := K^* \rho^\dagger_{t_i}$. 
The corresponding backprojection $w_{t_i}^0:=Kf_{t_i}$ can be visualized in Figure~\ref{fig:1:data}
for some selected time samples.
For the proposed data $f$ we solve the minimization problem 
\eqref{eq:discrete_minimization} employing Algorithm~\ref{alg:full}.
The obtained reconstructions are presented in Figure~\ref{fig:1:reconstruction_comparison},
where we display the results for two different parameter choices $\alpha,\beta$.
Given the simplicity of the considered example, we employed 
Subroutine~\ref{alg:multistart} with only 5 restarts, that is, $N_{\rm max}=5$.  
In the respective cases of parameters  $\alpha=\beta = 0.1$ and $\alpha=\beta = 0.4$, Algorithm \ref{alg:full} converged in $2$ and $1$ iterations, and had an execution time of $2$ and $1$ minutes (the employed CPU was an Apple M1 8 Core 3.2 GHz, running native arm64 Python 3.9.7).
\begin{figure}[b!]
  \begin{subfigure}[t]{0.22\textwidth}
      \includegraphics{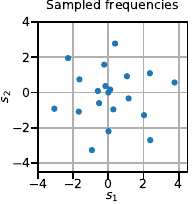}%
  \end{subfigure}%
  \begin{subfigure}[t]{0.775\textwidth}
    \includegraphics{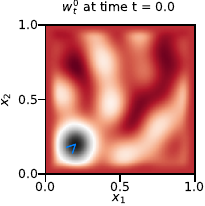}%
    \includegraphics{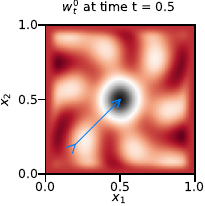}%
    \includegraphics{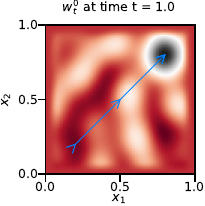}%
    \raisebox{0.14in}{%
      \includegraphics{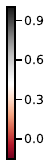}%
    }%
  \end{subfigure}
  \caption{Time-constant frequency samples $\{S_{i,k}\}_{k}$ and backprojections
    $w_{t_i}^0$ for data $f_{t_i} := K^*\rho^\dagger_{t_i}$
    at times $t_0 = 0,\ t_{25} = 0.5,\ t_{50} = 1$, with 
 the ground-truth curve $\gamma^\dagger$ superimposed in blue color.}
  \label{fig:1:data}
\end{figure}
\begin{figure}[t!]
  \begin{subfigure}[c]{0.3\textwidth}
      \includegraphics{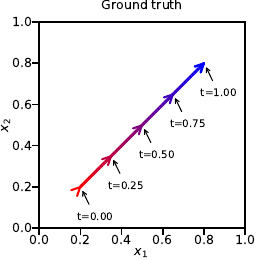}%
    \caption{Considered ground-truth. \\ \ }
    \label{fig:1:ground_truth}
  \end{subfigure}%
  \begin{minipage}[c]{0.7\textwidth}
  \begin{subfigure}[c]{\textwidth}
    \includegraphics{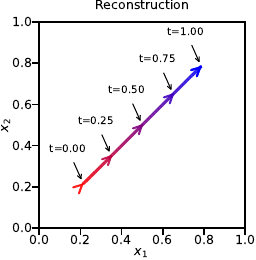}%
    \includegraphics{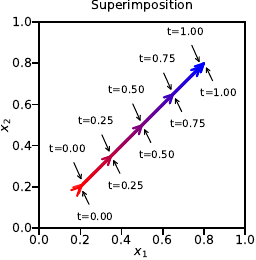}%
    \includegraphics{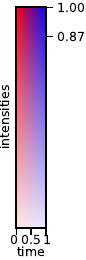}%
    \caption{Computed reconstruction with parameters 
    $(\alpha,\beta) = (0.1,0.1)$.\\ \ }
    \label{fig:1:recons0.1}
  \end{subfigure}
  \\
  \hspace{0.3\textwidth}%
  \begin{subfigure}[b]{\textwidth}
    \includegraphics{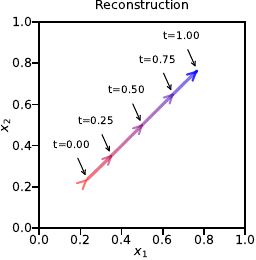}%
    \includegraphics{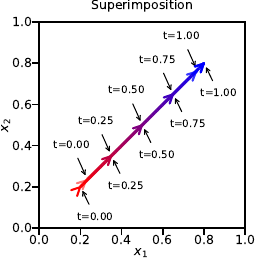}%
    \includegraphics{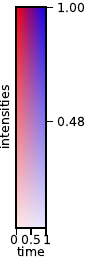}%
    \caption{Computed reconstruction with parameters
    $ (\alpha,\beta) = (0.4, 0,4)$.}
    \label{fig:1:recons}
  \end{subfigure}
\end{minipage}
  \caption{Reconstruction results for Experiment 1. 
  We use color to indicate position
  in time and transparency to indicate intensity of the respective atom.
  From left to right, we plot the employed ground-truth, 
  the obtained reconstruction with the specified parameters, and
then the superimposition of ground-truth  and reconstruction.}
  \label{fig:1:reconstruction_comparison}
\end{figure}

As common for Tikhonov regularization methods, 
the obtained reconstructions differ from the ground-truth due to the effect of regularization. %
Specifically, in Figure~\ref{fig:1:recons}, we notice a stronger effect of the regularization with parameters $\alpha=\beta=0.4$ at the endpoints of the reconstructed curve. Such phenomenon is expected: as argued in Section \ref{sec:insstepheu}, each curve found by Subroutine \ref{alg:multistart} belongs to the set of stationary points $\mathfrak{S}_F$; due to the optimality conditions proven in Proposition~\ref{cor:optimality}, any of such curves has zero initial and final speed, i.e., $\dot \gamma(0) = \dot \gamma(1) = 0$. The mentioned constraint is achieved with a slower transition for larger speed penalizations $\beta$.
We can quantify the discrepancy 
between the ground-truth curve $\gamma^\dagger$ and a reconstructed curve $\overline \gamma$ with respect to the $L^2$ norm by computing  
$D(\gamma^\dagger, \overline \gamma) := \norm{\gamma^{\dagger} - \overline \gamma}_{L^2}/\norm{\gamma^{\dagger}}_{L^2}$.
We obtain that $D(\gamma^\dagger, \overline \gamma) = 0.00515$ for
$\alpha=\beta = 0.1$  and
$D(\gamma^\dagger, \overline \gamma) = 0.017$ for $\alpha=\beta = 0.4$.
The reconstructed intensities for the parameter choices $\alpha=\beta = 0.1$ and $\alpha=\beta = 0.4$  are $87\%$ and $48\%$ of the 
ground truth's intensity respectively, as observed in Figure~\ref{fig:1:reconstruction_comparison}.

\subsubsection{Experiment 2 - Complex example with time varying measurements}
\label{example:complex}
The following example is given to showcase the full strength of the proposed
regularization and algorithm. The frequencies are sampled over lines through the origin of $\R^2$, which are rotating in time. Specifically, let $\Theta \in \N$ be a bound on the number of lines, $h > 0$ a fixed spacing between measured frequencies on a given line, and consider frequencies $S_i \in (\R^2)^{n_i}$ defined by  
\begin{equation} \label{eq:rot_sampl}
  S_{i,k} := 
  \begin{pmatrix}
    \cos( \theta_i )  & \sin(\theta_i) \\
    -\sin(\theta_i) & \cos(\theta_i)
  \end{pmatrix} 
  \begin{pmatrix}
    h (k-(n_i+1)/2) \\
    0
  \end{pmatrix},
  \quad i \in \{0,\ldots,50\},\ k \in \{1,\ldots,n_i\}\,.
\end{equation}
Here the matrix represents a rotation of angle $\theta_i$, with $\theta_i := \frac{i}{\Theta}\pi$. For such frequencies, we consider the associated forward operators $K_{t_i}^*$ as in \eqref{def:forward_numerics}.

In the presented experiment the parameters are chosen as $\Theta = 4$, $h=1$ and  $n_i = 15$ for all $i = 0,\ldots,50$: in other words, we sample along 4 different lines which rotate at each time-sample;
see Figure~\ref{fig:3:backprojection} for two examples of them.
The ground-truth is a measure $\mu^\dagger$ composed of 3 atoms, whose associated curves are as in Figure~\ref{fig:3:ground_truth}. 
Note that $\mu^\dagger$ displays:
different non-constant speeds, 
a contact point between two underlying curves,
a strong kink, and intensities equal to $1$.
We point out that equal intensities are considered for the sole purpose of easing graph visualization: the reconstruction quality is not
affected by different intensity choices. The noiseless data is defined by $f_{t_i}:=K_{t_i}^* \rho_{t_i}^\dagger$. Following the noisy model described in Section~\ref{subsubsec:noise}, 
we also consider data $f^{0.2}$ and $f^{0.6}$ with
added 20\% and 60\% of relative noise, respectively. 
For the noisy data we employed the same realization of the randomly generated noise vector 
$\nu$ defined in \eqref{eq:noise}.
In Figure~\ref{fig:3:backprojection} 
we present the backprojections for the noiseless and noisy data at two different times.
We can observe that at each time-sample
the backprojected data exhibits a line-constant 
behavior, as explained in Remark~\ref{rem:rotating_line_measurements} below.
\begin{figure}[t!]%
  \begin{subfigure}[t]{\textwidth}%
    \center
    \includegraphics{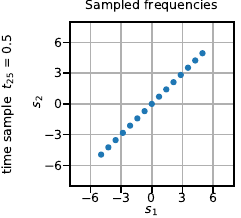}%
    \hspace{0.4in}%
    \includegraphics{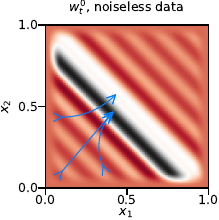}%
    \includegraphics{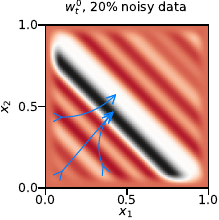}%
    \raisebox{0.14in}{%
      \includegraphics{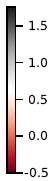}%
    }%
  \end{subfigure}%
  \\
  \begin{subfigure}[t]{\textwidth}%
    \center
    \includegraphics{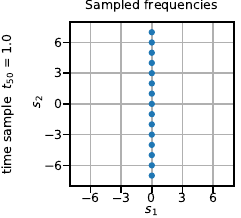}%
    \hspace{0.4in}%
    \includegraphics{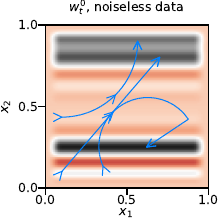}%
    \includegraphics{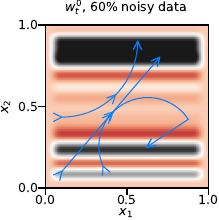}%
    \raisebox{0.14in}{%
    \includegraphics{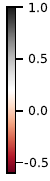}%
    }%
  \end{subfigure}%
  \caption{Sampled frequencies and backprojected data for  
  Experiment 2 at times $t_{25} = 0.5$, $t_{50} = 1$. 
  The backprojected data is displayed for noiseless data and for  
  20\% and 60\% of relative noise, i.e., $f^{0.2}$ and $f^{0.6}$ respectively, with superimposed ground-truth's curves in blue color.}
  \label{fig:3:backprojection}
\end{figure}

We apply Algorithm~\ref{alg:full} to obtain reconstructions
for the cases of noiseless and 20\% of added noise data 
 for parameters $\alpha=\beta=0.1$ (see Figure~\ref{fig:3:reconstruction_comparison});
in Figure~\ref{fig:3:reconstruction_comparison2},
we present the obtained reconstructions for the case of added 60\% noise, 
where we further show the regularization effects of employing larger 
$\alpha,\beta$ values, namely $\alpha=\beta=0.1$ and $\alpha=\beta=0.3$.

\begin{figure}[t!]
  \begin{subfigure}[c]{0.3\textwidth}
      \includegraphics{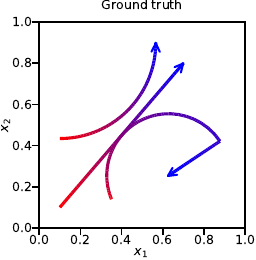}%
    \caption{Considered ground-truth. \\ \ }
    \label{fig:3:ground_truth}
  \end{subfigure}%
  \begin{minipage}[c]{0.7\textwidth}
  \begin{subfigure}[c]{\textwidth}
    \includegraphics{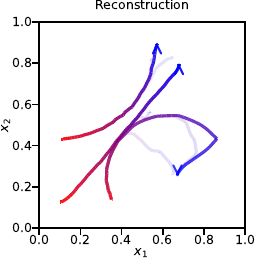}%
    \includegraphics{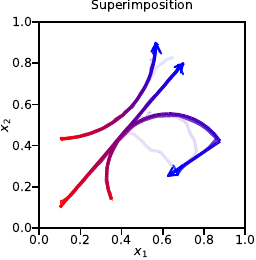}%
    \includegraphics{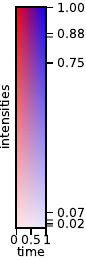}%
    \caption{Reconstruction with noiseless data.}
    \label{fig:3:reconsN00}
  \end{subfigure}
  \\[0.1in]
  \hspace{0.3\textwidth}%
  \begin{subfigure}[b]{\textwidth}
    \includegraphics{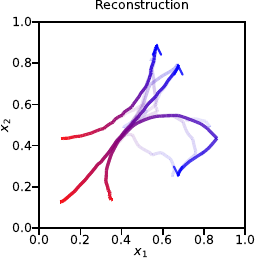}%
    \includegraphics{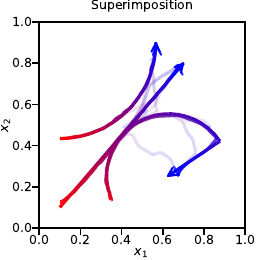}%
    \includegraphics{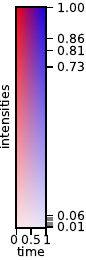}%
    \caption{Reconstruction with 20\% of relative noise.}
    \label{fig:3:reconsN20}
  \end{subfigure}
\end{minipage}
\caption{Reconstruction results for Experiment 2, with noiseless data, 20\% of
  relative noise data and parameters $ \alpha=\beta = 0.1$. The intensities of the reconstructed atoms 
are represented by the rightmost black and grey tick-lines in the colorbars.}
  \label{fig:3:reconstruction_comparison}
\end{figure}

In the noiseless case, presented in Figure~\ref{fig:3:reconsN00},
we can observe an accurate reconstruction, with some low
intensity artifacts that for most of the time share 
paths with the higher intensity atoms.
In Figure~\ref{fig:3:reconsN20}, where we add 20\% of noise to the data,  
we notice a surge of low intensity artifacts, but nonetheless, we see that 
the obtained solution is close, in the sense of measures, 
to the original ground-truth.
In Figure~\ref{fig:3:reconstruction_comparison2},  
for the case of 60\% added noise, we can notice that by increasing the regularization parameters
  the quality of the obtained reconstruction increases, displaying small regularization-induced
distortions. The examples in Figures~\ref{fig:3:reconstruction_comparison}, \ref{fig:3:reconstruction_comparison2}
demonstrate the power of the proposed regularization,
given its reconstruction accuracy when 
simultaneously employing highly ill-posed forward measurements, 
as pointed out in Remark~\ref{rem:rotating_line_measurements} below, together
with strong noise.

\begin{figure}[t!]
  \begin{subfigure}[c]{0.3\textwidth}
      \includegraphics{figures_ex3_recons_N00_ground_truth.pdf}%
    \caption{Considered ground-truth. \\ \ }
    \label{fig:3:ground_truth_v2}
  \end{subfigure}%
  \begin{minipage}[c]{0.7\textwidth}
  \begin{subfigure}[c]{\textwidth}
    \includegraphics{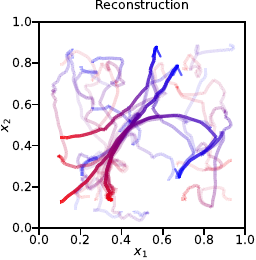}%
    \includegraphics{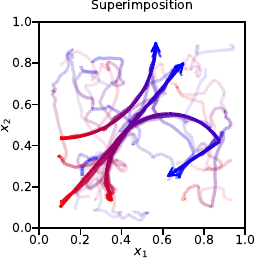}%
    \includegraphics{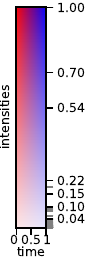}%
    \caption{Reconstruction with parameter choice $\alpha=\beta = 0.1$.}
    \label{fig:3:reconsN40}
  \end{subfigure}
  \\[0.1in]
  \hspace{0.3\textwidth}%
  \begin{subfigure}[b]{\textwidth}
    \includegraphics{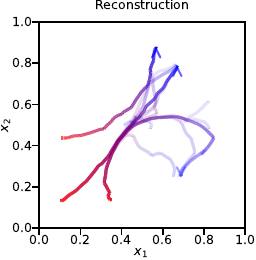}%
    \includegraphics{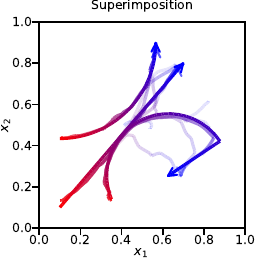}%
    \includegraphics{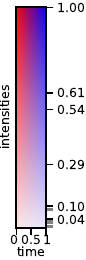}%
    \caption{Reconstruction with parameter choice $ \alpha=\beta = 0.3$.}
    \label{fig:3:reconsN40reg}
  \end{subfigure}
\end{minipage}
\caption{Reconstruction results for Experiment 2, with 60\% of relative noise. Algorithm~\ref{alg:full} is applied to the same data, with regularization parameter choices $\alpha=\beta=0.1$ and $\alpha=\beta=0.3$.}
  \label{fig:3:reconstruction_comparison2}
\end{figure}

We finalize this example by presenting the
convergence plots for the case of 60\% added noise, this being the most complex experiment (see Figure~\ref{fig:3:convergence}). We plot the numerical residual $\tilde{r}(\mu^n)$ and the numerical primal-dual gap $\tilde{G}(\mu^n)$ for the iterates $\mu^n$, where $\tilde{r}$ and $\tilde{G}$ are defined in Section \ref{subsec:conv_alg2}. 
We observe that the algorithm
exhibits a linear rate of convergence, instead of the proven sublinear one. Such linear convergence has been numerically observed in all the tested
examples. 
Additionally, 
we see that the algorithm is greatly accelerated when one considers strong 
regularization parameters $\alpha,\beta$. Finally, the plot confirms the efficacy of the proposed descent strategy for the insertion step \eqref{eq:sec5_ins}: indeed we note that the inequality $\tilde{r}(\mu^n)\leq \tilde{G}(\mu^n)$ holds for most of the iterations in the performed experiments. As such inequality is proven in Lemma \ref{lem:primal dual} for the actual residual and primal dual gap, %
we have confirmation that $\gamma^*_{N_n+1}$ in Algorithm \ref{alg:full} is a good approximated solution for \eqref{eq:sec5_ins}. 
Regarding execution times, iterations until convergence and number of restarts, they are summarized in Table~\ref{table:experiment_2}. 

\begin{figure}[t!]
    \center
    \includegraphics{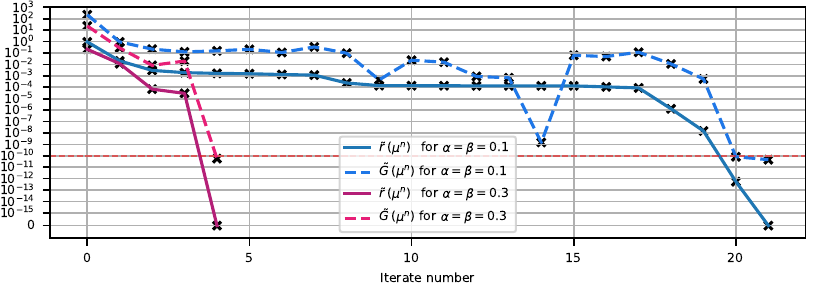}%
  \caption{
    Convergence plot for Experiment 2 with 60\% of added noise and parameter choices $\alpha=\beta=0.1$ and $\alpha=\beta=0.3$. 
    For each iterate,
    we plot the numerical residual $\tilde{r}(\mu^n)$, 
    together with the corresponding numerical dual gap $\tilde{G}(\mu^n)$.
}
  \label{fig:3:convergence}
\end{figure}

\begin{table}[t!]
  \begin{tabular}[h!]{|c|c|c|c|c|}
    \hline 
    \bf Relative Noise & $ \boldsymbol{(\alpha,\beta)}$ & \bf Restarts & \bf Iterations & \bf Execution time \\
    \hline
    0\% & $ (0.1,0.1)$ & 200 & 4 & 1.5 hours \\
    \hline
    20\% & $ (0.1,0,1)$ & 1000 & 7 & 5.8 hours \\
    \hline
    60\% & $ (0.1,0.1)$ & 10000 &  21 & 10.5 days \\
    \hline 
    60\% & $ (0.3,0.3)$ & 5000 & 4 & 16.8 hours  \\
    \hline
  \end{tabular} \\[1ex]
  \caption{
    Convergence information and execution times for Experiment 2. We display the
    considered relative noise level of the data, the employed regularization
    parameters $\alpha,\beta$, the number of restarts $N_{\rm max}$ in
    Subroutine~\ref{alg:multistart}, the number of iterations until reaching
    convergence, and the total execution time of  Algorithm~\ref{alg:full}.
    The employed CPU was an Apple M1 8 Core 3.2 GHz, running native arm64 Python 3.9.7.
     For comments on execution times, see Sections~\ref{sec:exectime}, \ref{subsec:conclusions}. }
     
  \label{table:experiment_2}
\end{table}

\begin{rem}
  \label{rem:rotating_line_measurements}
  Consider the Fourier-type forward measurements $K_t^*$ defined by \eqref{def:forward_numerics} with frequencies $S_{i} \in (\R^2)^{n_i}$ sampled along rotating lines, as in \eqref{eq:rot_sampl}.  
  In this case, at each fixed time-sample $t_i$, 
the operator $K_{t_i}^*$
 does not encode sufficient information 
to accurately resolve the location of the unknown ground-truth at time $t_i$. As a consequence, any static reconstruction technique,
that is, one that does not jointly employ information from different time samples in order to perform a reconstruction,
would not be able to accurately recover any ground-truth under these measurements.
To justify this claim,
notice that for all time-samples $t_i$, 
the family $\{S_{i,k}\}_{k=1}^{n_i}$ defined at \eqref{eq:rot_sampl} 
is collinear, and as such, there exists a corresponding vector $S_i^{\perp} \in \R^2$ such that $S_{i,k} \cdot S_{i}^{\perp} = 0$ for all $k = 1,\ldots, n_i$. 
Therefore, for any given time-static source $\rho^\dagger = \delta_{x^\dagger}$
with $x^\dagger \in (0.1,0.9)^2 \subset \R^2$, 
the measured forward data is invariant 
along $S_{i}^{\perp}$, that is, 
\begin{equation*}
  K_{t_i}^* \delta_{x^\dagger} = K_{t_i}^* \delta_{x^\dagger + \lambda S_{i}^{\perp}}\,,
  \quad  \text{ for all } \, \lambda \in \R \, \text{ such that } \, x^\dagger + \lambda S_{i}^{\perp}
  \in (0.1,0.9)^2.
\end{equation*}
Hence, solely with the information of a single time-sample, 
it is not possible to distinguish a source along a line, and therefore,
it is not possible to accurately resolve it. This is in contrast with the dynamic model presented in this paper, which is able to perform an efficient reconstruction, as demonstrated in Experiment 2. 
\end{rem}

\subsubsection{Experiment 3 - Crossing example} \label{subsec:Ex3}
The following is an example in which the considered model 
is not able to track dynamic sources: although the reconstruction is close to the ground truth in the sense of measures, 
its underlying curves do not resemble those of the ground truth. This effect is due to the non-injectivity of the map at \eqref{eq:map}: even if the sparse measure we wish to recover is unique, its decomposition into atoms might not be. A simple example in which injectivity fails is given by the crossing of two curves (see Remark \ref{rem:crossing}): this is the subject of the numerical experiment performed in this section. 
Specifically, the ground-truth $\mu^\dagger$ considered is of the form 
\begin{equation}   \label{eq:ex2:curves}
\begin{gathered}
\mu^\dagger := a_{\gamma_1^\dagger}^{-1} \rho_{ \gamma_1^\dagger} +
a_{\gamma_2^\dagger}^{-1} \rho_{\gamma_2^\dagger}\,, \\	
\gamma_1^\dagger(t) := (0.2,0.2)+ t(0.6, 0.6), \qquad
 \gamma_2^\dagger(t) := (0.8,0.2)+  t(-0.6, 0.6)\,.
\end{gathered}
\end{equation}
Notice that $\gamma_1^\dagger$ and $\gamma_2^\dagger$ cross at time $t=0.5$, and the respective atoms have both intensity 1.
For the forward measurements, we employ the time-constant Archimedean spiral
family of frequencies defined in the first numerical experiment in Section \ref{subsec:Ex1}, resulting in the constant in time  operator $K_{t_i}^*=K^*$. The reconstruction is performed for noiseless data $f_{t_i}:=K^*\rho^\dagger_{t_i}$. In Figure~\ref{fig:2:data} 
we present the considered frequency samples, 
together with the backprojected data at selected time samples. 
\begin{figure}[t!]
  \begin{subfigure}[t]{0.22\textwidth}
      \includegraphics{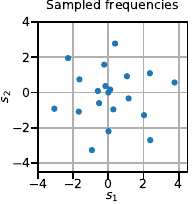}%
  \end{subfigure}\hspace{0.005\textwidth}%
  \begin{subfigure}[t]{0.775\textwidth}
    \includegraphics{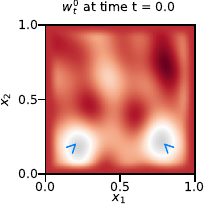}%
    \includegraphics{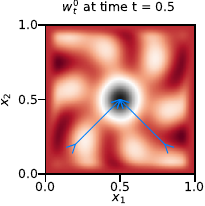}%
    \includegraphics{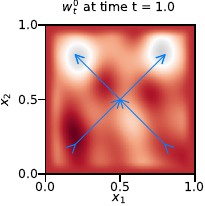}%
    \raisebox{0.14in}{%
      \includegraphics{figures_ex1_dual_var_colorbar.pdf}%
    }%
  \end{subfigure}
  \caption{Time-constant frequency samples $\{S_{i,k}\}_{k}$ and corresponding backprojected data
    $w_t^0$ for Experiment 3. 
    The backprojected data is taken in the noiseless case, and 
    presented at times $t_0 = 0,\ t_{25} = 0.5,\ t_{50} = 1$, with 
    associated ground-truth superimposed in blue color.}
  \label{fig:2:data}
\end{figure}
\begin{figure}[t!]
  \begin{subfigure}[c]{0.3\textwidth}
      \includegraphics{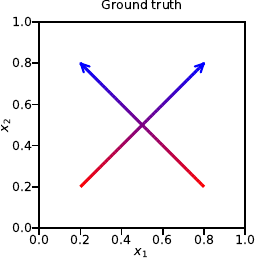}%
    \caption{Considered ground-truth.}
    \label{fig:2:ground_truth}
  \end{subfigure}%
  \begin{subfigure}[c]{0.7\textwidth}
    \includegraphics{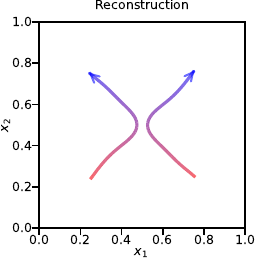}%
    \includegraphics{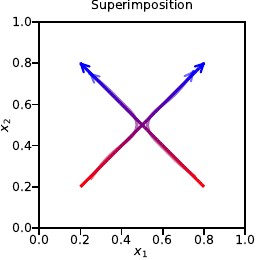}%
    \includegraphics{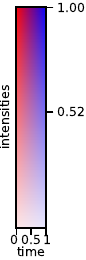}%
    \caption{Computed reconstruction with parameters 
    $(\alpha,\beta) = (0.5,0.5)$.}
    \label{fig:2:recons}
  \end{subfigure}
  \caption{Reconstruction results for Experiment 3. In this case the method fails to reconstruct the crossing, but still approximates the ground-truth in terms of measures.}
  \label{fig:2:reconstruction_comparison}
\end{figure}

In Figure~\ref{fig:2:reconstruction_comparison} we display 
the obtained reconstruction for high regularization parameters $\alpha=0.5$ and $\beta=0.5$. 
It is observed that the reconstructed atoms 
are not close, as curves, to the ones in \eqref{eq:ex2:curves}:  rather than a crossing at time $t=0.5$, the two curves rebound.  
As already mentioned, this phenomenon is a consequence of the lack of uniqueness for the sparse representation of $\mu^\dagger$, which in this case is both represented by crossing curves and rebounding curves. The fact that Algorithm \ref{alg:full} outputs rebounding curves is due to employed regularization: indeed the considered Benamou-Brenier-type penalization selects a solution
whose squared velocity is minimized, which discourages the reconstruction to follow the crossing path. However, we remark that the algorithm proposed yields a good solution in terms of the model, given that, in the sense of measures, the obtained
reconstruction is very close to the ground truth $\mu^\dagger$. More sophisticated models are needed in order to resolve the crossing, as briefly discussed in Section \ref{sec:perspectives}.

\subsection{Remark on execution times}\label{sec:exectime}
It is observed that the execution times of our algorithm are quite high for
some of the presented examples in Table \ref{table:experiment_2}. This is
mainly due to the computational cost of the insertion step, since the algorithm
is set to run several gradient descents to find a global minimum of the
linearized problem \eqref{eq:aux4} at each iteration, and these descents are
executed in a non-parallel fashion on a single CPU core.  The other components
of the algorithm, namely the routines \texttt{sample} and \texttt{crossover}
included in Subroutine \ref{alg:multistart}, the coefficient optimization step
and the sliding step, have, in comparison, negligible computational cost. In
particular, the sliding step grows
in execution time with the number of active atoms, but this effect appears 
towards the last iterations of the algorithm, and it is shadowed by
the insertion step, whose gradient descents become longer as the iterate is
getting closer to the optimal value.  As a confirmation of the role of the
insertion step in the overall computational cost, one can see that the
execution times of the algorithm linearly depend on the total number of
gradient descents that are run in each example. Indeed, since the total number
of gradient descents is given by the number of restarts multiplied by the
number of iterations (see Table \ref{table:experiment_2}), the ratio
between the execution times and the total number of gradient descents
is of the same order for all the presented examples (between \(0.0008\) and \(0.0019\) hour/gradient descent).
It is worth pointing out that the multistart gradient descent
is a highly parallelizable method. Some early tests in this direction indicate
that much lower computational times are achievable by simultaneously computing
gradient descents on several GPUs for the presented examples.

\subsection{Conclusions and discussion} \label{subsec:conclusions}
The presented numerical experiments confirm the effectiveness of Algorithm~\ref{alg:full}, and that the proposed Benamou-Brenier-type energy is an excellent candidate to regularize dynamic inverse problems. 
This is in particular evidenced by Experiment 2 in Section \ref{example:complex}, where we consider
a dynamic inverse problem that is impossible to tackle with a purely static approach, as discussed in Remark \ref{rem:rotating_line_measurements}. Even in the extreme case of $60$\% 
added noise, our method recovered satisfactory reconstructions.

We can further observe the distortions induced 
by the considered regularization.
Precisely, the attenuation of the reconstruction's velocities and intensities is a direct consequence of the
minimization of the Benamou-Brenier energy and the
total variation norm in the objective functional; for this reason, the choice of the regularization parameters $\alpha$ and $\beta$ affects the reconstruction and the magnitude of such distortions.
Additionally, a further effect of the regularization %
is the phenomenon presented in Experiment 3 (Section \ref{subsec:Ex3}). 
As the dynamic inverse problem is formulated in the space of measures, if the sparse ground truth possesses many different decompositions into extremal points, the preferred reconstruction may be the one favoring the regularizer.  
Finally, concerning the execution times, we emphasize that the presented simulations are a proof of concept and not the result of a carefully optimized algorithm. There are many improvement directions, where the most promising one is a GPU implementation to parallelize the multiple insertion step. To increase the likelihood of finding a global minimizer for the
insertion step, Subroutine~\ref{alg:multistart} was employed with a high
number of restarts $N_{\rm max}$, which was tuned manually, prioritizing high reconstruction accuracy over execution time. To improve on this aspect, one could include early stopping conditions in Subroutine~\ref{alg:multistart}, for example by exiting the routine when a sufficiently high ratio of starts descend towards the same stationary curve. Last, 
the code was written having in mind readibility, as well as adaptability to 
a broad class of inverse problems. 
Therefore, the experienced execution times are not an accurate estimation
of what would be possible in specific applications.

\section{Future perspectives}\label{sec:perspectives}
In this section we propose several research directions to expand on the research presented in this paper. %
A first relevant question concerns the proposed DGCG algorithm, and, in particular, the possibility of proving a theoretical linear convergence rate under suitable structural assumptions on the minimization problem \eqref{eq:prel_main_prob}. Linear convergence has been recently proven for the GCG method applied to the BLASSO problem
 \cite{flinth2020,pieper2020}. It seems feasible to extend such an analysis to the DGCG algorithm presented in this paper, especially seeing the linear convergence observed in the experiments provided in  Section \ref{subsec:experiments}, and the fact that our proof of sublinear convergence (see Theorem \ref{thm:convergence}) does not fully exploit the coefficients optimization step, as commented in Remark~\ref{rem:convergence}.  This line of research is currently under investigation by the authors \cite{inprep2}.

Another interesting research direction is the extension of the DGCG method introduced in this paper to the case of unbalanced optimal transport. Precisely, one can regularize the inverse problem \eqref{intro:dip} by replacing the Benamou-Brenier energy $B$ in  \eqref{intro:inverseproblem} with the so-called Wasserstein-Fischer-Rao energy, as proposed in \cite{bf}. Such energy, first introduced in \cite{chizat,kmv,liero} as a model for unbalanced optimal transport, accounts for more general  displacements $t \mapsto \rho_t$, %
in particular allowing the total mass of $\rho_t$ to vary during the evolution.  %
The possibility to numerically treat such a problem with conditional gradient methods would rest on the characterization of the extremal points for the Wasserstein-Fischer-Rao energy recently achieved by the authors in \cite{superposition}. 

In addition, it is a challenging open problem to design alternative dynamic regularizers that allow to reconstruct accurately a ground-truth composed of crossing atoms, such as the ones considered in the experiment in Section \ref{subsec:Ex3}. Due to the fact that the considered Benamou-Brenier-type regularizer penalizes the square of the velocity field associated to the measure, the reconstruction obtained by our DGCG algorithm does not follow the crossing route (Figure \ref{fig:2:recons}).
A possible solution is to consider additional high-order  regularizers in \eqref{eq:prel_main_prob}, such as curvature-type penalizations. The challenging part is devising a penalization that can be enforced at the level of Borel measures, and whose extremal points are measures concentrated  on sufficiently regular curves.

Finally, keeping into account the possible improvements discussed in Section \ref{subsec:conclusions}, the implementation of an accelerated and parallelized version of Algorithm \ref{alg:full} will be the subject of future work.

\section*{Acknowledgements}

KB and SF gratefully acknowledge support by the Christian Doppler Research Association (CDG) and Austrian Science Fund (FWF) through the Partnership in Research
project PIR-27 ``Mathematical methods for motion-aware medical imaging'' and project P 29192 ``Regularization graphs for variational imaging''.
MC is supported by the Royal Society (Newton International Fellowship NIF\textbackslash R1\textbackslash 192048). %
The Institute of Mathematics and Scientific Computing, to which KB, SF, FR are affiliated, is a member of NAWI Graz (\url{http://www.nawigraz.at/}). The authors  KB, SF, FR are further members of/associated with BioTechMed Graz (\url{https://biotechmedgraz.at/}).
This version of the article has been accepted for publication, after
peer review but is not the Version of Record and does not reflect post-acceptance improvements, or any corrections. The Version of Record is available online at: http://dx.doi.org/10.1007/s10208-022-09561-z.

    \bibliographystyle{my_plain}
    \bibliography{bibliography.bib}

\appendix
\section{}

\subsection{Lemmas on optimal transport regularization} \label{subsec:properties_benamou}

In this section we recall several results concerning the continuity equation \eqref{cont weak}, the functionals $B$ and $J_{\alpha,\beta}$ introduced at \eqref{intro convex} and  \eqref{prel:reg} respectively, and the data spaces $L^2_H$ at \eqref{def:ltwoh}. For proofs of such results we refer the reader to Propositions 2.2, 2.4 and Lemmas 4.2, 4.5, 4.6 in \cite{bf}, and to Proposition 5.18 in \cite{sant}.

\begin{lem}[Properties of the continuity equation] \label{lem:prop cont}
	Assume that $\mu=(\rho,m) \in \M$ satisfies \eqref{cont weak} and that $\rho \in \M^+(X)$. Then $\rho$ disintegrates with respect to time into $\rho =dt \otimes  \rho_t$, 
	where $\rho_t \in \M^+(\overline{\Om})$ for a.e.~$t$, and $t \mapsto \rho_t(\overline{\Om})$ is constant, with $\rho_t(\overline{\Om}) = \rho(X)$ for a.e. $t \in (0,1)$. Moreover $t \mapsto \rho_t$ belongs to $\pcurves$ if, in addition, $m=v\rho$ for some measurable $v \colon X \to \R^d$ such that 
	\[
	\int_0^1\int_{\olom} |v(t,x)| \, d\rho_t(x) \, dt < +\infty\,.
	\] 
\end{lem}

\begin{lem}[Properties of $B$] \label{lem:prop B}
	The functional $B$ defined in \eqref{intro convex} is non-negative, convex, one-homogeneous and sequentially lower semicontinuous with respect to the weak* topology on $\M$. Moreover the following properties hold:
	\begin{enumerate}
		\item [i)] if $B(\rho,m)<+\infty$, then $\rho \geq 0$ and $m \ll \rho$, that is, there exists a measurable map $v \colon X \to \R^d$ such that $m=v\rho$,
	\item [ii)] let $\Psi$ be the map at \eqref{intro Psi}. If $\rho \geq 0$ and $m = v \rho$ for some $v \colon X \to \R^d$ measurable, then
\begin{equation} \label{formula B}
B(\rho,m) =\int_X \Psi (1,v) \, d\rho =  \frac12 \int_X |v|^2 \, d\rho \,.
\end{equation}
	\end{enumerate}
\end{lem}

\begin{lem}[Properties of $J_{\alpha,\beta}$] \label{lem:prop J}
Let $\alpha, \beta >0$. The functional $J_{\alpha,\beta}$ at \eqref{prel:reg} is non-negative, convex, one-homogeneous and sequentially lower semicontinuous with respect to weak* convergence on $\M$.
		For $\mu=(\rho,m) \in \mathcal{M}$ such that $J_{\alpha,\beta}(\mu)< +\infty$ we have that 
	\begin{equation} \label{lem:prop J est}
       \max\{   \alpha \norm{\rho}_{\M(X)},  C \norm{m}_{\M(X;\R^d)}\} \leq
          J_{\alpha,\beta}(\mu) \,,
	\end{equation}	
	where $C:=\min\{2\alpha,\beta\}$. 
	Moreover, if $\{\mu^n\}$ sequence in $\M$ is such that
$\{J_{\alpha,\beta} (\mu^n)\}$ is uniformly bounded,
then $\rho^n = dt \otimes \rho_t^n$ for some $(t \mapsto \rho_t^n) \in \pcurves$, and there exists 
	$\mu=(\rho,m) \in \mathcal{D}$ with $\rho= dt \otimes \rho_t$ and $(t\mapsto \rho_t) \in \pcurves$, such that, up to subsequences, 
	\begin{equation} \label{topology} \left\{
	\begin{gathered}
	(\rho^n,m^n) \weakstar (\rho,m) \,\, \text{ weakly* in } \,\, \M \,, \\
	\rho_t^n \weakstar \rho_t \,\, \text{ weakly* in }  \,\, \M(\olom)\,,\, \text{ for every } \,\, t \in [0,1]\,.   
	\end{gathered} \right.
	\end{equation}
\end{lem}

 \begin{lem}[Properties of $K_t^*$] \label{lem:prop K}
  Assume \ref{ass:H1}-\ref{ass:H3}, \ref{ass:K1}-\ref{ass:K3} as in Section \ref{sec:assumptions}. If $t \mapsto \rho_t$ is in $\curves$, then the map $t \mapsto K_t^* \rho_t$ belongs to $\Ltwo$. Let $\{(\rho^n,m^n)\}$ in $\mathcal{D}$ be such that $\rho^n=dt \otimes \rho^n_t$ and $(t \mapsto \rho_t^n)  \in \pcurves$. If $(\rho^n,m^n)$ converges to $(\rho,m)$ in the sense of \eqref{topology}, where $\rho=dt \otimes \rho_t$, $(t \mapsto \rho_t)  \in \pcurves$, then we have $K^*\rho^n \weak K^*\rho$ weakly in $\Ltwo$.
  \end{lem}

\subsection{Proof of Lemma \ref{lem:additivity}} \label{app:proof:lemma}

The fact that $J_{\alpha,\beta}(\mu_j)=1$ follows by \eqref{formula B}. 
Define the vector field %
\[
v(t,x):= 
\begin{cases}
\dot\gamma_j(t)   & \text{ if } \,\, (t,x) \in \gr \gamma_j  \,,\\
0 & \text{ otherwise }\,,
\end{cases}
\]
where $\gr \gamma_j:=\{(t,\gamma_j(t)) \, \colon \, t \in (0,1)\}\subset X$. 
Notice that $v$ is well-defined up to negligibly many $t \in (0,1)$: indeed, we have $\dot \gamma_i= \dot \gamma_j$ a.e.~in $\{t: \gamma_i(t) =  \gamma_j(t)\}$ for every $i,j$ (see \cite[Theorem 4.4]{evansgariepy}). Hence $\dot \gamma_i(t) = \dot \gamma_j(t)$ for a.e.~$t \in (0,1)$ and every $x$ such that $(t,x) \in \gr \gamma_j \cap \gr \gamma_i$. Set now
$\mu:=\sum_{j=1}^N c_j \mu_{\gamma_j}$. 
It is immediate to see that $v=d m/d \rho$ and that $v$ satisfies $v(t,\gamma_j(t)) = \dot \gamma_j(t)$ for all $j =1,\ldots,N$.
Moreover, by linearity, $\mu$ satisfies the continuity equation \eqref{cont weak}. Employing %
\eqref{formula B} and the definition of $a_{\gamma_j}$, we conclude noting that
\[
\begin{aligned}
  J_{\alpha,\beta}(\mu) & = \frac{\beta}{2} \int_X |v(t,x)|^2 \, d\rho + \alpha \norm{\rho}_{\M(X)} 
 =  \sum_{j=1}^N c_j a_{\gamma_j} \left(\frac{\beta}{2} \int_0^1 |v(t,\gamma_j(t))|^2\, dt + \alpha \right)\\
 & =  \sum_{j=1}^N c_j a_{\gamma_j} \left(\frac{\beta}{2} \int_0^1 |\dot{\gamma}_j(t))|^2\, dt + \alpha \right)
 = \sum_{j=1}^N c_j\,.
\end{aligned}
\] 

\subsection{Existence of minimizers for linearized problems}

In this section we show existence of minimizers for the problems at \eqref{aux2} and \eqref{eq:aux4}. To this end, we prove existence for a slightly more general functional (see \eqref{eq:prox:1} below), which coincides with \eqref{aux2} and \eqref{eq:aux4} for $\f$ as in \eqref{def:fi} and $\f(t):=\rchi(-\infty,1](t)$, respectively. 

  \begin{thm} \label{ex min gen}
    Assume \ref{ass:H1}-\ref{ass:H3}, \ref{ass:K1}-\ref{ass:K3} and let $f \in \Ltwo$, $\alpha, \beta >0$. Given $(t \mapsto \tilde{\rho}_t) \in \curves$ define 
    $w_t:=-K_t(K_t^*\tilde \rho_t - f_t ) \in C(\olom)$ 
     for a.e.~$t \in (0,1)$. Let $\f \colon \R \to [0,+\infty]$ with $\f(0)=0$ be monotonically increasing, lower semicontinuous and super-linear at infinity, i.e.,  $\f(t)/t \to +\infty$ as $t \to +\infty$. 
  Then there exists $\mu^*=(\rho^*,m^*) \in \mathcal{D}$ that solves the minimization problem
  \begin{equation}
   \min_{\mu \in \calM} - \ps{\rho}{w} + \varphi(J_{\alpha,\beta}(\mu))\,,
    \label{eq:prox:1}
  \end{equation}
  where the product $\ps{\cdot}{\cdot}$ is defined at \eqref{eq:scalarproduct_alg}. 
  Moreover $\rho^*=dt \otimes \rho^*_t$ with $(t \mapsto \rho^*_t) \in \pcurves$.
  \end{thm}

  \begin{proof}
  First notice that the functional at \eqref{eq:prox:1} is proper since $J_{\alpha,\beta}(0)=0$ (by Lemma \ref{lem:prop B}) and $\f(0)=0$. Let $\{\mu^n\}$ be a minimizing sequence, so that, in particular,
  \begin{equation} \label{thm:wolfe:1}
  \sup_n \,\, - \ps{\rho^n}{w} + \varphi(J_{\alpha,\beta}(\mu^n)) < +\infty \,.
  \end{equation}
  We claim that $\sup_n J_{\alpha,\beta}(\mu^n) < +\infty$. Indeed, assume by contradiction that $J_{\alpha,\beta}(\mu^n) \to +\infty$ as $n \to +\infty$ (subsequentially). Fix $C'>0$. Since $\f$ is super-linear there exists $n_0 \in \N$ such that 
  \begin{equation} \label{thm:wolfe:22}
  \f (J_{\alpha,\beta}(\mu^n)) \geq C' J_{\alpha,\beta}(\mu^n) \,\,\, \text{ for all } \,\, n \geq n_0\,.	
  \end{equation}
Moreover, notice that for $n$ fixed we have $J_{\alpha,\beta}(\mu^n)<+\infty$, as \eqref{thm:wolfe:1} holds. In particular, we obtain that $\mu^n \in \mathcal{D}$ and $\rho^n=dt \otimes \rho^n_t$ with $(t \mapsto \rho_t^n) \in \pcurves$, thanks to Lemmas \ref{lem:prop cont}, \ref{lem:prop B}. By definition of $\ps{\cdot}{\cdot}$ at \eqref{eq:scalarproduct_alg}, assumptions \ref{ass:K1}-\ref{ass:K2} and Cauchy-Schwarz we obtain, for all fixed $n \in \N$,
  \begin{equation} \label{thm:wolfe:2}
  \begin{aligned}
  \ps{\rho^n}{w} & =   
  -\int_0^1 \ps{K_t^*\rho_t^n}{K_t^*\tilde{\rho}_t- f_t}_{H_t} \, dt \leq 
  \int_0^1 \norm{K_t^*\rho_t^n}_{H_t} \norm{K_t^*\tilde{\rho}_t- f_t}_{H_t} \, dt
  \\
  & \leq \, C \, \sup_t \norm{\rho_t^n}_{\M(\olom)} \int_0^1 \norm{K_t^*\tilde{\rho}_t- f_t}_{H_t} \, dt \leq 
  C \norm{\rho^n}_{\M(X)} \,  \norm{K^* \tilde \rho -f}_{L^2_H} \,,
  \end{aligned}
  \end{equation}
where $C>0$ is the constant from \ref{ass:K2}, and where we used that $\norm{\rho^n_t}_{\M(\olom)}=\norm{\rho^n}_{\M(X)}$ for each $t \in [0,1]$ (see Lemma \ref{lem:prop cont}). 
From \eqref{lem:prop J est}, \eqref{thm:wolfe:2} and \eqref{thm:wolfe:22} we get
\[
\begin{aligned}
- \ps{\rho^n}{w} + \varphi(J_{\alpha,\beta}(\mu^n)) & \geq -  C \norm{\rho^n}_{\M(X)} \,  \norm{K^* \tilde \rho -f}_{L^2_H} + 
 C' J_{\alpha,\beta}(\mu^n) \\
& \geq  J_{\alpha,\beta}(\mu^n) \lbrack C' - C \alpha^{-1}\norm{K^* \tilde \rho -f}_{L^2_H} \rbrack\,, %
\end{aligned}
\]
for all $n \geq n_0$. By choosing $C'>0$ sufficiently large, the above estimate contradicts \eqref{thm:wolfe:1}, showing that $\sup_n J_{\alpha,\beta}(\mu^n) < +\infty$. 
In this case Lemma \ref{lem:prop J} ensures that $(\rho^n,\mu^n)$ converges to $\mu^*=(\rho^*,m^*)$ in the sense of \eqref{topology}, up to subsequences, and $\mu^* \in \mathcal{D}$, $\rho^*=dt \otimes \rho_t^*$ with $(t \mapsto \rho^*_t) \in \pcurves$.  %
  In particular $K^*\rho^n \weak K^*\rho^*$ weakly in $\Ltwo$ by Lemma \ref{lem:prop K}. As $w_t = - K_t(K_t^* \tilde{\rho}_t- f_t)$ with $(K^* \tilde{\rho}- f) \in \Ltwo$ (Lemma \ref{lem:prop K}), from \ref{ass:K1} we deduce that $\ps{\rho^n}{w} \to \ps{\rho^*}{w}$ for $n \to +\infty$. 
  Recall that $J_{\alpha,\beta}$ is weak* lower semicontinuous (Lemma \ref{lem:prop J}). As $\f$ is lower-semicontinuous and monotonically increasing, we deduce that $\f \circ J_{\alpha,\beta}$ is weak* lower-semicontinuous. As $\{\mu^n\}$ is a minimizing sequence, by \eqref{topology} and Lemma \ref{lem:prop K}, we conclude that $\mu^*$ solves \eqref{eq:prox:1}. 
  \end{proof}

  \subsection{Analysis for the insertion step}\label{sec:computingextremal}

In this section we show that, under the assumptions \ref{ass:F1}-\ref{ass:F4} of Section \ref{sec:insstepheu} on the forward operators $K_t^*$, and assumptions \ref{ass:H1}-\ref{ass:H3} of Section \ref{sec:assumptions} on the sampling spaces $H_t$,  it is possible to tackle the insertion step problem \eqref{eq:sec5_ins} numerically, by means of gradient descent methods. %
  To this end, it is convenient to introduce the functionals  $F,W,L \colon H^1([0,1];\olom) \to \R$ as
\begin{equation} \label{def:operators FWL}
F(\gamma):=\frac{W(\gamma)}{L(\gamma)}   \,,
  \quad W(\gamma):=- \int_0^1 w_t(\gamma(t)) \,dt \,, \quad L(\gamma) := \frac{\beta}{2} \int_0^1 | \dot{\gamma}(t)|^2 \, dt + \alpha  \,.
  \end{equation}
As observed in Remark \ref{rem:min_ins}, the only case of interest is when the minimum value of the problem at \eqref{eq:sec5_ins} is stricly negative. 
Thus, we assume to be in such situation, and consider
\begin{equation} \label{eq:curves_F}
  \min_{\gamma \in H^1([0,1];\olom)} F(\gamma)  \,.
  \end{equation}
As discussed in Section \ref{sec:insstepheu}, we are interested in computing stationary points of $F$ by gradient descent. 
In order to make this possible, we first extend $F$ to the Hilbert space $H^1:=H^1([0,1];\R^d)$, in a way that the set of stationary points of $F$ is not altered. To be more precise, by assumptions \ref{ass:F1}-\ref{ass:F3} we have that the dual variable $w_t:=-K_t(K_t^*{\tilde{\rho}}_t - f_t)$ belongs to $C^{1,1}(\olom)$ for a.e.~$t \in (0,1)$. Additionally, \ref{ass:F4} implies that $\supp w_t , \,\supp \nabla w_t \subset E$ for a.e.~$t \in (0,1)$,  where $E \Subset \Om$ is closed and convex. We can then extend $w_t$ to the whole $\R^d$ by setting $w_t (x):=0$ for all $x \in \R^d \smallsetminus \olom$ and a.e.~$t \in (0,1)$. Consequently, the functional $F$ is well defined via \eqref{def:operators FWL} over the space $H^1$. 

In the above setting we are able to prove that $F$ is continuously Fr\'echet differentiable over $H^1$ (Proposition \ref{prop:gateaux} below). Denote by $D_\gamma F \in (H^1)^*$ the Fr\'echet derivative of $F$ at $\gamma$.  We also show that stationary points of $F$, i.e., curves $\gamma^* \in H^1$ such that $D_{\gamma^*}F=0$, satisfy $\gamma^*([0,1])\subset E \subset \Om$ whenever $F(\gamma^*)\neq 0$ (Proposition \ref{cor:optimality} below). Therefore %
problem \eqref{eq:curves_F} is equivalent to 
\begin{equation} \label{eq:curves_H1}
  \min_{\gamma \in H^1} F(\gamma)  \,.
  \end{equation}
We can now apply the gradient descent algorithm to compute stationary points of $F$, in the attempt of approximating solutions to  \eqref{eq:curves_H1}, and hence to \eqref{eq:curves_F}. 
The main result of this section states that descent sequences for $F$, in the sense of \eqref{def:sec5_descent}, converge (subsequentially) to stationary points.

\begin{thm} \label{thm:gradient_descent}
Assume \ref{ass:F1}-\ref{ass:F4} as in Section \ref{sec:insstepheu} and \ref{ass:H1}-\ref{ass:H3} as in Section \ref{sec:assumptions}. Assume given $(t \mapsto \tilde{\rho}_t) \in \curves$, $f \in \Ltwo$ and $\alpha,\beta>0$. For a.e.~$t \in (0,1)$ set $w_t:=-K_t(K_t^*\tilde{\rho}_t-f_t)$ and $w_t(x):=0$ for all $x \in \R^d \smallsetminus \olom$. Then, the corresponding functional $F \colon H^1 \to \R$ defined via \eqref{def:operators FWL} is continuously Fr\'echet differentiable. If $\{\gamma^n\}$ in ${H^1}$ is a descent sequence for $F$ in the sense of \eqref{def:sec5_descent} then, up to subsequences, $\gamma^n \to \gamma^*$ strongly in $H^1$. Any such accumulation point $\gamma^*$ satisfies $F(\gamma^*)<0$ and is stationary for $F$, that is, $D_{\gamma^*}F=0$. Moreover  $\gamma^*([0,1]) \subset E$, where $E \Subset \Om$ is the closed convex set in \ref{ass:F4}.  
\end{thm}

The proof of Theorem \ref{thm:gradient_descent}, postponed to Section \ref{sec:grad_descent} below, relies on differentiability results for $F$ and on properties of its stationary points, as discussed in the following Section \ref{sec:diff_F}. Finally, in Section \ref{subsec:test} we provide an implementable criterion to determine whether the minimum of \eqref{eq:sec5_ins} is strictly negative.

\subsubsection{Differentiability of $F$ and stationary points} \label{sec:diff_F}

In this section we discuss Fr\'echet differentiability and stationary points properties for the (extended) operator $F \colon H^1 \to \R$ defined at \eqref{def:operators FWL}. Before proceeding with the discussion, we establish a few notations and make some remarks on assumptions \ref{ass:F1}-\ref{ass:F4}.

In the following, for any closed $S \subseteq \R^d$ we denote by $C^{1,1}(S)$ the space of differentiable maps $\f \colon S \to \R$ such that the norm 
  \[
  \norm{\f}_{C^{1,1}(\olom)}:=\norm{\f}_\infty + \norm{\nabla \f}_\infty + \rm{Lip}(\nabla \f), \qquad \rm{Lip}(\nabla \f):=\sup_{x \neq y} \frac{|\nabla \f(x) - \nabla \f (y)|}{|x-y|}\,,
  \]
  is finite, where $\nabla \colon C^{1,1}(S) \to C(S;\R^d)$ is the gradient operator. Notice that in this case $\nabla$ is linear and continuous.  We will also consider the Bochner space $L^2([0,1];C^{1,1}(S))$ equipped with the norm $\norm{w}_{1,1}^2:=\int_0^1 \norm{w_t}_{C^{1,1}(S)}^2 \, dt$. Finally, for two Banach spaces $X, Y$ and a continuously Fr\'echet differentiable map $G \colon X \to Y$, we denote the differential of $G$ by $DG \colon X \to Y^*$, with evaluation at $u \in X$ given by the linear continuous functional $v \mapsto D_u G(v)$ belonging to $Y^*$.

  \begin{rem} \label{rem:assumptions}
Notice that \ref{ass:F3} also holds for $\rho \in C^{1,1}(\olom)^*$: indeed, let $\{\rho^n\}$ sequence in $\M(\olom)$ be such that $\norm{\rho^n-\rho}_{C^{1,1}(\olom)^*}\to 0$ as $n \to +\infty$. Then the maps $t \mapsto K_t^* \rho^n$ are strongly measurable at each fixed $n$ by \ref{ass:F3}. From \ref{ass:F2} %
  we have $\norm{K_t^*\rho^n-K_t^*\rho}_{H_t}\leq C \norm{\rho^n-\rho}_{C^{1,1}(\olom)^*}\to 0$ for a.e.~$t \in (0,1)$, implying that $t \mapsto K_t^*\rho$ is strongly measurable by \cite[Remark 3.4]{bf}. 
  
  A way to ensure that \ref{ass:F1}-\ref{ass:F4} hold is as follows. First notice that if $K_t^*$ satisfies \ref{ass:F1}-\ref{ass:F3} and $D \colon C^{1,1}(\olom) \to C^{1,1}(\olom)$ is a linear bounded operator, then one can check that also $\tilde{K}_t^*:=K_t^* D^*$ satisfies \ref{ass:F1}-\ref{ass:F3}. Now let $E' \Subset E$ be closed and let $\xi_E$ be a cut-off function such that
\begin{equation} \label{def:cut-off}
\xi_E \in C^{1,1}(\olom), \,\,\, 0 \leq \xi_E \leq 1\,,  \,\,\, \xi_E=1 \, \text{ in } \, E'\,, \, \,\, \xi_E=0 \, \text{ in } \, \olom \smallsetminus E\,.
\end{equation}
Defining $D$ by $D\f :=\xi_E \f$ for all $\f \in C^{1,1}(\olom)$ yields that $\tilde{K}_t=K_t D$ satisfies \ref{ass:F4}.  
\end{rem}

In order to show that $F$ is differentiable, we first investigate the regularity for dual variables $t \mapsto w_t$ of the form considered in \eqref{eq:sec5_ins}. The differentiability properties for $F$ are considered afterwards.

 \begin{lem} \label{lem:C1}
 Assume \ref{ass:F1}-\ref{ass:F4},  \ref{ass:H1}-\ref{ass:H3} and let $g \in \Ltwo$. For a.e.~$t \in (0,1)$ set $w_t:=K_tg_t$ and $w_t(x):=0$ for all $x \in \R^d \smallsetminus \olom$. Then  $w$ belongs to $L^2([0,1];C^{1,1}(\R^d))$. Moreover $w_t, \nabla w_t$ are Carath{\'e}odory functions in $[0,1] \times \R^d$, that is, $x \mapsto w_t(x)$, $x \mapsto  \nabla w_t(x)$ are continuous for a.e. $t \in (0,1)$ fixed and $t \mapsto w_t(x)$, $t \mapsto  \nabla w_t(x)$ are measurable for all $x \in \R^d$ fixed.  %
  \end{lem}
   
   \begin{proof}
   We first show that $w \in L^2([0,1];C^{1,1}(\olom))$ and that $w$, $\nabla w$ are Carath{\'e}odory in $[0,1] \times \olom$. In order to do so, let us check that the map $t  \mapsto K_t g_t$ is strongly measurable in the classic sense \cite[Ch II]{diestel}. Since $C^{1,1}(\olom)$ is separable, by the Pettis measurability theorem \cite[Ch II.1, Thm 2]{diestel}, strong measurability is equivalent to weak measurability, that is, we need to prove that
   \begin{equation} \label{lem:C1:1}
   t \mapsto \ps{\rho}{K_tg_t}_{C^{1,1}(\olom)^*,C^{1,1}(\olom)}
   \end{equation}
   is measurable for each $\rho \in C^{1,1}(\olom)^*$. Note that $\ps{\rho}{K_tg_t}_{C^{1,1}(\olom)^*,C^{1,1}(\olom)}=\ps{K_t^* \rho}{g_t}_{H_t}$ by \ref{ass:F1}. Moreover $t \mapsto K_t^*\rho$ is strongly measurable by \ref{ass:F3} and Remark \ref{rem:assumptions}. Since $g$ is strongly measurable (as $g \in \Ltwo$), by \cite[Remark 3.4]{bf} we conclude that $t \mapsto \ps{K_t^* \rho}{g_t}_{H_t}$ is measurable. Therefore the measurability of the map at \eqref{lem:C1:1} follows.   
From \ref{ass:F1}-\ref{ass:F2}  we infer $\int_0^1 \norm{K_t g_t}_{C^{1,1}(\olom)}^2 \, dt<+\infty$,   
  since $g \in \Ltwo$. By \cite[Ch II.2, Thm 2]{diestel} we then conclude $w \in L^2([0,1];C^{1,1}(\olom))$. In particular the maps $x \mapsto w_t(x)$, $x \mapsto \nabla w_t (x)$ are continuous for a.e. $t \in (0,1)$ fixed and $x$ varying in $\olom$. Let now $x \in \olom$ be fixed. By \ref{ass:F1} and Remark \ref{rem:assumptions} we have
  $w_t(x)=\ps{\delta_x}{K_tg_t}_{\M(\olom),C(\olom)} = \ps{K_t^*\delta_x}{g_t}_{H_t}$.
  As the map $t \mapsto K_t^*\delta_x$ is strongly measurable by \ref{ass:F3}, and $g$ is strongly measurable since it belongs to $\Ltwo$, from \cite[Remark 3.4]{bf} we conclude that $t \mapsto \ps{K_t^*\delta_x}{g_t}_{H_t}$ is measurable. Thus $w$ is Carath{\'e}odory in $[0,1] \times \olom$. Similarly, we have  
    $\partial_{x_i} w_t (x)%
    =\ps{\nabla^* e_i  \delta_x}{K_tg_t}_{C^{1,1}(\olom)^*,C^{1,1}(\olom)}$ 
  for all $i=1,\ldots,d$, where $e_i$ is the $i$-th coordinate vector in $\R^d$. 
  Notice that $\nabla^* e_i \delta_x \in C^{1,1}(\olom)^*$. Hence the measurability of $t \mapsto \nabla w_t(x)$ is implied by setting $\rho=\nabla^* e_i\delta_x$ in \eqref{lem:C1:1}, showing that $\nabla w$ is Carath{\'e}odory in $[0,1] \times \olom$. 
  Finally, the facts that $w \in L^2([0,1];\R^d)$ and that $w$, $\nabla w$ are Carath{\'e}odory in $[0,1] \times \R^d$, follow since $w_t$ is extended to zero in $\R^d \smallsetminus \olom$ and \ref{ass:F4} holds. %
   \end{proof}

  \begin{prop} \label{prop:gateaux}
   Assume \ref{ass:F1}-\ref{ass:F4}, \ref{ass:H1}-\ref{ass:H3}. Let $(t \mapsto \tilde{\rho}_t) \in \curves$, $f \in \Ltwo$ and $\alpha,\beta>0$ be given. For a.e.~$t \in (0,1)$ set $w_t:=-K_t(K_t^*\tilde{\rho}_t-f_t)$ and $w_t(x):=0$ for all $x \in \R^d \smallsetminus \olom$. Then, the corresponding functionals $F, W, L$ defined at \eqref{def:operators FWL} are continuously Fr{\'e}chet differentiable in $H^1:=H^1([0,1];\R^d)$. The derivatives of $F$, $W$, $L$ at $\gamma \in H^1$ are given by 
    \begin{gather}
     D_{\gamma} F (\eta)  = \frac{ D_{\gamma} W(\eta)}{  L(\gamma)} - F(\gamma) \,\frac{ D_{\gamma} L(\eta)     }{L(\gamma)} \,, \label{gateaux F} \\
      D_{\gamma} W(\eta)  
      = - \int_0^1 \nabla w_t(\gamma(t)) \cdot \eta(t) \, dt \, , \quad 
       D_{\gamma} L(\eta) 
        = \beta \int_0^1  \dot{\gamma}(t) \cdot \dot{\eta}(t) \,dt \,, \label{gateaux derivatives}
        \end{gather}
    for each $\eta \in  H^1$. In addition  we have 
    \begin{equation} \label{gateaux:estimates}
    \sup_{\gamma \in H^1}|F(\gamma)| \leq \frac{\norm{w}_{1,1}}{\alpha} \,, \quad
    \sup_{\gamma \in H^1} \norm{D_\gamma F}_{(H^1)^*} \leq \frac{\norm{w}_{1,1}}{\alpha} \, \left( 1 + \sqrt{\frac{\beta}{2\alpha}} \right) \,,
    \end{equation}
    where $\norm{w}_{1,1}^2:=\int_0^1 \norm{w_t}_{C^{1,1}(\olom)}^2 \, dt$. Last, the map $\gamma \mapsto D_{\gamma} F$ is locally Lipschitz, that is, for all $R>0$ fixed it holds %
    \begin{equation} \label{gateaux:lip}
    \norm{D_{\gamma^1}F - D_{\gamma^2}F}_{(H^1)^*} \leq (C_1 R + C_2) \norm{\gamma^1 - \gamma^2}_{H^1}\,,
    \end{equation}
    for all $\gamma^i \in H^1$ such that $\norm{\gamma^i}_{H^1} \leq R$, $i=1,2$, where $C_1,C_2>0$ are constants depending only on $\alpha,\beta$ and $w$.  
  \end{prop}

  \begin{proof}
The continuous Fr{\'e}chet differentiability of $L$ is standard, and the proof is omitted. Moreover, continuous differentiability of $F$ and formula \eqref{gateaux F} follow from continuous differentiability of $W$ and $L$, and from the quotient rule, given that $L \geq \alpha >0$. Therefore, let us show that $W$ is continuously differentiable with derivative as in \eqref{gateaux derivatives}. Since $(t \mapsto \tilde{\rho}_t) \in \curves$, by Lemma \ref{lem:prop K} we have that $t \mapsto K_t^*\tilde{\rho}_t$ belongs to $\Ltwo$, so that also $g:=-(K^*\tilde{\rho} - f)$ belongs to $\Ltwo$. Set $w_t:=K_t^*g_t$ and $w_t(x):=0$ for all $x \in \R^d \smallsetminus \olom$. By Lemma \ref{lem:C1} we know that $w \in L^2([0,1];C^{1,1}(\R^d))$ and $w$, $\nabla w$ are Carath{\'e}odory maps in $[0,1] \times \R^d$. In particular, for a fixed $\gamma \in H^1$, the maps 
   $t \mapsto w_t(\gamma(t)), t \mapsto \nabla w_t( \gamma(t))$ 
   are measurable \cite[Proposition 3.7]{dacorogna}. 
   Since $w \in L^2([0,1];C^{1,1}(\R^d))$, we can proceed as in the proof of Theorem 3.37 in \cite{dacorogna} and show that
   the G\^{a}teaux derivative of $W$ at $\gamma$, along the direction $\eta$, is given by the first formula in \eqref{gateaux derivatives}.
We are left to prove that $\gamma \mapsto D_\gamma W$ is continuous from $H^1$ into $(H^1)^*$. 
To this end, fix $\gamma^1,\gamma^2 \in H^1$ and notice that
  \begin{equation} \label{gateaux:4}
  \begin{aligned}
  \norm{ D_{\gamma^1} W - D_{\gamma^2} W  }_{(H^1)^*} %
  & \leq \sup_{\substack{\eta \in H^1, \\ \norm{\eta}_{H^1} \leq 1}} \norm{\eta}_\infty \int_0^1 {\rm Lip}(\nabla w_t) | \gamma^1(t) - \gamma^2(t) | \, dt \\
   & \leq \sqrt{2} \norm{w}_{1,1} \norm{ \gamma^1 - \gamma^2}_{H^1}\,, 
   \end{aligned}
  \end{equation}
  where in the last inequality we employed Cauchy-Schwarz and the estimate $\norm{\eta}_{\infty} \leq \sqrt{2} \norm{\eta}_{H^1}$. Notice that \eqref{gateaux:4} shows that the map $\gamma \mapsto D_\gamma W$ is Lipschitz from $H^1$ into $(H^1)^*$. Thus, in particular, $W$ is continuously Fr{\'e}chet differentiable. 
    We will now prove the estimates at \eqref{gateaux:estimates}-\eqref{gateaux:lip}.  
  The first bound in \eqref{gateaux:estimates} follows immediately from the definition of $F$, the fact that $w \in L^2([0,1];C^{1,1}(\R^d))$, and the estimate $L \geq \alpha >0$.  
  As for the second estimate in \eqref{gateaux:estimates}, by \eqref{gateaux F} and the triangle inequality we have
   \begin{equation} \label{gateaux:6}
   \norm{D_\gamma F}_{(H^1)^*} \leq \frac{\norm{ D_\gamma W }_{(H^1)^*}   }{L(\gamma)} +  \norm{\frac{ D_\gamma L}{L(\gamma)} }_{(H^1)^*} \, |F(\gamma)| \,. 
   \end{equation}
  Notice that $\norm{D_\gamma W}_{(H^1)^*} \leq \norm{w}_{1,1}$, thanks to \eqref{gateaux derivatives} and H\"older's inequality.   Moreover, by \eqref{gateaux derivatives} and H\"older's inequality,
    \begin{equation} \label{gateaux:est_L} 
   \norm{\frac{ D_\gamma L}{L(\gamma)} }_{(H^1)^*} \leq 
   \frac{\beta \left( \int_0^1 |\dot{\gamma}|^2 \, dt \right)^{1/2}}{\frac{\beta}{2}   \int_0^1 |\dot{\gamma}|^2 \, dt + \alpha } \leq \sqrt{ \frac{\beta}{2 \alpha} } \,,
   \end{equation}
   where the second estimate is obtained by noting that the real map $s \mapsto \beta s/(\beta s^2 /2 + \alpha)$ is differentiable, with maximum value given by $\sqrt{\beta/(2\alpha)}$. 
   By the first estimate in \eqref{gateaux:estimates} and the fact that $L \geq \alpha$, from \eqref{gateaux:6}-\eqref{gateaux:est_L} we conclude \eqref{gateaux:estimates}. Finally we prove \eqref{gateaux:lip}. To this end, fix $R>0$ and $\gamma^1,\gamma^2 \in H^1$ such that $\norm{\gamma^1}_{H^1},\norm{\gamma^2}_{H^1} \leq R$. Note that,  as a consequence of \eqref{gateaux F}, we get  
      \begin{equation} \label{gateaux:est_11} 
  \norm{D_{\gamma^1}F-D_{\gamma^2}F}_{(H^1)^*} \leq 
  \norm{  \frac{D_{\gamma^1}W}{L(\gamma^1)} - \frac{D_{\gamma^2}W}{L(\gamma^2)}  }_{(H^1)^*} +  \norm{ F(\gamma^1) \, \frac{D_{\gamma^1}L}{L(\gamma^1)} - F(\gamma^2) \, \frac{D_{\gamma^2}L}{L(\gamma^2)}  }_{(H^1)^*}\,.
  \end{equation}
Concerning the first term in \eqref{gateaux:est_11}, observe that, by the estimate $L\geq \alpha$,
     \begin{equation} \label{gateaux:est_12} 
  \begin{aligned}
  \left| \frac{1}{L(\gamma^1)}  - \frac{1}{L(\gamma^2)}      \right|  & =  \left| \frac{L(\gamma^1)-L(\gamma^2)}{L(\gamma^1)L(\gamma^2)}   \right| \leq \frac{\beta}{2\alpha^2} \left|  \int_0^1 (|\dot{\gamma}^1| + |\dot \gamma^2|  )(|\dot{\gamma}^1| - |\dot \gamma^2|  ) \, dt \right|\\
  & \leq \frac{\beta}{2\alpha^2} \left( \norm{\gamma^1}_{H^1}+\norm{\gamma^2}_{H^1} \right) \left( \int_0^1 |   \dot{\gamma}^1 - \dot \gamma^2|^2   \, dt \right)^{1/2} \\
 & \leq R \,\frac{\beta}{\alpha^2} \norm{\gamma^1 - \gamma^2 }_{H^1}\,.
  \end{aligned}
  \end{equation}
  Recall that $\gamma \mapsto D_\gamma W$ is bounded, with  $\norm{D_\gamma W}_{(H^1)^*} \leq \norm{w}_{1,1}$. Also the map $\gamma \mapsto 1/L(\gamma)$ is bounded by $1/\alpha$. Therefore by the Lipschitz estimates \eqref{gateaux:4} and \eqref{gateaux:est_12} we obtain 
  \begin{equation} \label{gateaux:est_13}
  \begin{aligned}
  	 \norm{  \frac{D_{\gamma^1}W}{L(\gamma^1)} - \frac{D_{\gamma^2}W}{L(\gamma^2)}  }_{(H^1)^*} 
  	 & \leq \norm{D_{\gamma^1}W}_{(H^1)^*}   \left| \frac{1}{L(\gamma^1)}  - \frac{1}{L(\gamma^2)}      \right| \\
  	 & \qquad \qquad \qquad \qquad + \frac{1}{L(\gamma^2)} \norm{D_{\gamma^1}W - D_{\gamma^2}W}_{(H^1)^*} \\
  	 & \leq \norm{w}_{1,1}  \left( \frac{R \beta}{\alpha^2}  + \frac{\sqrt{2}}{\alpha} \right)   \norm{\gamma^1 - \gamma^2 }_{H^1}\,.
  \end{aligned}	
  \end{equation}
 We now estimate the second term in \eqref{gateaux:est_11}. First note that, as a consequence of \eqref{gateaux:estimates} and of the mean value theorem, the map $\gamma \mapsto F(\gamma)$ is bounded by $\norm{w}_{1,1}/\alpha$ and has (global) Lipschitz constant bounded by $\norm{w}_{1,1} (1+\sqrt{\beta/(2\alpha)})/\alpha$. Moreover, the map $\gamma \mapsto G(\gamma):=D_\gamma L/L(\gamma)$ is bounded by $\sqrt{\beta/(2\alpha)}$ (see \eqref{gateaux:est_L}). It is easy to check that $G$ is continuously Fr\'echet differentiable. Employing the estimates $L \geq \alpha$ and \eqref{gateaux:est_L}, we also check that $\gamma \mapsto D_\gamma G$ is bounded uniformly by $3\beta/(2\alpha)$. By the mean value theorem we then conclude that $G$ is globally Lipschitz with constant controlled by $3\beta/(2\alpha)$. Arguing as in \eqref{gateaux:est_13}, we compute
   \begin{equation} \label{gateaux:est_14}
  \norm{ F(\gamma^1) \, \frac{D_{\gamma^1}L}{L(\gamma^1)} - F(\gamma^2) \, \frac{D_{\gamma^2}L}{L(\gamma^2)}  }_{(H^1)^*} \leq \frac{\norm{w}_{1,1}}{\alpha} \left( \sqrt{\frac{\beta}{2\alpha}} + \frac{2\beta}{\alpha} \right) \norm{\gamma^1-\gamma^2}_{H^1} \,.
 \end{equation}
The inequality at \eqref{gateaux:lip} follows from \eqref{gateaux:est_11}, \eqref{gateaux:est_13},  \eqref{gateaux:est_14}, and the proof is concluded.  
  \end{proof}

Finally, we show that stationary points of $F$ with non-zero energy are curves contained in $E \subset \Om$.

   \begin{prop} \label{cor:optimality}
  Assume \ref{ass:F1}-\ref{ass:F4},  \ref{ass:H1}-\ref{ass:H3}. Let $(t \mapsto \tilde{\rho}_t) \in \curves$, $f \in \Ltwo$ and $\alpha,\beta>0$. For a.e.~$t \in (0,1)$ define $w_t:=-K_t(K_t^*\tilde{\rho}_t-f_t) \in C^{1,1}(\olom)$ and $w_t(x):=0$ for all $x \in \R^d \smallsetminus \olom$. %
  Consider the corresponding functional $F$ defined at \eqref{def:operators FWL}. %
  If $\gamma^* \in H^1$ is a stationary point for $F$, that is, $D_{\gamma^*} F=0$, then $\gamma^*$ satisfies the following system is in the weak sense 
   \begin{equation}
      \label{eq:critical_points}
          \beta  F(\gamma)  \, \ddot{\gamma}(t)  =\nabla w_t(\gamma(t))   \quad  \text{ for all } \,  t \in (0,1) \,, \qquad
           \dot \gamma(0)  = \dot \gamma(1) = 0   \,.
    \end{equation}
    If in addition $F(\gamma^*) \neq 0$, we have $\gamma^*([0,1]) \subset E$, where $E \Subset \Omega$ is the closed convex set in \ref{ass:F4}.  
    \end{prop}

\begin{proof}
By Lemma \ref{lem:C1} we have that $w \in L^2([0,1];C^{1,1}(\R^d))$. Moreover Proposition \ref{prop:gateaux} ensures that $F$ is continuously Fr\'echet differentiable over $H^1$. If $\gamma^*$ is such that $D_{\gamma^*}F=0$, from \eqref{gateaux F}-\eqref{gateaux derivatives} and the inequality $L\geq \alpha>0$ we deduce the weak formulation of \eqref{eq:critical_points}, i.e.,
\begin{equation} \label{eq:critical_weak}
-\int_0^1 \nabla w_t(\gamma^*(t)) \cdot \eta (t) \, dt = \beta F(\gamma^*) \int_0^1 \dot \gamma^*(t) \cdot \dot \eta(t) \, dt\,,  \quad \text{ for all } \,\, \eta \in H^1\,,
\end{equation}
Suppose that $F(\gamma^*) \neq 0$ and set 
 $A:=\{ t\in [0,1] \, \colon \, \gamma^*(t) \notin E\}$. %
 Assume by contradiction that $A \neq \emptyset$. Note that  $A \neq [0,1]$, since $F(\gamma^*) \neq 0$ and \ref{ass:F4} holds. 
 Since $E$ is closed and $\gamma^*$ is continuous, then $A$ is relatively open in $[0,1]$.  Therefore $A=\cup_{i \in \N} I_i$, with $I_i$ pairwise disjoint, which are either of the form $(a_i,b_i)$ with $0<a_i<b_i<1$, or $[0,a_i)$, or  $(a_i,1]$, with $0<a_i<1$, or empty. Assume that there exists $i \in \N$ such that $I_i=(a_i,b_i)$ with $0<a_i<b_i<1$. Let $\f \in C^1_c(a_i,b_i)$ and extend it to zero to the whole $[0,1]$. Set $\eta:=e_j \f$, with $e_j$ the $j$-th coordinate vector in $\R^d$. Since $\supp \nabla w_t \subset E$ for a.e.~$t \in (0,1)$ (see \ref{ass:F4}), $\gamma^*(t) \notin E$ for $t \in (a_i,b_i)$,  and $F(\gamma^*) \neq 0$, testing \eqref{eq:critical_weak} against $\eta$ yields $\int_{a_i}^{b_i} \dot\gamma^*_j \dot\f \, dt =0$, where $\gamma^*_j$ is the $j$-th component of $\gamma^*$. %
Therefore $\gamma^*$ %
 is linear in $[a_i,b_i]$. Since by construction $\gamma^*(a_i), \gamma^*(b_i) \in E$, by convexity of $E$ we obtain $\gamma^*(t) \in E$ for all $t \in [a_i,b_i]$, which is a contradiction. Assume now that there exists $i \in \N$ such that    $I_i=[0,a_i)$ for some $0<a_i<1$. Let $\f \in L^2(0,a_i)$, extend it to zero in $[a_i,1]$, and set $\eta(t):=-e_j \int_{a_i}^t \f(s)\, ds$. %
 Testing \eqref{eq:critical_weak} against $\eta$, allows to conclude that $\gamma^*$ is constant in $[0,a_i]$, which is a contradiction since by construction $\gamma^*(a_i) \in E$. Similarly, the remaining case $I_i=(a_i,1]$ for some $0<a_i<1$ leads to a contradiction. Thus we conclude that $A=\emptyset$, finishing the proof. 
\end{proof}

\subsubsection{Gradient descent} \label{sec:grad_descent}
 
 In this section we prove Theorem \ref{thm:gradient_descent} on descent sequences for the functional $F$ at \eqref{def:operators FWL}. The proof relies on the following lemma.

\begin{lem} \label{lem:descent}
Assume \ref{ass:H1}-\ref{ass:H3}, \ref{ass:F1}-\ref{ass:F4}. Let $(t \mapsto \tilde{\rho}_t) \in \curves$, $f \in \Ltwo$, $\alpha,\beta>0$. For a.e.~$t \in (0,1)$ define $w_t:=-K_t(K_t^*\tilde{\rho}_t-f_t) \in C^{1,1}(\olom)$ and $w_t(x):=0$ for all $x \in \R^d \smallsetminus \olom$. Consider the corresponding functional $F$ defined at \eqref{def:operators FWL}. Then, for all $M<0$, there exists $R>0$ depending only on $M,\Om,w, \alpha,\beta$, such that
\begin{equation} \label{eq:coercive:11}
 \{ \gamma \in H^1 \, \colon \, F(\gamma)\leq M\} \subset 
\{ \gamma \in H^1 \, \colon \, \norm{\gamma}_{H^1} \leq R\}	\,.	
\end{equation}
Moreover, let $\{\gamma^n\}$ in ${H^1}$ be a sequence such that
\begin{equation} \label{eq:descent_assumption}
	F(\gamma^n) \to c  \,, \qquad \norm{D_{\gamma^n}F}_{(H^1)^*} \to 0  \,, \qquad \text{ as } \,\, n \to +\infty \,,
\end{equation}
 for some $c<0$. Then, up to subsequences, $\gamma^n \to \gamma^*$ strongly in $H^1$. Any such accumulation point $\gamma^*$ satisfies $F(\gamma^*)=c$ and is stationary for $F$, namely, $D_{\gamma^*}F=0$.  
\end{lem}

\begin{proof}
Assume that $F(\gamma)\leq M$ for some $M<0$. %
Since $w \in L^2 ([0,1];C^{1,1}(\R^d))$ by Lemma \ref{lem:C1}, %
\begin{equation} \label{eq:descent_compact1}
\int_0^1 |\dot \gamma(t)|^2 \, dt \leq - \frac{2}{\beta} \left(\,\frac{\norm{w}_{1,1}}{M}+ \alpha \,	\right)\,,
\end{equation}
where we also used that $|W|\leq \norm{w}_{1,1}$ and $L>0$. 
As $\supp w_t \subset E$ for a.e.~$t \in (0,1)$ by \ref{ass:F4}, and $L(\gamma)>0$, the condition $F(\gamma)<0$, together with the continuity of $\gamma$, implies the existence of some $\hat t \in [0,1]$ such that $\gamma(\hat t) \in E$ (otherwise we would have $F(\gamma)=0$). Hence we can estimate
\begin{equation} \label{eq:descent_compact2}
\begin{aligned}
\sup_{t \in [0,1]} |\gamma(t)|	& \leq \sup_{t \in [0,1]}  |\gamma(t)-\gamma(\hat t)| + |\gamma(\hat t)| \\
& \leq 
\int_0^1 |\dot \gamma(s)| \, ds + \max_{x \in E} |x| 
\leq \left(\int_0^1 |\dot \gamma(s)|^2\, ds \right)^{1/2} + \max_{x \in \olom} |x|\,.
\end{aligned}
\end{equation}
From \eqref{eq:descent_compact1}-\eqref{eq:descent_compact2} we immediately deduce \eqref{eq:coercive:11} for some $R>0$. 
Assume now that $\{\gamma^n\}$ in $H^1$ satisfies \eqref{eq:descent_assumption} for some $c<0$. We will prove that $\{\gamma^n\}$ has at least one accumulation point with respect to the strong convergence of $H^1$. As $F(\gamma^n) \to c$ with $c<0$, from \eqref{eq:coercive:11} we deduce that $\{\gamma^n\}$ is uniformly bounded in $H^1$. %
Hence, there exists $\gamma \in H^1$ such that $\gamma^n \weak \gamma$ weakly in $H^1$ and $\gamma^n \to \gamma$ uniformly in $[0,1]$, up to subsequences (not relabelled). We will now prove that $\gamma^n \to \gamma$ strongly in $H^1$. %
By the uniform convergence $\gamma^n \to \gamma$ and regularity of $w$, dominated convergence yields %
\begin{equation} \label{eq:descent_proof:0}
W(\gamma^n) \to W(\gamma) \,\, \text{ as } \,\, n \to +\infty \,.
\end{equation}
Assume that $\dot\gamma \not \equiv 0$ and define, for $n$ sufficiently large, 
\[
\eta^n := \frac12 \, \gamma^n + \frac{\alpha}{\beta \int_0^1 \dot{\gamma}^n \cdot \dot \gamma \, dt} \, \gamma \,, \qquad  \eta:=\frac12 \, \gamma + \frac{\alpha}{\beta \int_0^1 |\dot \gamma|^2 \, dt} \, \gamma\,.
\]
Notice that $\eta^n \weak \eta$ weakly in $H^1$.
In particular $\{\eta^n\}$ is bounded in $H^1$, so that
\begin{equation} \label{eq:descent_proof:2}
|D_{\gamma^n}F(\eta^n)| \leq \norm{D_{\gamma^n} F}_{(H^1)^*} \norm{\eta^n}_{H^1} \to 0 \,\, \text{ as } \,\, n \to +\infty \,,
\end{equation}
where we employed continuous differentiability of $F$ (Proposition \ref{prop:gateaux}) and \eqref{eq:descent_assumption}. Notice now that $\eta^n \to \eta$ strongly in $L^2([0,1];\R^d)$, by Sobolev embeddings. Recalling \eqref{gateaux derivatives} and using the uniform convergence $\gamma^n \to \gamma$, together with the regularity of $w$, by dominated convergence we get
\begin{equation} \label{eq:descent_proof:3}
D_{\gamma^n}W(\eta^n) \to D_{\gamma}W(\eta) \,\, \text{ as } \,\, n \to +\infty \,.
\end{equation}
Moreover by definition of $\eta^n$ and \eqref{gateaux derivatives} one can check that $D_{\gamma^n} L(\eta^n) = L(\gamma^n)$
for all $n \in \N$. Taking the latter into account and substituting $\gamma^n$ and $\eta^n$ into \eqref{gateaux F} yields
\begin{equation} \label{eq:descent_proof:4}
L(\gamma^n) D_{\gamma^n}F(\eta^n) = D_{\gamma^n}W(\eta^n) - W(\gamma^n)
\end{equation}
for all $n \in \N$. Recalling that $\{\gamma^n\}$ is bounded in $H^1$, we also infer that $\{L(\gamma^n)\}$ is bounded. Therefore we can invoke \eqref{eq:descent_proof:0}, \eqref{eq:descent_proof:2}, \eqref{eq:descent_proof:3} to pass to the limit in \eqref{eq:descent_proof:4} and infer
\begin{equation} \label{eq:descent_proof:5}
D_{\gamma}W(\eta) = W(\gamma) \,.
\end{equation}
Substituting the definition of $\eta$  into \eqref{gateaux derivatives} yields $D_{\gamma}W(\eta) = L(\gamma) D_{\gamma}W(\gamma)/ D_\gamma L (\gamma)$.
By definition of $F$, the previous identity, and \eqref{eq:descent_proof:5}, we get that $F(\gamma)= D_\gamma W(\gamma) / D_\gamma L(\gamma)$.
On the other hand, substituting $\gamma^n$ and $\gamma$ into \eqref{gateaux F}, and recalling that $D_{\gamma^n}F(\gamma) \to 0$ by \eqref{eq:descent_assumption}, and that $L(\gamma^n) \geq \alpha >0$, results in
\begin{equation} \label{eq:descent_proof:6}
\left[ D_{\gamma^n} W(\gamma) - F(\gamma^n) D_{\gamma^n}L(\gamma)\right] \to 0  \,\, \text{ as } \,\, n \to +\infty \,.  
\end{equation}
Concerning \eqref{eq:descent_proof:6}, first note that
$F(\gamma^n) \to c$ by assumption. Moreover, since $\gamma^n \weak \gamma$ weakly in $H^1$ and $w \in L^2([0,1];C^1(\R^d))$, by dominated convergence we see that $D_{\gamma^n}W(\gamma) \to D_{\gamma}W(\gamma)$ and $D_{\gamma^n}L(\gamma) \to D_{\gamma}L(\gamma)$. Thus from \eqref{eq:descent_proof:6} we deduce that $D_{\gamma}W(\gamma)=c \, D_{\gamma}L(\gamma)$. Recalling that $F(\gamma)= D_\gamma W(\gamma) / D_\gamma L(\gamma)$, we conclude $F(\gamma)=c$, so that $F(\gamma^n) \to F(\gamma)$ (recalling \eqref{eq:descent_assumption}). 
By the convergence $F(\gamma^n) \to F(\gamma)$, definition of $F$ and \eqref{eq:descent_proof:0}, we conclude that $L(\gamma^n) \to L(\gamma)$. By definition of $L$, the latter is equivalent to
$\int_0^1 |\dot \gamma^n|^2 \, dt \to \int_0^1 |\dot \gamma|^2 \, dt$
as $n \to +\infty$. 
Since $\gamma^n \weak \gamma$ weakly in $H^1$, %
we infer $\gamma^n \to \gamma$ strongly in $H^1$. Setting $\gamma^*:=\gamma$ concludes the convergence statement.
Assume now that $\dot \gamma \equiv 0$. As $\dot \gamma \equiv 0$, by \eqref{gateaux derivatives} we obtain $D_\gamma L=0$. %
As $\gamma^n \weak \gamma$ weakly in $H^1$, by dominated convergence we get $D_{\gamma^n}W(\eta) \to D_{\gamma}W(\eta)$ and $D_{\gamma^n}L(\eta) \to D_{\gamma}L(\eta)=0$. Hence, taking the limit as $n \to +\infty$ in \eqref{gateaux F} evaluated on $\gamma^n$ and $\eta \in H^1$, and recalling that $\{F(\gamma^n)\}$ is bounded, yields $L(\gamma^n) D_{\gamma^n}F(\eta) \to D_{\gamma}W(\eta)$. 
As $\{L(\gamma^n)\}$ is bounded,  by \eqref{eq:descent_assumption} we get $D_{\gamma}W=0$.
We now claim that
\begin{equation} \label{eq:descent_proof:12}
\int_0^1 | \dot \gamma^n|^2 \, dt \to 0  \,\, \text{ as } \,\, n \to +\infty \,. 	
\end{equation}
Assume by contradiction that \eqref{eq:descent_proof:12} does not hold. Then there exists a subsequence (not relabelled) such that $\int_0^1 | \dot \gamma^n|^2 \, dt \geq q >0$ for all $n \in \N$. Given that $\{\gamma^n\}$ is bounded in $H^1$, without loss of generality we can assume that $\int_0^1 | \dot \gamma^n|^2 \, dt \to q_0$ as $n \to +\infty$, for some $q_0>0$.  Define
\[
\eta^n := \frac12 \, \gamma^n + \frac{\alpha}{\beta \int_0^1 |\dot{\gamma}^n|^2  \, dt} \, \gamma^n \,, \qquad
\eta := \frac12 \, \gamma + \frac{\alpha}{\beta q_0}\, \gamma\,.
\]
Clearly $\eta^n \weak \eta$ weakly in $H^1$. Arguing as in the proof of \eqref{eq:descent_proof:5}, we conclude that $D_\gamma W(\eta)=W(\gamma)$. Recalling that $D_\gamma W=0$, %
we infer $W(\gamma)=0$. %
Now notice that $L(\gamma^n) \to \beta q_0/2 + \alpha >0$, because  $\int_0^1 | \dot \gamma^n|^2 \, dt \to q_0$. By \eqref{eq:descent_proof:0} and the fact that $W(\gamma)=0$, we then conclude that $F(\gamma^n) \to 0$, which contradicts  \eqref{eq:descent_assumption}. Thus \eqref{eq:descent_proof:12} holds. As $\gamma^n \weak \gamma$ weakly in $H^1$ and $\dot \gamma \equiv 0$, from \eqref{eq:descent_proof:12} we infer that $\gamma^n \to \gamma$ strongly in $H^1$. Setting $\gamma^*:=\gamma$ concludes the convergence statement. 
Finally, suppose that $\gamma^n \to \gamma^*$ strongly in $H^1$ (subsequentially). As $F$ is continuously Fr\'echet differentiable (Proposition \ref{prop:gateaux}), thanks to \eqref{eq:descent_assumption} we obtain that $F(\gamma^*)=c$ and $D_{\gamma^*}F=0$.  
\end{proof}

\begin{proof}[Proof of Theorem \ref{thm:gradient_descent}]
The functional $F$ is continuously Fr\'echet differentiable as a consequence of Proposition \ref{prop:gateaux}. Moreover recall that $DF$ is locally Lipschitz (Proposition \ref{prop:gateaux}), with local Lipschitz constant in a ball $\{ \gamma \in H^1 \, \colon \,\norm{\gamma}_{H^1} \leq R\}$ estimated by $C_1R+C_2$, for some constants $C_1,C_2>0$ depending only on $w,\alpha,\beta$. Assume now that $\{\gamma^n\}$ in $H^1$ is a descent sequence in the sense of \eqref{def:sec5_descent} and set $M:=F(\gamma^0)<0$. By \eqref{eq:coercive:11} in Lemma \ref{lem:descent}, we can find some $R>0$, depending only on $M,\Om,w,\alpha,\beta$, such that
\begin{equation} \label{eq:condition_stepsize:1}
\{ \gamma \in H^1 \, \colon \, F(\gamma) \leq F(\gamma^0) \}
\subset 	\{ \gamma \in H^1 \, \colon \,\norm{\gamma}_{H^1} \leq R\}\,.
\end{equation}
For such $R$, consider the corresponding local Lipschitz constant $C_1R+C_2$ for $DF$. It well-known that the Armijo-Goldstein or Backtracking-Armijo rules for the stepsize $\{\delta_n\}$ guarantee that
\begin{equation} \label{eq:condition_stepsize}
0<A< \delta_n < B < \frac{2}{C_1R+C_2} \,,
\end{equation}
for some $A,B>0$ and all $n \in \N$. It is also standard %
that \eqref{eq:condition_stepsize:1}-\eqref{eq:condition_stepsize} and regularity of $F$ imply
$\norm{D_{\gamma^n} F}_{(H^1)^*} \to 0$ and $F(\gamma^{n+1}) \leq F(\gamma^n)$ for all $n \in \N$.
Since $F(\gamma^0)<0$ and $|F| \leq \norm{w}_{1,1}/\alpha$ by \eqref{gateaux:estimates}, from the monotonicity of $\{F(\gamma^n)\}$ we infer that $F(\gamma^n) \to c$ for some $c<0$. Therefore $\{\gamma^n\}$ satisfies \eqref{eq:descent_assumption}, so that we can apply Lemma \ref{lem:descent} and infer that $\{\gamma^n\}$ is strongly precompact in $H^1$, and that any strong accumulation point $\gamma^*$ satisfies $F(\gamma^*)=c$ and $D_{\gamma^*}F=0$. Since $c<0$,  by Proposition \ref{cor:optimality} we also obtain that $\gamma^*([0,1]) \subset E$, concluding. 
\end{proof}

 \subsubsection{Test for zero minimum} \label{subsec:test}

  \begin{prop} \label{prop:test}
  Assume \ref{ass:H1}-\ref{ass:H3}, \ref{ass:F1}-\ref{ass:F4}. Let $(t \mapsto \tilde{\rho}_t) \in \curves$, $f \in \Ltwo$, $\alpha,\beta>0$. For a.e.~$t \in (0,1)$ define $w_t:=-K_t(K_t^*\tilde{\rho}_t-f_t) \in C^{1,1}(\olom)$ and $w_t(x):=0$ for all $x \in \R^d \smallsetminus \olom$. Consider the corresponding functional $F$ defined at \eqref{def:operators FWL}. Then $0$ is the minimum of  \eqref{eq:sec5_ins} 
  if and only if 
   \begin{equation} \label{test2}
  \int_0^1 \max_{x \in \olom} w_t(x) \, dt \leq 0	\,.
  \end{equation}
  \end{prop}

  \begin{proof}
  First note that $w$ is a Carath\'eodory map in $[0,1]\times \olom$ by the proof of Lemma \ref{lem:C1}, since $g_t:=-K_t^* \tilde{\rho}_t +f_t$ belongs to $L^2_H$ by Lemma \ref{lem:prop K}. Therefore $w$ is also Carath\'eodory in $[0,1]\times E$, because $\supp w_t \subset E$ with $E \Subset \Om$ closed and convex (see \ref{ass:F4}). Seeing that $E$ is compact, we can apply Theorem 18.19 in \cite{aliprantis} to obtain that the scalar map $t \mapsto \max_{x \in E} w_t(x)$ is measurable, and that there exists a measurable curve $\hat{\gamma} \colon [0,1] \to E$ such that 
  $\hat{\gamma}(t) \in \argmax_{x \in E} w_t(x)$  for all $t \in [0,1]$. By the condition $\supp w_t \subset E$, we infer $\max_{x \in E} w_t(x) = \max_{x \in \olom} w_t(x)$ for a.e.~$t \in (0,1)$, showing that $t \mapsto \max_{x \in \olom} w_t(x)$ is measurable. Thus the integral in \eqref{test2} is well defined. Moreover, by construction, 
  $w_t(\hat\gamma(t))= \max_{x \in \olom} w_t(x)$ for a.e.~$t \in (0,1)$.
Assume that $0$ is the minimum of  \eqref{eq:sec5_ins}. By the inequality $L(\gamma)\geq \alpha >0$, we infer
  \begin{equation} \label{test3}
  \int_0^1 w_t(\gamma(t)) \, dt \leq 0 
  \end{equation}
  for all $\gamma \in H^1([0,1];\olom)$. As $E \Subset \Om$, we can find a sequence $\{\gamma_n\}$ in $H^1([0,1];\olom)$ such that $\gamma_n \to \hat\gamma$ a.e.~in $(0,1)$ as $n \to + \infty$. Since $w_t \in C^{1,1}(\olom)$ for a.e.~$t$ fixed, we have $w_t(\gamma_n(t)) \to w_t(\hat\gamma(t))$  a.e.~in $(0,1)$. Moreover $|w_t(\gamma_n(t))| \leq \norm{w_t}_{C^{1,1}(\olom)}$. We can then substitute $\gamma^n$ in \eqref{test3} and apply dominated convergence to infer that $\hat \gamma$ satisfies \eqref{test3} as well. By maximality of $\hat{\gamma}$ we conclude \eqref{test2}. Conversely, assume that \eqref{test2} holds. For all $\gamma \in H^1([0,1];\olom)$ we get
  \[
  F(\gamma) = -\,\frac{\int_0^1 w_t(\gamma(t))\, dt}{L(\gamma)}  \geq 
  -\,\frac{\int_0^1 \max_{x \in \olom}w_t(x)\, dt}{L(\gamma)} \,.
  \]
  Since $L(\gamma) > 0$, we infer that $0$ is the minimum of  \eqref{eq:sec5_ins}.
 \end{proof}

  \subsection{Analysis for the sliding step} \label{sec:gradient_flow}

In this section we rigorously justify the sliding step discussed in Section \ref{subsec:sliding}, showing that, under Assumption \ref{def:additional K}, the target functional at \eqref{def:sliding} is differentiable. To fix notations,
  Let $N \in \N, N\geq 1$ and $c_j \in \R, c_j > 0$ be fixed. We denote by $(H^1_\Om)^N$ the space of points $\Gamma:=(\gamma_1,\ldots,\gamma_N)$ with $\gamma_j 
  \in H^1_\Om:=H^1([0,1];\olom)$. For $\Gamma \in (H^1_\Om)^N$ we define the measure
  \[
  \mu(\Gamma):= \sum_{j=1}^N c_j \mu_{\gamma_j} \in \M\,,
  \]
  where $\mu_{\gamma_j}:=(\rho_{\gamma_j},m_{\gamma_j}) \in \points$, according to \eqref{ext_meas}.
  Define the functional $\Phi \colon (H^1_\Om)^N \to \R$ by
  \begin{equation} \label{eq:grad_flow}
  \Phi (\Gamma) = \Phi (\gamma_1,\ldots,\gamma_N) := T_{\alpha,\beta,\boldsymbol{c}} (\mu(\Gamma))\,,
  \end{equation}
  where $T_{\alpha,\beta,\boldsymbol{c}}$ is defined in \eqref{def:sliding}, for some $f \in \Ltwo$ and $\alpha,\beta>0$ fixed. We also recall the notation $H^1:=H^1([0,1];\R^d)$.

  \begin{prop} \label{prop:grad_flow}
    Assume \ref{ass:F1}-\ref{ass:F3}, \ref{ass:H1}-\ref{ass:H3}. The functional $\Phi$ at \eqref{eq:grad_flow} is continuously Fr{\'e}chet differentiable at each $\Gamma \in (H^1_\Om)^N$ such that $\gamma_j([0,1]) \subset \Om$ for $j=1,\ldots,N$, with
    \begin{equation}
    \label{prop:grad_flow:gat}
    D_\Gamma \Phi(\Theta)= \sum_{j=1}^N c_j \, D_{\gamma_j} F (\eta_j)\,,
    \end{equation}
    for all $\Theta=(\eta_1 ,\ldots,\eta_N)$, $\eta_j \in H^1$. Here $F$ is as in 
    \eqref{def:operators FWL}, with respect to the dual variable
    \begin{equation} \label{prop:grad_flow:w}
  w_t := - K_t \left( \sum_{j=1}^N c_j a_{\gamma_j} K_t^*  \delta_{\gamma_j(t)}  - f_t  \right)\,.
  \end{equation}
  \end{prop}

  \begin{proof}
  Let $\Gamma \in (H^1_\Om)^N$	with $\gamma_j ([0,1]) \subset  \Om$. Let $\Theta=(\eta_1,\ldots,\eta_N)$ with $\eta_j \in H^1$ and $\e_0>0$ sufficiently small, so that $(\gamma_j + \e \eta_j) ([0,1]) \subset \Om$ for each $0<\e<\e_0$, $j=1,\ldots,N$. %
  By Lemma \ref{lem:additivity} 
  \begin{equation} \label{prop:grad_flow:1}
  J_{\alpha,\beta} (\mu(\Gamma)) = 
  J_{\alpha,\beta} (\mu(\Gamma+\e \Theta)) = \sum_{j=1}^N c_j\,,
  \end{equation}
  for all $0<\e<\e_0$.
  Define $w$ as in \eqref{prop:grad_flow:w} 
  and notice that $w \in L^2([0,1];C^{1,1}(\olom))$ by (the proof of) Lemma \ref{lem:C1}, since $f \in \Ltwo$ by assumption and $t \mapsto K_t^* \delta_{\gamma_j(t)}$ belongs to $\Ltwo$ by Lemma \ref{lem:prop K}, as $(t \mapsto \delta_{\gamma_j(t)}) \in \pcurves$.  By \eqref{prop:grad_flow:1}, linearity of $K_t^*$ and the identity \eqref{eq:bilinear2} with $\rho$ and $\hat{\rho}$ replaced by $\sum_{j=1}^N c_j \rho_{\gamma_j}$ and $\sum_{j=1}^N c_j \rho_{\gamma_j + \e \eta_j}$ respectively, one can compute that 
  \begin{align} 
  \frac{\Phi(\Gamma + \e \Theta)   - \Phi (\Gamma)}{\e}    %
   = -\frac{1}{\e} \sum_{j=1}^N c_j \, 
  \ps{  \rho_{\gamma_j + \e \eta_j}  - \rho_{\gamma_j}    }{w} + \frac{1}{2\e}
  \norm{    \sum_{j=1}^N c_j  K^*\left( \rho_{\gamma_j + \e \eta_j}  -  \rho_{\gamma_j }  \right)    }^2_{L^2_H} \,,\label{prop:grad_flow:2}
  \end{align}
  for all $0<\e<\e_0$. %
 By proceeding in the same way as in \eqref{eq:diff_computation:11}, we have
  \[
  \lim_{\e \to 0} -\frac{1}{\e} \sum_{j=1}^N c_j \, 
  \ps{  \rho_{\gamma_j + \e \eta_j}  - \rho_{\gamma_j}    }{w} =   \lim_{\e \to 0 } \, \sum_{j=1}^N c_j \, \frac{F( \gamma_j + \e \eta_j  ) - F( \gamma_j   )}{\e}
  =   \sum_{j=1}^N c_j \, D_{\gamma_j} F (\eta_j) \,,
  \]
  where we also used the definition of $F$ at \eqref{def:operators FWL} and Proposition \ref{prop:gateaux}.
  We claim that the second term in \eqref{prop:grad_flow:2} is infinitesimal as $\e \to 0$. By \ref{ass:F2} and Cauchy-Schwarz's inequality one has
\[
  \norm{   \sum_{j=1}^N c_j  K^*\left( \rho_{\gamma_j + \e \eta_j}  -  \rho_{\gamma_j }  \right)    }^2_{L^2_H}   
  \leq N C^2 \sum_{j=1}^N c_j^2 
  \int_0^1 \norm{a_{\gamma_j + \e \eta_j} \delta_{\gamma_j(t) + \e \eta_j(t)} -  a_{\gamma_j } \delta_{\gamma_j(t)}   }^2_{C^1(\olom)^*} \,dt \,
\]
where $C>0$ is the constant in \ref{ass:F2}.  %
Fix $t \in [0,1]$. Since $a_\gamma=1/L(\gamma)$ (see \eqref{def:operators FWL}, \eqref{ext_meas}),
  \begin{equation} \label{prop:grad_flow:5}
  \norm{a_{\gamma_j + \e \eta_j} \delta_{\gamma_j(t) + \e \eta_j(t)} -  a_{\gamma_j } \delta_{\gamma_j(t)}   }_{C^1(\olom)^*} =
  \sup_{\norm{\f}_{C^1(\olom)} \leq 1} \left| \frac{\f (\gamma_j(t) + \e \eta_j(t))}{L(\gamma_j + \e \eta_j)} -  \frac{\f (\gamma_j(t))}{L(\gamma_j )}   \right| \,.
  \end{equation}
  For a fixed $\f \in C^1(\olom)$, define the map $\Psi_t \colon H^1_\Om \to \R$ as 
  $\Psi_t(\gamma):= \f (\gamma(t))/L(\gamma)$.
  Since $\olom$ is bounded and $\f \in C^1(\olom)$, 
  one can check that $\gamma \in H^1_\Om \mapsto \f(\gamma(t)) \in\R$ is continuously Fr{\'e}chet differentiable at each $\gamma \in H^1_\Om$ with $\gamma([0,1])\subset \Om$, with derivative given by $\eta \mapsto \nabla \f (\gamma (t)) \cdot \eta(t)$. Moreover, $L$ is continuously differentiable by Proposition \ref{prop:gateaux}. Therefore $\Psi_t$ is continuously differentiable, given that $L \geq \alpha >0$. By differentiation rules and triangle inequality we also obtain the estimate
  \[
  \begin{aligned}
  \norm{D_\gamma \Psi_t}_{(H^1)^*} & = \sup_{\norm{\eta}_{H^1} \leq 1} \left|  \frac{\nabla \f (\gamma(t))\cdot \eta(t)}{L(\gamma)} - 
  \frac{D_\gamma L (\eta)}{L(\gamma)} \, \frac{\f (\gamma(t))}{L(\gamma)} \right|\\
    & \leq \frac{\tilde{C} \norm{\f}_{C^1(\olom)}}{ L(\gamma)} +    \norm{\frac{ D_\gamma L}{L(\gamma)} }_{(H^1)^*} \, \frac{\norm{\f}_{C^1(\olom)}}{L(\gamma)} \leq \frac{\norm{\f}_{C^1}}{\alpha} \left( \tilde{C} + \sqrt{\frac{\beta}{2 \alpha}} \right)
  \end{aligned}
  \]
  where $\tilde{C}>0$ is the Sobolev embedding constant for $H^1((0,1);\R^d) \hookrightarrow C([0,1];\R^d)$, and where in the last inequality we used that $L \geq \alpha$ and \eqref{gateaux:est_L}. 
  By the mean value theorem and \eqref{prop:grad_flow:5}
  \[
  \norm{a_{\gamma_j + \e \eta_j} \delta_{\gamma_j(t) + \e \eta_j(t)} -  a_{\gamma_j } \delta_{\gamma_j(t)}   }_{C^1(\olom)^*}  \leq  \e \, C \norm{\eta_j}_{H^1} \,,
  \]
  where $C$ depends only on $\alpha,\beta$ and on $\olom$.   %
  Putting together the above estimates shows that the second term in \eqref{prop:grad_flow:2} is infinitesimal as $\e \to 0$. This proves that the G\^{a}teaux derivative of $\Phi$ at $\Gamma$ in the direction $\Theta$ is given by \eqref{prop:grad_flow:gat}. From \eqref{prop:grad_flow:gat} and Proposition \ref{prop:gateaux} we also conclude that $\Gamma \mapsto D_{\Gamma} \Phi$ is continuous from $(H^1)^N$ into its dual, completing the proof.   
  \end{proof}

\subsection{Dynamic undersampled Fourier measurements} \label{subsec:Fourier_example}
In this section we detail a specific example of operators 
$K^*_t$ and measurement spaces $H_t$ satisfying the assumptions \ref{ass:H1}-\ref{ass:H3}, \ref{ass:K1}-\ref{ass:K3} in Section \ref{sec:assumptions}. Such example is contained in  \cite[Section 5]{bf}, and realizes, within our framework, a spatially undersampled Fourier transform with time-dependent mask. Let $\Om \subset \R^2$ be a bounded open domain, and $\sigma_t \in \M^+(\R^2)$ be a family of measures such that 
  	\begin{enumerate}[label=\textnormal{(M\arabic*)}]
    \item $\norm{\sigma_t}_{\M(\R^2)} \leq C$ for a.e.~$t \in (0,1)$,  \label{ass:M1}
    \item the map $t \mapsto \int_{\R^2} \f(s) \, d\sigma_t(s)$ is measurable for all $\f \in C_0(\R^2)$. \label{ass:M2}
    \end{enumerate}
The measurement spaces are defined as the real Hilbert space  $H_t:=L^2_{\sigma_t}(\R^2;\C)$, with scalar product given by 
$
\ps{f}{g}_{H_t}:={\rm Re} \left(\int_{\R^2} f(s)\overline{g(s)} \, d\sigma_t(s) \right)$, where ${\rm Re}$ denotes the real part in $\C$. For a measure $\rho \in \M(\olom)$ we denote its Fourier transform by
\begin{equation} \label{eq:fourier:1}
\mathfrak{F}\rho(s):=%
\int_{\R^2} \exp{(-2\pi i  x \cdot s)} \, d\rho(x) \,,
\end{equation}
for all $s \in \R^2$, where $\rho$ is extended to zero outside of $\olom$. Note that $\mathfrak{F}\rho \in C^\infty(\R^2;\C)$. We then define $K_t^* \colon \M(\olom) \to H_t$ by setting $K_t^*\rho:=\mathfrak{F}\rho$. In this way $K_t^*$ corresponds to  the Fourier transform sampled according to the measure $\sigma_t$. As a consequence of \cite[Lemma 5.4]{bf} we have that \ref{ass:H1}-\ref{ass:H3}, \ref{ass:K1}-\ref{ass:K3} hold whenever \ref{ass:M1}-\ref{ass:M2} are satisfied. It an easy check that in this case also \ref{ass:F1}-\ref{ass:F3} from Section \ref{sec:insstepheu} are satisfied. Moreover, define the operators $\tilde{K}_t^* \colon C^{1,1}(\olom)^* \to H_t$ as the dense extension of $\tilde{K}_t^* \rho:=\mathfrak{F}_E\rho$, where for $\rho \in \M(\olom)$ we set  
\begin{equation} \label{eq:fourier:2}
\mathfrak{F}_E \rho(s):=\int_{\R^2} \exp{(-2\pi i  x \cdot s)} \ \xi_E(x) \, d\rho(x) \,,
\end{equation}
for all $s \in \R^2$, 
and $\xi_E \colon \olom \to [0,1]$ is a cut-off with respect to a closed convex set $E \Subset \Om$ satisfying \eqref{def:cut-off}. Then, arguing as in Remark \ref{rem:assumptions}, we can show that $\tilde{K}_t^*$ satisfies \ref{ass:F1}-\ref{ass:F4}.

\subsubsection{Discrete sampling} \label{subsec:Fourier_discrete_sample}

As a particular case of the above setting, we sample the Fourier transform on a finite collection of time-dependent frequencies.
Specifically, fix $T \in \N$ and consider a time-grid $0 \leq t_0 < t_1 < \ldots <t_{T} \leq 1$. For each time $t_i$ we sample a given collection of frequencies $S_{i,1}, \ldots, S_{i,n_i} \in \R^2$, for some $n_i \in \N$. In order to incorporate this in the above setting, define a partition $A_0,\ldots,A_T$ of $[0,1]$, in a way that $t_i \in A_i$ for all $i=0,\ldots,T$. Then, define the scalar measure 
\[
\sigma_t :=  \sum_{i=0}^{T}  \sum_{k=1}^{n_i}\frac{1}{n_i}\, \delta_{S_{i,k}} \,\rchi_{A_i}(t)\,.
\] 
It is immediate to check that $\sigma_t$ satisfies \ref{ass:M1}-\ref{ass:M2}, so that the prescribed sampling falls within the above framework. 
 In this case the sampling space $H_t:=L^2_{\sigma_t}(\R^2;\C)$ is isomorphic to $\C^{n_i}$, whenever $t \in A_i$.  By defining the Fourier-type kernels $\psi_t \colon \R^2 \to \C^{n_i}$  
\[
\psi_t(x):= \left(\exp{(-2\pi i x \cdot S_{i,k}) }\right)_{k=1}^{n_i} \in \C^{n_i}\,,
\]
for any $t \in A_i$, we see that the operator  $K_t^* \colon \M(\olom) \to H_t$ defined by $K_t^*\rho:=\mathfrak{F}\rho$ in \eqref{eq:fourier:1} and its pre-adjoint $K_t \colon  H_t \to C(\olom)$ can be represented as
\begin{equation} \label{eq:app_Fourier2}
    K_{t}^*(\rho) = \int_{\R^2} \psi_t(x) \, d\rho(x)\,, \qquad
   K_{t}(h) = \left(\  x \mapsto  \left< \psi_{t}(x), h \right>_{H_{t}} \right)\,,
\end{equation}
for all $\rho \in \M(\olom)$, $h \in H_t$, where the first integral is computed component-wise. Similarly, the operator $K_t^*\rho:=\mathfrak{F}_E \rho$ in \eqref{eq:fourier:2} and its pre-adjoint $K_t$ are represented by \eqref{eq:app_Fourier2} with $\psi_t$ replaced by the cut-off kernel 
\begin{equation} \label{eq:app_Fourier3}
\psi_t(x):= \left(\exp{(-2\pi i x \cdot S_{i,k}) \ \xi_E(x) }\right)_{k=1}^{n_i} \in \C^{n_i}\,,
\end{equation}
for every $t \in A_i$.

  \end{document}